\documentclass[reqno]{amsart}
\usepackage{hyperref}

\AtBeginDocument{{\noindent\small
\emph{}}
\vspace{8mm}}
\begin{document}
\title[\hfil fractional Schr\"{o}dinger-Poisson system with critical growth]
{Existence and concentration of positive ground state solutions for nonlinear fractional Schr\"{o}dinger-Poisson system with critical growth}

\author[K. M. Teng ]
{Kaimin Teng}  
\address{Kaimin Teng (Corresponding Author)\newline
Department of Mathematics, Taiyuan
University of Technology, Taiyuan, Shanxi 030024, P. R. China}
\email{tengkaimin2013@163.com}

\author[R. P. Agarwal]
{Ravi P. Agarwal}  
\address{Ravi P. Agarwal\newline
Department of Mathematics, Texas A$\&$M University-Kingsville,
Texas, Kingsville 78363, USA}
\email{ravi.agarwal@tamuk.edu}


\begin{abstract}
In this paper, we study the following fractional Schr\"{o}dinger-Poisson system involving competing potential functions
\begin{equation*}
\left\{
  \begin{array}{ll}
    \varepsilon^{2s}(-\Delta)^su+V(x)u+\phi u=K(x)f(u)+Q(x)|u|^{2_s^{\ast}-2}u, & \hbox{in $\mathbb{R}^3$,} \\
    \varepsilon^{2t}(-\Delta)^t\phi=u^2,& \hbox{in $\mathbb{R}^3$,}
  \end{array}
\right.
\end{equation*}
where $\varepsilon>0$ is a small parameter, $f$ is a function of $C^1$ class, superlinear and subcritical nonlinearity, $2_s^{\ast}=\frac{6}{3-2s}$, $s>\frac{3}{4}$, $t\in(0,1)$, $V(x)$ $K(x)$ and $Q(x)$ are positive continuous function. Under some suitable assumptions on $V$, $K$ and $Q$, we prove that there is a family of positive ground state solutions with polynomial growth for sufficiently small $\varepsilon>0$, of which it is concentrating on the set of minimal points of $V(x)$ and the sets of maximal points of $K(x)$ and $Q(x)$. The methods are based on the Nehari manifold, arguments of Brezis-Nirenberg and concentration compactness of P. L. Lions.
\end{abstract}

\maketitle
\numberwithin{equation}{section}
\newtheorem{theorem}{Theorem}[section]
\newtheorem{lemma}[theorem]{Lemma}
\newtheorem{definition}[theorem]{Definition}
\newtheorem{remark}[theorem]{Remark}
\newtheorem{proposition}[theorem]{Proposition}
\newtheorem{corollary}[theorem]{Corollary}
\allowdisplaybreaks

\section{Introduction}
In this paper, we are concerned with the existence of ground state solution and its concentration phenomenon for the following fractional Schr\"{o}dinger-Poisson system
\begin{equation}\label{main}
\left\{
  \begin{array}{ll}
    \varepsilon^{2s}(-\Delta)^su+V(x)u+\phi u=K(x)f(u)+Q(x)|u|^{2_s^{\ast}-2}u, & \hbox{in $\mathbb{R}^3$,} \\
    \varepsilon^{2t}(-\Delta)^t\phi=u^2,& \hbox{in $\mathbb{R}^3$,}
  \end{array}
\right.
\end{equation}
where $s>\frac{3}{4},t\in(0,1)$, $2_s^{\ast}=\frac{6}{3-2s}$, $\varepsilon>0$ is a small parameter. We assume that the potentials $V(x)$, $K(x)$ and $Q(x)$ satisfy the following hypotheses:\\
$(V)$ $V(x)\in C(\mathbb{R}^3,\mathbb{R})$ and $0<V_0=\inf\limits_{\mathbb{R}^3}V(x)<\liminf\limits_{|x|\rightarrow+\infty}V(x)=V_{\infty}<+\infty$.\\
$(K)$ $K(x)\in C(\mathbb{R}^3,\mathbb{R})$ and $\lim\limits_{|x|\rightarrow+\infty}K(x)=K_{\infty}\in(0,+\infty)$ and $K(x)\geq K_{\infty}$ for $x\in\mathbb{R}^3$.\\
$(Q_0)$ $Q(x)\in C(\mathbb{R}^3,\mathbb{R})$ and $\lim\limits_{|x|\rightarrow+\infty}Q(x)=Q_{\infty}\in(0,+\infty)$ and $Q(x)\geq Q_{\infty}$.\\
$(Q_1)$ There exists $\delta>0$ and $\rho_0>0$ such that $|Q(x)-Q(x_0)|\leq \delta|x-x_0|^{\alpha}$ for $|x-x_0|\leq\rho_0$, where $\frac{3-2s}{2}\leq\alpha< \frac{3}{2}$ and $Q(x_0)=\max_{x\in\mathbb{R}^3}Q(x)$.\\
$(\mathcal{H})$ $\Theta:=\Theta_V\cap\Theta_K\cap\Theta_Q\neq\emptyset$, where $\Theta_V=\{x\in\mathbb{R}^3:\,\, V(x)=V_0=\inf_{x\in\mathbb{R}^3}V(x)\}$ and
\begin{align*}
&\Theta_K=\{x\in\mathbb{R}^3:\,\, K(x)=K_0=\max\limits_{x\in\mathbb{R}^3}K(x)\}\quad\Theta_Q=\{x\in\mathbb{R}^3:\,\, Q(x)=Q_0=\max\limits_{x\in\mathbb{R}^3}Q(x)\}.
\end{align*}
Without loss of generality, we may assume that $0\in\Theta$. Here the non-local operator $(-\Delta)^s$ ($s\in(0,1)$), which is called fractional Laplace operator, can be defined by
\begin{equation*}
(-\Delta)^sv(x)=C_s\,P.V.\int_{\mathbb{R}^3}\frac{v(x)-v(y)}{|x-y|^{3+2s}}\,{\rm d}y=C_s\lim_{\varepsilon\rightarrow0}\int_{\mathbb{R}^3\backslash B_{\varepsilon}(x)}\frac{v(x)-v(y)}{|x-y|^{3+2s}}\,{\rm d}y
\end{equation*}
for $v\in\mathcal{S}(\mathbb{R}^3)$, where $\mathcal{S}(\mathbb{R}^3)$ is the Schwartz space of rapidly decaying $C^{\infty}$ function, $B_{\varepsilon}(x)$ denote an open ball of radius $r$ centered at $x$ and the normalization constant $C_s=\Big(\int_{\mathbb{R}^3}\frac{1-\cos(\zeta_1)}{|\zeta|^{3+2s}}\,{\rm d}\zeta\Big)^{-1}$. For $u\in\mathcal{S}(\mathbb{R}^3)$, the fractional Laplace operator $(-\Delta)^s$ can be expressed as an inverse Fourier transform
\begin{equation*}
(-\Delta)^su=\mathcal{F}^{-1}\Big((2\pi|\xi|)^{2s}\mathcal{F}u(\xi)\Big),
\end{equation*}
where $\mathcal{F}$ and $\mathcal{F}^{-1}$ denote the Fourier transform and inverse transform, respectively.

Formally, system \eqref{main} consists of a fractional Schr\"{o}dinger equation coupled with a fractional Poisson equation. It can be regarded as the associated fractional case of the following classical Schr\"{o}dinger-Poisson system
\begin{equation}\label{main-1}
\left\{
  \begin{array}{ll}
    -\varepsilon^2\Delta u+V(x)u+\phi u=g(x,u), & \hbox{in $\mathbb{R}^3$,} \\
    -\varepsilon^2\Delta \phi=u^2,& \hbox{in $\mathbb{R}^3$.}
  \end{array}
\right.
\end{equation}
It is well known that system \eqref{main-1} has a strong physical meaning because it appears in quantum mechanics models (see e.g. \cite{CP,LS}) and in semiconductor
theory \cite{MRS}. In particular, systems like (1.2) have been introduced in \cite{BF} as a model to describe solitary waves. In \eqref{main-1}, the first equation is a nonlinear stationary equation (where the nonlinear term simulates the interaction between many particles) that is coupled with a Poisson equation, to be satisfied by $\phi$, meaning
that the potential is determined by the charge of the wave function. For this reason, \eqref{main-1} is referred to as a nonlinear Schr\"{o}dinger-Poisson system. In recent years, there has been increasing attention to systems like \eqref{main-1} on the existence of positive solutions, ground state solutions, multiple solutions and semiclassical states; see for examples \cite{AM,AR,BF,HZ,Ruiz,WTXZ,ZZ} and the references therein.

The other motivation to study the system \eqref{main} lies in the important roles that fractional equations
involving fractional operators play in the problems of Physics, Chemistry and Geometry, and so on. Indeed, fractional operators appear in many problems, such as: fractional quantum mechanics \cite{La,La1}, anomalous diffusion \cite{MK}, financial \cite{CT1}, obstacle problems \cite{S}, conformal geometry and minimal surfaces \cite{CM}. With the help of the harmonic extension technique developed by Caffarelli and Silvestre \cite{CS}, the non-local problem can be reduced to a local one, choosing a weighted Sobolev space as the work space, the usual variational methods have been successfully applied to nonlinear problems involving fractional Laplacian, we refer to interesting readers to see the related works \cite{AMi,BCPS,BV,CT,DMV,DPW,FL,Teng1} and so on. Another power technique is that one directly take the usual fractional Sobolev space as the work space so that the variational approaches can be applied, the related work can be referred to see \cite{BV,CW,DMV,DPDV,DPV,Sec,SZ,Teng} and so on.

To the best of our knowledge, there are only few recent papers dealing with a similar system like \eqref{main}. For example, in \cite{Teng2}, we established the existence of positive ground state solution for a similar system involving a critical Sobolev exponent
\begin{equation}\label{main1}
\left\{
  \begin{array}{ll}
    (-\Delta)^su+V(x)u+\phi u=|u|^{p-1}u+|u|^{2_s^{\ast}-2}u, & \hbox{in $\mathbb{R}^3$,} \\
    (-\Delta)^t\phi=u^2,& \hbox{in $\mathbb{R}^3$,}
  \end{array}
\right.
\end{equation}
by using the Nehari-Pohozaev manifold combing monotone trick with global compactness Lemma. Using the similar methods, in \cite{Teng3}, positive ground state solutions for subcritical problem, i.e., $|u|^{p-1}u+|u|^{2_s^{\ast}-2}u$ is replaced by $|u|^{p-1}u$ with $p\in(2,2_s^{\ast}-1)$, were established when $s=t$. In \cite{Wei}, the existence of infinitely many (but possibly sign changing) solutions by means of the Fountain Theorem under suitable assumptions on nonlinearity term. In \cite{ZDS}, the authors studied the existence of radial solutions for system \eqref{main1} with replacing $|u|^{p-1}u+|u|^{2_s^{\ast}-2}u$ by $f(u)$, where the nonlinearity $f(u)$ verifies the subcritical or critical assumptions of Berestycki-Lions type. In \cite{MS}, the authors studied the semiclassical state of the following system
\begin{equation*}
\left\{
  \begin{array}{ll}
    \varepsilon^{2s}(-\Delta)^su+V(x)u+\phi u=f(u), & \hbox{in $\mathbb{R}^N$,} \\
    \varepsilon^{\theta}(-\Delta)^{\frac{\alpha}{2}}\phi=\gamma_{\alpha}u^2,& \hbox{in $\mathbb{R}^N$,}
  \end{array}
\right.
\end{equation*}
where $s\in(0,1)$, $\alpha\in(0,N)$, $\theta\in(0,\alpha)$, $N\in(2s,2s+\alpha)$, $\gamma_{\alpha}$ is a positive constant, $f(u)$ satisfies the following subcritical growth assumptions: $0<KF(t)\leq f(t)t$ with some $K>4$ for all $t\geq0$ and $\frac{f(t)}{t^3}$ is strictly increasing on $(0,+\infty)$. By adapting some ideas of Benci, Cerami and Passaseo \cite{BC,BCP} and using the Ljusternick-Schnirelmann Theory, the authors obtained the multiplicity of positive solutions which concentrate on the minima of $V(x)$ as $\varepsilon\rightarrow0$. Of course, recently, this methods have been successfully applied to other many problems, such as: Schr\"{o}dinger-Poisson system \cite{HZ}, fractional Schr\"{o}dinger equations \cite{HZ1}, $p$-Laplacian problem \cite{AF}, Kirchhoff type problems \cite{HZ2}, and the references therein. In \cite{LZ}, by using the methods mentioned before, Liu and Zhang proved the existence and concentration of positive ground state solution for problem \eqref{main} when $K(x)\equiv1$ and $Q(x)\equiv1$, but they have not discuss the decay of solutions.

In the last decades, the existence, multiplicity and concentration behavior of solutions of nonlinear problems involving competing potential functions of the form
\begin{equation}\label{equ1-1}
-h^2\Delta u+V(x)u=K(x)|u|^{p-1}u+Q(x)|u|^{q-1}u \quad x\in\mathbb{R}^N
\end{equation}
where $h>0$, $1<q<p<\frac{N+2}{N-2}$, have been investigated by several scholars. Such as, Rabinowtiz \cite{Rabinowtiz} proved that if $V$ is coercive and
$K$, $Q$ satisfy suitable assumptions, a result implies the existence of ground state solutions for problem \eqref{equ1-1}, for any $h>0$. Wang and Zeng \cite{WZ} assumed that $V(x)$ has a positive lower bound, $K(x)$ is bounded and positive, $Q(x)$ is bounded (and allowed to change sign), they proved the existence of a ground state solution of \eqref{equ1-1}, for any $h>0$ small. Furthermore, they studied the concentration behaviour of such solutions and give a necessary condition for the location of the concentration points for positive bound states. Under the same assumptions on $V(x)$, $K(x)$ and $Q(x)$, Cingolani and Lazzo \cite{CL} used the Ljusternik-Schnirelman category to get the multiplicity of positive solutions for problem \eqref{equ1-1}. When $V(x)$, $K(x)$ and $Q(x)$ are all bounded and positive functions, Zhao and Zhao \cite{ZZ} considered the critical Schr\"{o}dinger-Poisson system \eqref{main-1} with $\varepsilon=1$ and $g(x,u)=K(x)|u|^{p-2}u+Q(x)|u|^{2^{\ast}-2}u$, and the existence of positive ground solutions was obtained. In \cite{WTXZ}, the authors proved the existence and concentration of positive solutions for system \eqref{main-1} with $g(x,u)=K(x)f(u)+Q(x)|u|^{2^{\ast}-2}u$, where $f$ is some superlinear-4 growth nonlinearity. In \cite{LG}, the Kirchhoff type problem with competing potential was considered and the existence and concentration behavior of positive solution were established. In addition, fractional Schr\"{o}dinger equations involving the competing potentials of the form
\begin{equation*}
(-\Delta)^{\frac{s}{2}}u+V(x)u=K(x)|u|^{p-2}u+Q(x)|u|^{2_s^{\ast}-2}u\quad x\in\mathbb{R}^N
\end{equation*}
with $2<p<2_s^{\ast}$, have been considered in \cite{Teng1}, the existence of ground state solution was obtained.

Motivated by the above-cited works, the aim of this paper is to consider the existence and concentration of positive solutions for fractional Schr\"{o}dinger-Poisson system with competing potentials. As far as we know, there are few results on the existence and concentration of positive solution for system \eqref{main}, and even in the $s=t=1$ case. There are some difficulties compared with classical Schrodinger-Poisson system. One is the $L^{\infty}$-estimate, owing to the work of Dipierro, Medina and Valdinoci \cite{DMV}, similarly, we can get the $L^{\infty}$-estimate. The other is the decay estimate of solutions, with the help of works in \cite{FQT}, \cite{AMi} and \cite{HZ1}, we can establish the decay estimate at infinity. Below we give some assumptions on the nonlinearity $f$:\\
$(f_0)$ $f\in C^1(\mathbb{R}^3)$, $f(t)=o(t^3)$, $f(t)t>0$ for $t>0$ and $f(t)=0$ for $t\leq0$;\\
$(f_1)$ $\frac{f(t)}{t^3}$ is strictly increasing on the interval $(0,\infty)$;\\
$(f_2)$ $f(t)\geq c_0t^{q-1}$ and $|f'(t)|\leq c_1(1+|t|^{p-2})$ for some constants $c_0,c_1>0$, where $4< q<p<2_s^{\ast}$.

Our main result is stated as follows.
\begin{theorem}\label{thm1-1}
Suppose that $(V)$, $(Q_0)$, $(Q_1)$, $(K)$, $(f_0)$--$(f_2)$ hold and let $s\in(\frac{3}{4},1)$. Then\\
$(i)$ there exists $\varepsilon_0>0$ such that system \eqref{main} possesses a positive solution $(u_{\varepsilon},\phi_{\varepsilon})\in H^s(\mathbb{R}^3)\times \mathcal{D}^{t,2}(\mathbb{R}^3)$, where $H^s(\mathbb{R}^3)$ and $\mathcal{D}^{t,2}(\mathbb{R}^3)$ are defined in Section 2;\\
$(ii)$ $u_{\varepsilon}$ possesses a maximum point $z_{\varepsilon}$ in $\mathbb{R}^3$ such that
\begin{equation*}
\lim_{\varepsilon\rightarrow0}V(z_{\varepsilon})=V_0,\quad\lim_{\varepsilon\rightarrow0}K(z_{\varepsilon})=K_0,\quad \lim_{\varepsilon\rightarrow0}Q(z_{\varepsilon})=Q_0,
\end{equation*}
and for any $z_{\varepsilon}\rightarrow x_0\in\Theta$ as $\varepsilon\rightarrow0$, set $v_{\varepsilon}(x)=u_{\varepsilon}(\varepsilon x+z_{\varepsilon})$, the solution $(v_{\varepsilon},\phi_{\varepsilon})$ converges strongly in $H^s(\mathbb{R}^3)$ to a solution $(v,\phi)$ of
\begin{equation}\label{equ1-2}
\left\{
  \begin{array}{ll}
    (-\Delta)^sv+V(x_0)v+\phi v=K(x_0)f(v)+Q(x_0)|v|^{2_s^{\ast}-2}v, & \hbox{in $\mathbb{R}^3$,} \\
    (-\Delta)^t\phi=v^2,& \hbox{in $\mathbb{R}^3$;}
  \end{array}
\right.
\end{equation}
$(iii)$ there exist two constants $C>0$ and $C_0\in\mathbb{R}$ such that
\begin{equation*}
u_{\varepsilon}\leq\frac{C\varepsilon^{3+2s}}{C_0\varepsilon^{3+2s}+|x-z_{\varepsilon}|^{3+2s}}\quad \text{for all}\,\, x\in\mathbb{R}^3.
\end{equation*}
\end{theorem}

We also obtain a supplementary result of a nonexistence of ground state solution for system \eqref{main}.
\begin{theorem}\label{thm1-2}
Assume that $(f_0)$--$(f_2)$ hold and the continuous functions $V(x)$, $K(x)$, $Q(x)$ satisfies\\
$(\mathcal{H})$ $V(x)\geq V_{\infty}=\lim\limits_{|x|\rightarrow\infty}V(x)=V_0$, $K(x)\leq K_{\infty}$ and $Q(x)\leq Q_{\infty}$, which one of the strictly inequality holds on a positive measure subset.\\
Then for any $\varepsilon>0$, system \eqref{main} has no ground state solution.
\end{theorem}

The paper is organized as follows. In section 2, we present some preliminaries results. In section 3, we will prove the compactness condition. In section 4, the existence of positive ground state solutions of autonomous problem defined in section 2 and system \eqref{main} are established. Section 5 is devoted to proving the concentration of positive solutions. Section 6 is to prove the nonexistence of ground state solutions. In Appendix, we give some estimates for extremal function defined in Section 3.

\section{Variational Setting}

In this section, we outline the variational framework for problem \eqref{main} and give some preliminary Lemma. In the sequel, we denote by $\|\cdot\|_{p}$ the usual norm of the space $L^p(\mathbb{R}^N)$, the letters $c_i$ ($i=1,2,\ldots$), $C_i$, $C$ will be indiscriminately used to denote various positive constants whose exact values are irrelevant. We denote $\widehat{u}$ the Fourier transform of $u$ for simplicity.

It is easily seen that, just performing the change of variables $u(x)\rightarrow u(x/\varepsilon)$ and $\phi(x)\rightarrow \phi(x/\varepsilon)$, and taking $z=x/\varepsilon$, problem \eqref{main} can be rewritten as the following equivalent form
\begin{equation}\label{main1}
\left\{
  \begin{array}{ll}
   (-\Delta)^su+V(\varepsilon z)u+\phi u=K(\varepsilon z)f(u)+Q(\varepsilon z)|u|^{2_s^{\ast}-2}u, & \hbox{in $\mathbb{R}^3$,} \\
    (-\Delta)^t\phi=u^2,& \hbox{in $\mathbb{R}^3$.}
  \end{array}
\right.
\end{equation}
which will be referred from now on.
\subsection{Work space stuff}
For $\alpha\in(0,1)$, we define the homogeneous fractional Sobolev space $\mathcal{D}^{\alpha,2}(\mathbb{R}^3)$ as follows
\begin{equation*}
\mathcal{D}^{\alpha,2}(\mathbb{R}^3)=\Big\{u\in L^{2_{\alpha}^{\ast}}(\mathbb{R}^3)\,\,\Big|\,\,|\xi|^{\alpha}\widehat{u}(\xi)\in L^2(\mathbb{R}^3)\Big\}
\end{equation*}
which is the completion of $C_0^{\infty}(\mathbb{R}^3)$ under the norm
\begin{equation*}
\|u\|_{\mathcal{D}^{\alpha,2}}=\Big(\int_{\mathbb{R}^3}|(-\Delta)^{\frac{\alpha}{2}}u|^2\,{\rm d}x\Big)^{\frac{1}{2}}=\Big(\int_{\mathbb{R}^3}|\xi|^{2\alpha}|\widehat{u}(\xi)|^2\,{\rm d}\xi\Big)^{\frac{1}{2}}
\end{equation*}
and for any $\alpha\in(0,1)$, there exists a best Sobolev constant $\mathcal{S}_{\alpha}>0$ such that
\begin{equation}\label{equ2-1}
\mathcal{S}_{\alpha}=\inf_{u\in \mathcal{D}^{\alpha,2}}\frac{\int_{\mathbb{R}^3}|(-\Delta)^{\frac{\alpha}{2}}u|^2\,{\rm d}x}{\Big(\int_{\mathbb{R}^3}|u(x)|^{2_{\alpha}^{\ast}}\,{\rm d}x\Big)^{\frac{2}{2_{\alpha}^{\ast}}}}.
\end{equation}

The fractional Sobolev space $H^{\alpha}(\mathbb{R}^3)$ can be described by means of the Fourier transform, i.e.
\begin{equation*}
H^{\alpha}(\mathbb{R}^3)=\Big\{u\in L^2(\mathbb{R}^3)\,\,\Big|\,\,\int_{\mathbb{R}^3}(|\xi|^{2\alpha}|\widehat{u}(\xi)|^2+|\widehat{u}(\xi)|^2)\,{\rm d}\xi<+\infty\Big\}.
\end{equation*}
In this case, the inner product and the norm are defined as
\begin{equation*}
(u,v)=\int_{\mathbb{R}^3}(|\xi|^{2\alpha}\widehat{u}(\xi)\overline{\widehat{v}(\xi)}+\widehat{u}(\xi)\overline{\widehat{v}(\xi)})\,{\rm d}\xi
\end{equation*}
\begin{equation*}
\|u\|_{H^{\alpha}}=\bigg(\int_{\mathbb{R}^3}(|\xi|^{2\alpha}|\widehat{u}(\xi)|^2+|\widehat{u}(\xi)|^2)\,{\rm d}\xi\bigg)^{\frac{1}{2}}.
\end{equation*}
From Plancherel's theorem we have $\|u\|_2=\|\widehat{u}\|_2$ and $\||\xi|^{\alpha}\widehat{u}\|_2=\|(-\Delta)^{\frac{\alpha}{2}}u\|_2$. Hence
\begin{equation*}
\|u\|_{H^{\alpha}}=\bigg(\int_{\mathbb{R}^3}(|(-\Delta)^{\frac{\alpha}{2}}u(x)|^2+|u(x)|^2)\,{\rm d}x\bigg)^{\frac{1}{2}},\quad \forall u\in H^{\alpha}(\mathbb{R}^3).
\end{equation*}
We denote $\|\cdot\|$ by $\|\cdot\|_{H^{\alpha}}$ in the sequel for convenience.

In terms of finite difference, the fractional Sobolev space $H^{\alpha}(\mathbb{R}^3)$ also can be defined as follows
\begin{equation*}
H^{\alpha}(\mathbb{R}^3)=\Big\{u\in L^2(\mathbb{R}^3)\,\,\Big|\,\,\frac{|u(x)-u(y)|}{|x-y|^{\frac{3}{2}+\alpha}}\in L^2(\mathbb{R}^3\times\mathbb{R}^3)\Big\}
\end{equation*}
endowed with the natural norm
\begin{equation*}
\|u\|=\bigg(\int_{\mathbb{R}^3}|u|^2\,{\rm d}x+\int_{\mathbb{R}^3}\int_{\mathbb{R}^3}\frac{|u(x)-u(y)|^2}{|x-y|^{3+2\alpha}}\,{\rm d}x\,{\rm d}y\bigg)^{\frac{1}{2}}.
\end{equation*}

Indeed, in view of Proposition 3.4 and Proposition 3.6 in \cite{NPV}, we have
\begin{equation*}
\|(-\Delta)^{\frac{\alpha}{2}}u\|_2^2=\int_{\mathbb{R}^3}|\xi|^{2\alpha}|\widehat{u}(\xi)|^2\,{\rm d}\xi=\frac{1}{C(\alpha)}\int_{\mathbb{R}^3}\int_{\mathbb{R}^3}\frac{|u(x)-u(y)|^2}{|x-y|^{3+2\alpha}}\,{\rm d}x\,{\rm d}y.
\end{equation*}

It is well known that $H^{\alpha}(\mathbb{R}^3)$ is continuously embedded into $L^p(\mathbb{R}^3)$ for $2\leq r\leq 2_{\alpha}^{\ast}$ ($2_{\alpha}^{\ast}=\frac{6}{3-2\alpha}$), and is locally compactly embedded into $L_{loc}^r(\mathbb{R}^3)$ for $1\leq r<2_{\alpha}^{\ast}$.

For any $\varepsilon>0$, let $H_{\varepsilon}=\{u\in H^s(\mathbb{R}^3)\,|\, \int_{\mathbb{R}^3}V(\varepsilon x)|u|^2\,{\rm d}x<\infty\}$ be the Sobolev space endowed with the norm
\begin{equation*}
\|u\|_{\varepsilon}=\Big(\int_{\mathbb{R}^3}(|(-\Delta)^{\frac{s}{2}}u|^2+V(\varepsilon x)|u|^2)\,{\rm d}x\Big)^{\frac{1}{2}}.
\end{equation*}
Clearly, by the assumption $(V)$, $\|\cdot\|_{\varepsilon}$ and $\|\cdot\|$ are equivalent norm on $H_{\varepsilon}$ uniformly for $\varepsilon>0$. Moreover, $H_{\varepsilon}$ is continuously embedded into $L^r(\mathbb{R}^3)$ for $2\leq r\leq 2_s^{\ast}$ and locally compact embedded into $L_{loc}^r(\mathbb{R}^3)$ for $1\leq r<2_s^{\ast}$.

\subsection{Formulation of Problem \eqref{main} and preliminaries}

By $(f_0)$ and $(f_1)$, for any $\varepsilon>0$, there exists $C_{\varepsilon}>0$ such that
\begin{equation}\label{equ2-2}
f(u)\leq \varepsilon|u|^2+C_{\varepsilon}|u|^{p-1}, F(u)\leq \varepsilon |u|^4+C_{\varepsilon}|u|^{p}\quad \text{for all}\,\, u\in\mathbb{R}.
\end{equation}
By $(f_1)$ and $(f_2)$, one can easily check that
\begin{equation}\label{equ2-3}
(\frac{1}{4}f(t)t-F(t))'>0, \quad f'(u)u^2-3f(u)u>0,\quad 0<4F(u)<f(u)u,\quad \text{for any}\,\, u\neq0.
\end{equation}

For problem \eqref{main1}, we first apply the usual "reduction" argument to reduce it to a single equation involving just $u$.

From \cite{Teng3}, the author has proved that if $4s+2t\geq3$, for each $u\in H^s(\mathbb{R}^3)$, the Lax-Milgram theorem implies that there exists a unique $\phi_u^t\in\mathcal{D}^{t,2}(\mathbb{R}^3)$ such that
\begin{equation*}
\int_{\mathbb{R}^3}(-\Delta)^{\frac{t}{2}}\phi_u^t(-\Delta)^{\frac{t}{2}}v\,{\rm d}x=\int_{\mathbb{R}^3}u^2v\,{\rm d}x,\quad \forall v\in\mathcal{D}^{t,2}(\mathbb{R}^3),
\end{equation*}
that is $\phi_u^t$ is a weak solution of
\begin{equation*}
(-\Delta)^t\phi_u^t=u^2,\quad x\in\mathbb{R}^3
\end{equation*}
and the representation formula holds
\begin{equation*}
\phi_u^t(x)=c_t\int_{\mathbb{R}^3}\frac{u^2(y)}{|x-y|^{3-2t}}\,{\rm d}y,\quad x\in\mathbb{R}^3,\quad c_t=\pi^{-\frac{3}{2}}2^{-2t}\frac{\Gamma(\frac{3-2t}{2})}{\Gamma(t)},
\end{equation*}
which is called $t$-Riesz potential.

Substituting $\phi_u^t$ in \eqref{main1}, it reduces to a single equation, i.e., the fractional Schr\"{o}dinger equation with a non-local term $\phi_u^t u$:
\begin{equation}\label{equ2-4}
(-\Delta)^su+V(\varepsilon x)u+\phi_u^tu=K(\varepsilon x)f(u)+Q(\varepsilon x)|u|^{2_s^{\ast}-2}u,\quad x\in\mathbb{R}^3,
\end{equation}
whose solutions can be obtained by looking for critical points of the functional $I_{\varepsilon}: H_{\varepsilon}\rightarrow\mathbb{R}$ defined by
\begin{align*}
I_{\varepsilon}(u)=\frac{1}{2}\int_{\mathbb{R}^3}\Big(|(-\Delta)^{\frac{s}{2}}u|^2+V(\varepsilon x)u^2\Big)\,{\rm d}x+\frac{1}{4}\int_{\mathbb{R}^3}&\phi_u^tu^2\,{\rm d}x-\int_{\mathbb{R}^3}K(\varepsilon x)F(u)\,{\rm d}x\\
&-\frac{1}{2_s^{\ast}}\int_{\mathbb{R}^3}Q(\varepsilon x)|u(x)|^{2_s^{\ast}}\,{\rm d}x
\end{align*}
which is well defined in $H_{\varepsilon}$ and $I_{\varepsilon}\in C^1(H_{\varepsilon},\mathbb{R})$. Moreover,
\begin{equation*}
\langle I'_{\varepsilon}(u),v\rangle=\int_{\mathbb{R}^3}\Big((-\Delta)^{\frac{s}{2}}u(-\Delta)^{\frac{s}{2}}v+V(\varepsilon x)uv+\phi_u^tuv-K(\varepsilon x)f(u)v-Q(\varepsilon x)|u|^{2_s^{\ast}-2}uv\Big)\,{\rm d}x.
\end{equation*}
\begin{definition}\label{def2-1}
(1) We call $(u_{\varepsilon},\phi_{\varepsilon})\in H_{\varepsilon}\times \mathcal{D}^{t,2}(\mathbb{R}^3)$ is a weak solution of system \eqref{main} if $u_{\varepsilon}$ is a weak solution of problem \eqref{equ2-4}.\\
(2) We call $u\in H_{\varepsilon}$ is a weak solution of \eqref{equ2-4} if
\begin{equation*}
\int_{\mathbb{R}^3}\Big((-\Delta)^{\frac{s}{2}}u(-\Delta)^{\frac{s}{2}}v+V(\varepsilon x)uv+\phi_u^tuv-K(\varepsilon x)f(u)v-Q(\varepsilon x)|u|^{2_s^{\ast}-2}uv\Big)\,{\rm d}x=0,
\end{equation*}
for any $v\in H_{\varepsilon}$.
\end{definition}
Obviously, the weak solutions of \eqref{equ2-4} are the critical points of $I_{\varepsilon}$.
Now let us summarize some properties of $\phi_u^t$.
\begin{lemma}\label{lem2-1}
For every $u\in H_{\varepsilon}$ with $4s+2t\geq3$, define $\Phi(u)=\phi_u^t\in \mathcal{D}^{t,2}(\mathbb{R}^3)$, where $\phi_u^t$ is the unique solution of equation $(-\Delta)^t\phi=u^2$. Then there hold:\\
$(i)$ If $u_n\rightharpoonup u$ in $H_{\varepsilon}$, then $\Phi(u_n)\rightharpoonup\Phi(u)$ in $\mathcal{D}^{t,2}(\mathbb{R}^3)$;\\
$(ii)$ $\Phi(tu)=t^2\Phi(u)$ for any $t\in\mathbb{R}$;\\
$(iii)$ For $u\in H_{\varepsilon}$, one has
\begin{equation*}
\|\Phi(u)\|_{\mathcal{D}^{t,2}}\leq C\|u\|_{\frac{12}{3+2t}}^2\leq C\|u\|_{\varepsilon}^2,\quad \int_{\mathbb{R}^3}\Phi(u)u^2\,{\rm d}x\leq C\|u\|_{\frac{12}{3+2t}}^4\leq C\|u\|_{\varepsilon}^4,
\end{equation*}
where constant $C$ is independent of $u$;\\
$(iv)$ Let $4s+2t>3$, if $u_n\rightharpoonup u$ in $H_{\varepsilon}$ and $u_n\rightarrow u$ in $L^r(\mathbb{R}^3)$ for $2\leq r<2_s^{\ast}$, then
\begin{equation*}
\int_{\mathbb{R}^3}\phi_{u_n}^tu_nv\,{\rm d}x\rightarrow\int_{\mathbb{R}^3}\phi_{u}^tuv\,{\rm d}x\quad \text{for any}\,\, v\in H_{\varepsilon}\quad\text{and}\quad\int_{\mathbb{R}^3}\phi_{u_n}^tu_n^2\,{\rm d}x\rightarrow\int_{\mathbb{R}^3}\phi_{u}^tu^2\,{\rm d}x;
\end{equation*}
$(v)$ For $u, v\in H_{\varepsilon}$, there holds
\begin{equation*}
\|\Phi(u)-\Phi(v)\|_{\mathcal{D}^{t,2}}\leq C(\|u\|_{\varepsilon}+\|v\|_{\varepsilon})(\|u-v\|_{\varepsilon}).
\end{equation*}
\end{lemma}
\begin{proof}
We only need to check that $(vi)$ and $(v)$ hold, the proof of others can be found in \cite{Teng3}.

$(iv)$ By H\"{o}lder's inequality and $4s+2t>3$ implying that $\frac{12}{3+2t}<\frac{6}{3-2s}$, we have that
\begin{align*}
\Big|\int_{\mathbb{R}^3}\phi_{u_n}^tu_nv\,{\rm d}x&-\int_{\mathbb{R}^3}\phi_{u}^tuv\,{\rm d}x\Big|=\Big|\int_{\mathbb{R}^3}\int_{\mathbb{R}^3}\frac{u(y)v(y)}{|x-y|^{3-2t}}(u_n^2(x)-u^2(x))\,{\rm d}x\,{\rm d}y\\
&+\int_{\mathbb{R}^3}\phi_{u_n}^t(u_n-u)v\,{\rm d}x\Big|\\
&\leq\int_{\mathbb{R}^3}(\phi_{u}^t)^{\frac{1}{2}}(\phi_{v}^t)^{\frac{1}{2}}|u_n^2(x)-u^2(x)|\,{\rm d}x+\|\phi_{u_n}^t\|_{2_t^{\ast}}\|u_n-u\|_{\frac{12}{3+2t}}\|v\|_{\frac{12}{3+2t}}\\
&\leq\|\phi_{u}^t\|_{2_t^{\ast}}^{\frac{1}{2}}\|\phi_{v}^t\|_{2_t^{\ast}}^{\frac{1}{2}}(\|u_n\|_{\frac{12}{3+2t}}+\|u\|_{\frac{12}{3+2t}})\|u_n-u\|_{\frac{12}{3+2t}}\\
&+\|\phi_{u_n}^t\|_{2_t^{\ast}}\|u_n-u\|_{\frac{12}{3+2t}}\|v\|_{\frac{12}{3+2t}}\rightarrow0
\end{align*}
for any $v\in H_{\varepsilon}$. For the second part, using the similar argument, we have
\begin{align*}
\Big|\int_{\mathbb{R}^3}\phi_{u_n}^tu_n^2\,{\rm d}x&-\int_{\mathbb{R}^3}\phi_{u}^tu^2\,{\rm d}x\Big|=\Big|\int_{\mathbb{R}^3}\int_{\mathbb{R}^3}\frac{(u_n(x)^2u_n(y)^2-u(x)^2u(y)^2)}{|x-y|^{3-2t}}\,{\rm d}x\,{\rm d}y\Big|\\
&=\Big|\int_{\mathbb{R}^3}\phi_{u_n}^t(u_n^2-u^2)\, {\rm d}x+\int_{\mathbb{R}^3}\phi_{u}^t(u_n^2-u^2)\, {\rm d}x\Big|\\
&\leq\|\phi_{u_n}^t\|_{2_t^{\ast}}\|u_n^2-u^2\|_{\frac{6}{3+2t}}+\|\phi_{u}^t\|_{2_t^{\ast}}\|u_n^2-u^2\|_{\frac{6}{3+2t}}\\
&\leq C\Big(\|u_n\|_{\varepsilon}^2+\|u\|_{\varepsilon}^2\Big)\Big(\|u_n\|_{\frac{12}{3+2t}}+\|u\|_{\frac{12}{3+2t}}\Big)\|u_n-u\|_{\frac{12}{3+2t}}\\
&\rightarrow0.
\end{align*}

$(v)$ By the definition of $\phi_u^t$ and $\phi_v^t$, using H\"{o}lder's inequality and $4s+2t\geq3$ implying that $H_{\varepsilon}\hookrightarrow L^{\frac{12}{3+2t}}(\mathbb{R}^3)$, we have that
\begin{align*}
\|\Phi(u)-\Phi(v)\|_{\mathcal{D}^{t,2}}^2&=\int_{\mathbb{R}^3}|(-\Delta)^{\frac{s}{2}}(\phi_u^t-\phi_v^t)|^2\,{\rm d}x=\int_{\mathbb{R}^3}(u^2-v^2)(\phi_u^t-\phi_v^t)\,{\rm d}x\\
&\leq C\|\phi_u^t-\phi_v^t\|_{\mathcal{D}^{t,2}}\Big(\int_{\mathbb{R}^3}|u^2-v^2|^{\frac{6}{3+2t}}\,{\rm d}x\Big)^{\frac{3+2t}{6}}\\
&\leq C\|\Phi(u)-\Phi(v)\|_{\mathcal{D}^{t,2}}\Big(\int_{\mathbb{R}^3}(|u||u-v|+|v||u-v|)^{\frac{6}{3+2t}}\,{\rm d}x\Big)^{\frac{3+2t}{6}}\\
&\leq C\|\Phi(u)-\Phi(v)\|_{\mathcal{D}^{t,2}}\Big(\|u\|_{\frac{12}{3+2t}}+\|v\|_{\frac{12}{3+2t}}\Big)\|u-v\|_{\frac{12}{3+2t}}\\
&\leq C\|\Phi(u)-\Phi(v)\|_{\mathcal{D}^{t,2}}\Big(\|u\|_{\varepsilon}+\|v\|_{\varepsilon}\Big)\|u-v\|_{\varepsilon}.
\end{align*}
\end{proof}
The minimax level of the autonomous equation associated with equation \eqref{equ2-4}
\begin{equation}\label{equ2-5}
(-\Delta)^su+\nu u+\phi_u^t u=\kappa f(u)+\mu|u|^{2_s^{\ast}-2}u, \quad \text{in}\,\,\mathbb{R}^3, \\
\end{equation}
plays an important role in the proof of compactness of $(PS)$ sequence and concentration behavior of solutions, where $\mu,\nu,\kappa>0$ are arbitrary positive constants and $\phi_u^t=\int_{\mathbb{R}^3}\frac{u^2(y)}{|x-y|^{3-2t}}\,{\rm d}x$.
For $\nu>0$, let
\begin{equation*}
E_{\nu}=\{u\in H^s(\mathbb{R}^3)\,\Big|\,\, \int_{\mathbb{R}^3}\nu |u|^2\,{\rm d}x<\infty\}
\end{equation*}
be a Sobolev space endowed with the norm
\begin{equation*}
\|u\|_{E_{\nu}}=\Big(\int_{\mathbb{R}^3}\Big(|(-\Delta)^{\frac{s}{2}}u|^2+\nu|u|^2\Big)\,{\rm d}x\Big)^{\frac{1}{2}}.
\end{equation*}
In fact, the Sobolev space $H_{\varepsilon}=E_{\nu}=H^s(\mathbb{R}^3)$ for any $\varepsilon>0$ and $\nu>0$.

The energy functional $I_{\nu,\kappa,\mu}: E_{\nu}\rightarrow\mathbb{R}$ is given by
\begin{align*}
I_{\nu,\kappa,\mu}(u)=\frac{1}{2}\int_{\mathbb{R}^3}(|(-\Delta)^{\frac{s}{2}}u|^2+\nu u^2)\,{\rm d}x+&\frac{1}{4}\int_{\mathbb{R}^3}\phi_u^tu^2\,{\rm d}x-\kappa\int_{\mathbb{R}^3} F(u)\,{\rm d}x\\
&-\frac{\mu}{2_s^{\ast}}\int_{\mathbb{R}^3}|u(x)|^{2_s^{\ast}}\,{\rm d}x.
\end{align*}
It is easy to see that $I_{\nu,\kappa,\mu}\in C^1(E_{\nu},\mathbb{R})$ and
\begin{align*}
\langle I_{\nu,\kappa,\mu}'(u),\varphi\rangle=\int_{\mathbb{R}^3}\Big((-\Delta)^{\frac{s}{2}}u(-\Delta)^{\frac{s}{2}}\varphi+\nu u\varphi\Big)\,{\rm d}x+&\int_{\mathbb{R}^3}\phi_u^tu\varphi\,{\rm d}x-\kappa\int_{\mathbb{R}^3}f(u)\varphi\,{\rm d}x\\
&-\mu\int_{\mathbb{R}^3}|u(x)|^{2_s^{\ast}-2}u\varphi\,{\rm d}x,
\end{align*}
for any $u,\varphi\in E_{\nu}$. Moreover, the critical points of $I_{\nu,\kappa,\mu}$ in $E_{\nu}$ are weak solutions of equation \eqref{equ2-5}.

In section 3, we will apply the concentration-compactness principle of P. L. Lions \cite{Lions1,Lions2} and vanishing Lemma \cite{Sec} to prove the compactness of $(PS)_c$ sequence of $I_{\varepsilon}$ on some low energy level. We first recall these results as follows.
\begin{proposition}\label{pro2-1}
Let $\rho_n(x)\in L^1(\mathbb{R}^N)$ be a non-negative sequence satisfying
\begin{equation*}
\int_{\mathbb{R}^N}\rho_n(x)\,{\rm d}x=l>0.
\end{equation*}
Then there exists a subsequence, still denoted by $\{\rho_n(x)\}$ such that one of the following cases occurs.\\
(i) (compactness) there exists $y_n\in\mathbb{R}^N$, such that for any $\varepsilon>0$, exists $R>0$ such that
\begin{equation*}
\int_{B_R(y_n)}\rho_n(x)\,{\rm d}x\geq l-\varepsilon,\quad n=1,2,\cdots.
\end{equation*}
(ii) (vanishing) for any fixed $R>0$, there holds
\begin{equation*}
\lim_{n\rightarrow\infty}\sup_{y\in\mathbb{R}^N}\int_{B_R(y)}\rho_n(x)\,{\rm d}x=0.
\end{equation*}
(iii) (dichotomy) there exists $\alpha\in(0,l)$ such that for any $\varepsilon>0$, there exists $n_0\geq1$, $\rho_n^{(1)}$, $\rho_n^{(2)}\in L^1(\mathbb{R}^N)$, for $n\geq n_0$, there holds
\begin{equation*}
\|\rho_n-(\rho_n^{(1)}+\rho_n^{(2)})\|_{L^1(\mathbb{R}^N}<\varepsilon,\quad \Big|\int_{\mathbb{R}^N}\rho_n^{(1)}(x)\,{\rm d}x-\alpha\Big|<\varepsilon,\quad \Big|\int_{\mathbb{R}^N}\rho_n^{(2)}(x)\,{\rm d}x-(l-\alpha)\Big|<\varepsilon
\end{equation*}
and
\begin{equation*}
{\rm dist}({\rm supp}\rho_n^{(1)},{\rm supp}\rho_n^{(2)})\rightarrow\infty,\quad \text{as}\,\, n\rightarrow\infty.
\end{equation*}
\end{proposition}

\begin{lemma}\label{lem2-2}
Assume that $\{u_n\}$ is bounded in $H^s(\mathbb{R}^N)$ and it satisfies
\begin{equation*}
\lim_{n\rightarrow+\infty}\sup_{y\in\mathbb{R}^N}\int_{B_R(y)}|u_n(x)|^2\,{\rm d}x=0
\end{equation*}
where $R>0$. Then $u_n\rightarrow0$ in $L^r(\mathbb{R}^N)$ for every $2<r<2_s^{\ast}$.
\end{lemma}

\section{Compactness}
Define the Nehari manifold associated to the functional $I_{\varepsilon}$ as
\begin{equation*}
\mathcal{N}_{\varepsilon}=\Big\{u\in H_{\varepsilon}\backslash\{0\}\,\Big|\,\, G_{\varepsilon}(u)=0\Big\},
\end{equation*}
where
\begin{align*}
 G_{\varepsilon}(u)=\langle I'_{\varepsilon}(u),u\rangle=&\int_{\mathbb{R}^3}\Big(|(-\Delta)^{\frac{s}{2}}u|^2+V(\varepsilon x)u^2\Big)\,{\rm d}x+\int_{\mathbb{R}^3}\phi_u^tu^2\,{\rm d}x\\
&-\int_{\mathbb{R}^3}K(\varepsilon x)f(u)u\,{\rm d}x-\int_{\mathbb{R}^3}Q(\varepsilon x)|u(x)|^{2_s^{\ast}}\,{\rm d}x.
\end{align*}
Thus, for any $u\in \mathcal{N}_{\varepsilon}$, we have that
\begin{align*}
I_{\varepsilon}(u)=\frac{1}{4}\int_{\mathbb{R}^3}(|(-\Delta)^{\frac{s}{2}}u|^2+V(\varepsilon x)u^2)\,{\rm d}x&+\int_{\mathbb{R}^3}K(\varepsilon x)(\frac{1}{4}f(u)u-F(u))\,{\rm d}x\\
&+\frac{4s-3}{12}\int_{\mathbb{R}^3}Q(\varepsilon x)|u(x)|^{2_s^{\ast}}\,{\rm d}x.
\end{align*}
\begin{remark}\label{rem3-1}
Observing that $s>\frac{3}{4}$ implies that $4s+2t>3$ holds trivially.
\end{remark}
Also, we define the Nehari manifold associated with functional $I_{\mu,\kappa,\nu}$ as follows
\begin{equation*}
\mathcal{N}_{\nu,\kappa,\mu}=\{u\in E_{\nu}\backslash\{0\}\,\,\Big|\,\, G_{\nu,\kappa,\mu}(u)=0\},
\end{equation*}
where
\begin{equation*}
 G_{\nu,\kappa,\mu}(u)=\|u\|_{E_{\nu}}^2+\int_{\mathbb{R}^3}\phi_u^tu^2\,{\rm d}x-\kappa\int_{\mathbb{R}^3}f(u)u\,{\rm d}x-\mu\int_{\mathbb{R}^3}|u(x)|^{2_s^{\ast}}\,{\rm d}x.
\end{equation*}
Particularly, when $\mu=V_{\infty}$, $\kappa=K_{\infty}$, $\mu=Q_{\infty}$, and $\mu=V_0$, $\kappa=K_0$, $\mu=Q_0$, the functional $I_{\infty}$ and $I_0$ are defined as
\begin{align*}
I_{\infty}(u)=\frac{1}{2}\int_{\mathbb{R}^3}(|(-\Delta)^{\frac{s}{2}}u|^2+V_{\infty} u^2)\,{\rm d}x+&\frac{1}{4}\int_{\mathbb{R}^3}\phi_u^tu^2\,{\rm d}x-\int_{\mathbb{R}^3} K_{\infty}F(u)\,{\rm d}x\\
&-\frac{1}{2_s^{\ast}}\int_{\mathbb{R}^3}Q_{\infty}|u(x)|^{2_s^{\ast}}\,{\rm d}x
\end{align*}
and
\begin{align*}
I_{0}(u)=\frac{1}{2}\int_{\mathbb{R}^3}(|(-\Delta)^{\frac{s}{2}}u|^2+V_0 u^2)\,{\rm d}x+&\frac{1}{4}\int_{\mathbb{R}^3}\phi_u^tu^2\,{\rm d}x-\int_{\mathbb{R}^3} K_0F(u)\,{\rm d}x\\
&-\frac{1}{2_s^{\ast}}\int_{\mathbb{R}^3}Q_0|u(x)|^{2_s^{\ast}}\,{\rm d}x.
\end{align*}
Also, we denote the Nehari manifolds by
\begin{equation*}
\mathcal{N}_{\infty}=\{u\in H^s(\mathbb{R}^3)\backslash\{0\}\,\,\Big|\,\,\langle I_{\infty}'(u),u\rangle=0\} \quad \mathcal{N}_0=\{u\in H^s(\mathbb{R}^3)\backslash\{0\}\,\,\Big|\,\,\langle I_0'(u),u\rangle=0\}.
\end{equation*}
In order to find the least energy solutions of problem \eqref{equ2-4} and \eqref{equ2-5}, we define the least energy levels as follows
\begin{equation*}
m_{\varepsilon}=\inf_{u\in\mathcal{N}_{\varepsilon}}I_{\varepsilon}(u),\quad m_{\nu,\kappa,\mu}=\inf_{u\in\mathcal{N}_{\nu,\kappa,\mu}}I_{\nu,\kappa,\mu}(u),\quad m_0=\inf_{u\in\mathcal{N}_0}I_0(u),\quad m_{\infty}=\inf_{u\in\mathcal{N}_{\infty}}I_{\infty}(u).
\end{equation*}
The following lemma describes some properties of the Nehari manifold $\mathcal{N}_{\varepsilon}$, $\mathcal{N}_{\nu,\kappa,\mu}$ and $I_{\varepsilon}$.
\begin{lemma}\label{lem3-1}
Under the assumptions $(V)$, $(Q_0)$, $(K)$ and $(f_0)-(f_2)$, the the following statements hold:\\
$(i)$ $\mathcal{N}_{\varepsilon}$ is a manifold of $C^1$-class diffeomorphic to the unite sphere of $H_{\varepsilon}$;\\
$(ii)$ For every $u\in H_{\varepsilon}\backslash\{0\}$, there exists a unique $t_u>0$ such that $t_uu\in\mathcal{N}_{\varepsilon}$ and
\begin{equation*}
I_{\varepsilon}(t_uu)=\max_{t\geq0}I_{\varepsilon}(tu);
\end{equation*}
$(iii)$ For every $u\in \mathcal{N}_{\varepsilon}$, there exists $C>0$ such that $\|u\|_{\varepsilon}\geq C>0$;\\
$(iv)$ If $\{u_n\}$ is a $(PS)_c$ sequence in $H_{\varepsilon}$, i.e., $I_{\varepsilon}(u_n)\rightarrow c$ and $I_{\varepsilon}'(u_n)\rightarrow0$ in $(H_{\varepsilon})'$ as $n\rightarrow\infty$, then there exists $u\in H_{\varepsilon}$ such that $u_n\rightharpoonup u$ in $H_{\varepsilon}$ and $I_{\varepsilon}'(u)=0$ if $2s+2t>3$;\\
$(v)$ Let $\{u_n\}\subset H_{\varepsilon}$ be a sequence satisfying
\begin{equation*}
\langle I_{\varepsilon}'(u_n),u_n\rangle\rightarrow0\quad \text{and} \quad \int_{\mathbb{R}^3}\Big(K(\varepsilon x)f(u_n)u_n+Q(\varepsilon x)|u_n|^{2_s^{\ast}}\Big)\,{\rm d}x\rightarrow a
\end{equation*}
as $n\rightarrow\infty$, where $a$ is a positive constant. Then, up to a subsequence, there exists $t_n>0$ such that
\begin{equation*}
\langle I_{\varepsilon}'(t_nu_n),t_nu_n\rangle=0\quad \text{and}\quad t_n\rightarrow1 \,\,\text{as}\,\, n\rightarrow\infty;
\end{equation*}
$(vi)$ Let $\{u_n\}$ be a sequence such that $u_n\in \mathcal{N}_{\varepsilon}$ and $I_{\varepsilon}(u_n)\rightarrow m_{\varepsilon}$, then we may assume that $\{u_n\}$ is a $(PS)_{m_{\varepsilon}}$ sequence in $H_{\varepsilon}$.
\end{lemma}

\begin{proof}
The proof of $(ii)$ and $(iii)$ is standard, we only to verify the remain conclusions.

$(i)$ Let $u\in\mathcal{N}_{\varepsilon}$, by computation, using \eqref{equ2-3}, we get
\begin{align*}
\langle G_{\varepsilon}'(u),u\rangle&=2\int_{\mathbb{R}^3}\Big(|(-\Delta)^{\frac{s}{2}}u|^2+V(\varepsilon x)u^2\Big)\,{\rm d}x+4\int_{\mathbb{R}^3}\phi_u^tu^2\,{\rm d}x\\
&-\int_{\mathbb{R}^3}K(\varepsilon x)\Big(f'(u)u^2+f(u)u\Big)\,{\rm d}x-2_s^{\ast}\int_{\mathbb{R}^3}Q(\varepsilon x)|u(x)|^{2_s^{\ast}}\,{\rm d}x\\
&=-2\int_{\mathbb{R}^3}\Big(|(-\Delta)^{\frac{s}{2}}u|^2+V(\varepsilon x)u^2\Big)\,{\rm d}x-\int_{\mathbb{R}^3}K(\varepsilon x)\Big(f'(u)u^2-3f(u)u\Big)\,{\rm d}x\\
&-(2_s^{\ast}-4)\int_{\mathbb{R}^3}Q(\varepsilon x)|u(x)|^{2_s^{\ast}}\,{\rm d}x\\
&\leq-2\|u\|_{\varepsilon}^2<0.
\end{align*}
$(iv)$ Let $\{u_n\}\subset H_{\varepsilon}$ be a $(PS)_c$ sequence, it is easy to show that $\{u_n\}$ is bounded in $H_{\varepsilon}$. By the reflexivity of $H_{\varepsilon}$ and using the Sobolev embedding property, up to a subsequence, still denoted by $\{u_n\}$, we may assume that there exists $u\in H_{\varepsilon}$ such that $u_n\rightharpoonup u$ in $H_{\varepsilon}$, $u_n\rightarrow u$ in $L_{loc}^r(\mathbb{R}^3)$ with $1\leq r<2_s^{\ast}$ and $u_n\rightarrow u$ a.e. in $\mathbb{R}^3$. By \eqref{equ2-2}, we see that $f(u)\in L^{\frac{2_s^{\ast}-1}{p-1}}(\mathbb{R}^3)$ because of $2<\frac{2(2_s^{\ast}-1)}{p-1}<2_s^{\ast}-1$. Therefore, $Q(\varepsilon x)^{\frac{2_s^{\ast}-1}{2_s^{\ast}}}|u_n|^{2_s^{\ast}-2}u_n$ and $K(\varepsilon x)f(u_n)$ are bounded in $L^{\frac{2_s^{\ast}}{2_s^{\ast}-1}}(\mathbb{R}^3)$ and $L^{\theta\frac{2_s^{\ast}-1}{p-1}}(\mathbb{R}^3)$ with $\theta\in(\frac{p-1}{2_s^{\ast}-1}\frac{2_s^{\ast}}{2_s^{\ast}-1},\frac{2_s^{\ast}}{2_s^{\ast}-1})$, respectively. Here using $s>\frac{3}{4}$, we can choose $\theta\in(\frac{p-1}{2_s^{\ast}-1}\frac{2_s^{\ast}}{2_s^{\ast}-1},\frac{2_s^{\ast}}{2_s^{\ast}-1})$ such that $2<2\theta\frac{2_s^{\ast}-1}{p-1}<2_s^{\ast}$, $2<\theta(2_s^{\ast}-1)<2_s^{\ast}$ and $2<\frac{\theta(2_s^{\ast}-1)}{\theta(2_s^{\ast}-1)-(p-1)}<2_s^{\ast}$. Using the fact $u_n\rightarrow u$ a..e in $\mathbb{R}^3$, we obtain that
\begin{equation}\label{equ3-1}
Q(\varepsilon x)^{\frac{2_s^{\ast}-1}{2_s^{\ast}}}|u_n|^{2_s^{\ast}-2}u_n\rightharpoonup Q(\varepsilon x)^{\frac{2_s^{\ast}-1}{2_s^{\ast}}}|u|^{2_s^{\ast}-2}u\quad \text{in}\,\,L^{\frac{2_s^{\ast}}{2_s^{\ast}-1}}(\mathbb{R}^3)
\end{equation}
and
\begin{equation}\label{equ3-2}
K(\varepsilon x)f(u_n)\rightharpoonup K(\varepsilon x)f(u)\quad \text{in} \,\,L^{\theta\frac{2_s^{\ast}-1}{p-1}}(\mathbb{R}^3).
\end{equation}
Next, we show that $\phi_{u_n}^t\rightarrow\phi_u^t$ a.e. in $\mathbb{R}^3$. In fact, using $2s+2t>3$ and choosing $\frac{3}{3-2t}<p<\frac{3}{2s}$, $q>\frac{3}{2s}>\frac{3}{3-2t}$ so that $2p',2q'\in(2,2_s^{\ast})$, using H\"{o}lder's inequality, we deduce that
\begin{align*}
|\phi_{u_n}^t (x)-\phi_u^t (x)|&\leq c_t\int_{\mathbb{R}^3}\frac{|u_n^2(y)-u^2(y)|}{|x-y|^{3-2t}}\,{\rm d}y\\
&\leq c_t\Big(\int_{|x-y|\leq R}\frac{1}{|x-y|^{p(3-2t)}}\,{\rm d}y\Big)^{\frac{1}{p}}\Big(\int_{|x-y|\leq R}|u_n^2-u^2|^{p'}\,{\rm d}y\Big)^{\frac{1}{p'}}\\
&+\Big(\int_{|x-y|> R}\frac{1}{|x-y|^{q(3-2t)}}\,{\rm d}y\Big)^{\frac{1}{q}}\Big(\int_{|x-y|> R}|u_n^2-u^2|^{q'}\,{\rm d}y\Big)^{\frac{1}{q'}}
\end{align*}
concluding the pointwise convergence. Owing to $\phi_{u_n}^tu_n$ is bounded in $L^{\gamma}(\mathbb{R}^3)$, where $\gamma$ satisfies $2\leq\frac{\gamma 2_t^{\ast}}{2_t^{\ast}-\gamma}\leq2_s^{\ast}$ and $2\leq\gamma'\leq2_s^{\ast}$ (using $4s+2t\geq3$). Using $\phi_{u_n}^tu_n\rightarrow \phi_u^tu$ a.e. in $\mathbb{R}^3$, we have that
\begin{equation}\label{equ3-3}
\phi_{u_n}^tu_n\rightharpoonup\phi_u^tu\quad \text{in}\,\, L^{\gamma}(\mathbb{R}^3).
\end{equation}
Therefore, combining with \eqref{equ3-1}, \eqref{equ3-2}, \eqref{equ3-3} and using the weakness convergence in $H_{\varepsilon}$, we infer that
\begin{equation*}
\langle I_{\varepsilon}'(u_n),\varphi\rangle-\langle I_{\varepsilon}'(u),\varphi\rangle\rightarrow0,\quad \text{for any}\,\, \varphi\in H_{\varepsilon}.
\end{equation*}
which leads to $I_{\varepsilon}'(u)=0$.

$(v)$ By the assumptions, it is easy to see that $\|u_n\|_{\varepsilon}\neq0$ for large $n$. Using the conclusion $(ii)$, there exists $t_n>0$ such that $t_nu_n\in\mathcal{N}_{\varepsilon}$ i.e., $\langle I_{\varepsilon}(t_nu_n),t_nu_n\rangle=0$. Now we prove that $t_n\rightarrow1$ as $n\rightarrow\infty$. Set
\begin{equation*}
a_n=\|u_n\|_{\varepsilon}^2,\,\, b_n=\int_{\mathbb{R}^3}\phi_{u_n}^tu_n^2\,{\rm d}x,\,\, c_n=\int_{\mathbb{R}^3}K(\varepsilon x)f(u_n)u_n\,{\rm d}x,\,\, d_n=\int_{\mathbb{R}^3}Q(\varepsilon x)|u_n|^{2_s^{\ast}}\,{\rm d}x.
\end{equation*}
By assumption, we have that $c_n+d_n\rightarrow a>0$ as $n\rightarrow\infty$. Passing to a subsequence, we may assume that
\begin{equation*}
a_n\rightarrow a_0,\,\,b_n\rightarrow b_0,\,\,c_n\rightarrow c_0,\,\,d_n\rightarrow d_0.
\end{equation*}
Thus, $a_0+b_0=c_0+d_0=a$ and $a_0>0$ (if not, contradiction with $a>0$). From $\langle I_{\varepsilon}(t_nu_n),t_nu_n\rangle=0$, by \eqref{equ2-2} and $(f_2)$, we have that
\begin{equation*}
t_n^2a_n+t_n^4b_n\geq t_n^qK_{\infty}\int_{\mathbb{R}^3}|u_n|^q\,{\rm d}x+t_n^{2_s^{\ast}}d_n
\end{equation*}
and
\begin{equation*}
t_n^2a_n+t_n^4b_n\leq C\Big(\varepsilon  t_n^2\int_{\mathbb{R}^3}|u_n|^2\,{\rm d}x+t_n^p\int_{\mathbb{R}^3}|u_n|^p\,{\rm d}x\Big)+t_n^{2_s^{\ast}}d_n
\end{equation*}
which imply that there exist $T_1,T_2>0$ such that $0<T_1\leq t_n\leq T_2$. Hence, up to a subsequence, still denoted by $\{t_n\}$, we may assume that $t_n\rightarrow T$ as $n\rightarrow\infty$. Since $\langle I_{\varepsilon}'(t_nu_n),t_nu_n\rangle=0$ and $\langle I_{\varepsilon}'(u_n),u_n\rangle\rightarrow0$, we get
\begin{equation*}
T^2a_n+T^4b_n=\int_{\mathbb{R}^3}K(\varepsilon x)f(Tu_n)Tu_n\,{\rm d}x+T^{2_s^{\ast}}d_n+o_n(1),\quad a_n+b_n=c_n+d_n+o_n(1)
\end{equation*}
which leads to
\begin{align*}
\int_{\mathbb{R}^3}K(\varepsilon x)\Big(\frac{f(Tu_n)}{(Tu_n)^3}-\frac{f(u_n)}{(u_n)^3}\Big)u_n^4\,{\rm d}x&=(\frac{1}{T^2}-1)a_n+(1-T^{2_s^{\ast}-4})d_n+o_n(1)\\
&=(\frac{1}{T^2}-1)a_0+(1-T^{2_s^{\ast}-4})d_0+o_n(1).
\end{align*}
Observing that $(\frac{1}{T^2}-1)a_0+(1-T^{2_s^{\ast}-4})d_0<0$ when $T>1$, but it follows from Fatou's Lemma and $u_n\rightarrow u$ a.e. in $\mathbb{R}^3$ that
\begin{equation*}
\int_{\mathbb{R}^3}K(\varepsilon x)\Big(\frac{f(Tu)}{(Tu)^3}-\frac{f(u)}{u^3}\Big)u^4\,{\rm d}x\geq0
\end{equation*}
which is impossible. Using similar argument, we can get a contradiction when $T<1$. Hence, it is only true that $T=1$.

$(vi)$ Suppose $\{u_n\}$ be a minimizing sequence of $I_{\varepsilon}$ constrained in $\mathcal{N}_{\varepsilon}$. By the Ekeland's variational principle in \cite{Willem} (Theorem 8.5, Page 122), there exists a sequence $\{v_n\}\subset\mathcal{N}_{\varepsilon}$ such that $I_{\varepsilon}(v_n)\rightarrow m_{\varepsilon}$, $\|v_n-u_n\|_{\varepsilon}\rightarrow 0$ and  $I_{\varepsilon}'(v_n)-\lambda_nG_{\varepsilon}'(v_n)\rightarrow 0$ as $n\rightarrow\infty$. Thus, in view of $v_n\in\mathcal{N}_{\varepsilon}$, we have that $\langle I_{\varepsilon}'(v_n),v_n\rangle=\lambda_n\langle G_{\varepsilon}(v_n),v_n\rangle +o_n(1)=o_n(1)$. We assume that $\lim\limits_{n\rightarrow\infty}\langle G_{\varepsilon}(v_n),v_n\rangle=\gamma\leq0$, if $\gamma=0$, owing to the proof of $(i)$, we see that $\|v_n\|_{\varepsilon}\rightarrow0$ as $n\rightarrow\infty$. This yields a contradiction with $m_{\varepsilon}>0$. Therefore, $\lambda_n\rightarrow0$ as $n\rightarrow\infty$ and so $I_{\varepsilon}'(v_n)\rightarrow0$ as $n\rightarrow\infty$. Hence, without loss of generalization, we may assume that $I_{\varepsilon}(u_n)\rightarrow m_{\varepsilon}$ and $I_{\varepsilon}'(u_n)\rightarrow 0$ as $n\rightarrow\infty$, i.e., $\{u_n\}$ is a $(PS)_{m_{\varepsilon}}$ sequence for $I_{\varepsilon}$.
\end{proof}

The functional $I_{\varepsilon}$ satisfies the mountain pass geometry.
\begin{lemma}\label{lem3-2}
Suppose $(V)$, $(K)$, $(Q_0)$ and $(f_0)-(f_2)$ hold, then the functional $I_{\varepsilon}$ has the following properties:\\
$(i)$ there exists $\alpha,\rho>0$ such that $I_{\varepsilon}(u)\geq\alpha$ for $\|u\|_{\varepsilon}=\rho$;\\
$(ii)$ there exists $e\in H_{\varepsilon}$ satisfying $\|e\|_{\varepsilon}>\rho$ such that $I_{\varepsilon}(e)<0$.
\end{lemma}
The proof of Lemma \ref{lem3-2} is standard and hence is omitted. By Lemma \ref{lem3-2} and Theorem 1.15 in \cite{Willem} (Mountain pass theorem without Palais-Smale condition), it follows that there exists a $(PS)_{c_{\varepsilon}}$ sequence $\{u_n\}\subset H_{\varepsilon}$ such that
\begin{equation*}
I_{\varepsilon}(u_n)\rightarrow c_{\varepsilon}\quad\text{and}\quad I'_{\varepsilon}(u_n)\rightarrow0\,\,\text{as}\,\, n\rightarrow\infty,
\end{equation*}
where $c_{\varepsilon}=\inf_{\gamma\in \Gamma_{\varepsilon}}\max_{t\in[0,1]}I_{\varepsilon}(\gamma(t))$,
where $\Gamma_{\varepsilon}=\{\gamma\in C([0,1],H_{\varepsilon})\,\,|\,\, \gamma(0)=0, I_{\varepsilon}(\gamma(1))<0\}$.
Similarly to the arguments in \cite{Rabinowtiz} or \cite{HZ},  by \eqref{equ2-3}, the equivalent characterization of $c_{\varepsilon}$ is given by
\begin{equation*}
c_{\varepsilon}=m_{\varepsilon}=\inf_{u\in H_{\varepsilon}\backslash\{0\}}\max_{t\geq0}I_{\varepsilon}(tu).
\end{equation*}
The following Lemma gives the estimate of the critical value $c_{\varepsilon}$.
\begin{lemma}\label{lem3-3}
Suppose that $(V)$, $(K)$, $(Q_0)$, $(Q_1)$ and $(f_0)-(f_2)$ hold, then the infinimum $c_{\varepsilon}$ satisfies
\begin{equation*}
0<c_{\varepsilon}<\frac{s}{3Q(x_0)^{\frac{3-2s}{2s}}}\mathcal{S}_s^{\frac{3}{2s}},
\end{equation*}
for $\varepsilon$ small enough, where $\mathcal{S}_s$ is the best Sobolev constant defined by \eqref{equ2-1}.
\end{lemma}
\begin{proof}
We define
\begin{equation*}
u_{\varepsilon}(x)=\psi(x-x_0/\varepsilon)U_{\varepsilon}(x),\quad x\in\mathbb{R}^3,
\end{equation*}
where $U_{\varepsilon}(x)=\varepsilon^{\frac{3-2s}{2}}u^{\ast}(x)$, $u^{\ast}(x)=\frac{\kappa}{(\varepsilon^2+|x-x_0/\varepsilon|^2)^{\frac{3-2s}{2}}}$ (see Appendix), and $\psi\in C^{\infty}(\mathbb{R}^3)$ such that $0\leq\psi\leq1$ in $\mathbb{R}^3$, $\psi(x)\equiv1$ in $B_{R}(0)$ and $\psi\equiv0$ in $\mathbb{R}^3\backslash B_{2R}(0)$. From Lemma A.2 and Lemma A.3 in Appendix, we know that
\begin{equation}\label{equ3-4}
\int_{\mathbb{R}^{3}}|(-\Delta)^{\frac{s}{2}}u_{\varepsilon}(x)|^2\,{\rm d}x\leq \|(-\Delta)^{\frac{s}{2}}\bar{u}\|_2^2+O(\varepsilon^{3-2s}),
\end{equation}
\begin{equation}\label{equ3-5}
\int_{\mathbb{R}^{3}}|u_{\varepsilon}(x)|^{2_s^{\ast}}\,{\rm d}x=\|\bar{u}\|_{2_s^{\ast}}^{2_s^{\ast}}+O(\varepsilon^3),
\end{equation}
where $\bar{u}=\frac{\kappa}{(1+|x|^2)^{\frac{3-2s}{2}}}$ is such that $\mathcal{S}_s=\frac{|(-\Delta)^{\frac{s}{2}}\bar{u}\|_2^2}{\|\bar{u}\|_{2_s^{\ast}}^2}$,
and
\begin{equation}\label{equ3-6}
\int_{\mathbb{R}^3}|u_{\varepsilon}(x)|^p\,{\rm d}x=\left\{
\begin{array}{ll}
O(\varepsilon^{\frac{(2-p)3+2sp}{2}}),&\hbox{$p>\frac{3}{3-2s}$,} \\
O(\varepsilon^{\frac{(2-p)3+2sp}{2}}|\log\varepsilon|), & \hbox{$p=\frac{3}{3-2s}$,} \\
O(\varepsilon^{\frac{3-2s}{2}p}), & \hbox{$1< p<\frac{3}{3-2s}$.}
\end{array}
\right.
\end{equation}

By $(ii)$ of Lemma \ref{lem3-1}, there exists $t_{\varepsilon}>0$ such that $\sup\limits_{t\geq0}I_{\varepsilon}(tu_{\varepsilon})=I_{\varepsilon}(t_{\varepsilon}u_{\varepsilon})$. Hence $\frac{{\rm d}I_{\varepsilon}(tu_{\varepsilon})}{{\rm d} t}\Big|_{t=t_{\varepsilon}}=0$, that is
\begin{align*}
t_{\varepsilon}^2\int_{\mathbb{R}^3}(|(-\Delta)^{\frac{s}{2}}u_{\varepsilon}|^2+V(\varepsilon x) u_{\varepsilon}^2)\,{\rm d}x+t_{\varepsilon}^4\int_{\mathbb{R}^3}\phi_{u_{\varepsilon}}^tu_{\varepsilon}^2\,{\rm d}x=&\int_{\mathbb{R}^3} K(\varepsilon x)f(t_{\varepsilon}u_{\varepsilon})t_{\varepsilon}u_{\varepsilon}\,{\rm d}x\\
&+t_{\varepsilon}^{2_s^{\ast}}\int_{\mathbb{R}^3}Q(\varepsilon x)|u_{\varepsilon}|^{2_s^{\ast}}\,{\rm d}x.
\end{align*}
By $(f_0)$, we have that
\begin{align*}
t_{\varepsilon}^2\int_{\mathbb{R}^3}(|(-\Delta)^{\frac{s}{2}}u_{\varepsilon}|^2+V(\varepsilon x) u_{\varepsilon}^2)\,{\rm d}x+t_{\varepsilon}^4\int_{\mathbb{R}^3}\phi_{u_{\varepsilon}}^tu_{\varepsilon}^2\,{\rm d}x\geq t_{\varepsilon}^{2_s^{\ast}}\int_{\mathbb{R}^3}Q(\varepsilon x)|u_{\varepsilon}|^{2_s^{\ast}}\,{\rm d}x.
\end{align*}
It follows from $(iii)$ of Lemma \ref{lem2-1} that
\begin{align}\label{equ3-7}
t_{\varepsilon}^{2_s^{\ast}}\leq\frac{1}{\|Q(\varepsilon x)^{\frac{1}{2_s^{\ast}}}u_{\varepsilon}\|_{2_s^{\ast}}^{2_s^{\ast}}}\Big(t_{\varepsilon}^2\int_{\mathbb{R}^3}(|(-\Delta)^{\frac{s}{2}}u_{\varepsilon}|^2+V(\varepsilon x) u_{\varepsilon}^2)\,{\rm d}x+Ct_{\varepsilon}^4\|u_{\varepsilon}\|_{\frac{12}{3+2t}}^4\Big).
\end{align}
\eqref{equ3-4}, \eqref{equ3-5}, \eqref{equ3-6} and \eqref{equ3-7} imply that $|t_{\varepsilon}|\leq C_1$, where $C_1$ is independent of $\varepsilon>0$ small. On the other hand, we may assume that there is a positive constant $C_2>0$ such that $t_{\varepsilon}\geq C_2>0$ for $\varepsilon>0$ small. Otherwise, we can find a sequence $\varepsilon_n\rightarrow0$ as $n\rightarrow\infty$ such that $t_{\varepsilon_n}\rightarrow0$ as $n\rightarrow\infty$. Therefore
\begin{equation*}
0<c_{\varepsilon}\leq \sup_{t\geq0}I_{\varepsilon}(tu_{\varepsilon_n})=I_{\varepsilon}(t_{\varepsilon_n}u_{\varepsilon_n})\rightarrow0,
\end{equation*}
which is a contradiction.

Denote $g(t)=\frac{t^2}{2}\int_{\mathbb{R}^3}|(-\Delta)^{\frac{s}{2}}u_{\varepsilon}|^2\,{\rm d}x-\frac{ t^{2_s^{\ast}}}{2_s^{\ast}}\int_{\mathbb{R}^3}Q(x_0)|u_{\varepsilon}|^{2_s^{\ast}}\,{\rm d}x$, by \eqref{equ3-4} and \eqref{equ3-5}, it is easy to check that
\begin{align*}
\sup_{t\geq0}g(t)=\frac{s}{3}\frac{\Big(\int_{\mathbb{R}^3}|(-\Delta)^{\frac{s}{2}}u_{\varepsilon}|^2\,{\rm d}x\Big)^{\frac{3}{2s}}}{\Big(Q(x_0)\int_{\mathbb{R}^3}|u_{\varepsilon}|^{2_s^{\ast}}\,{\rm d}x\Big)^{\frac{3-2s}{2s}}}&=\frac{s}{3Q(x_0)^{\frac{3-2s}{2s}}}\frac{\Big(\|(-\Delta)^{\frac{s}{2}}\bar{u}\|_2^2+O(\varepsilon^{3-2s})\Big)^{\frac{3}{2s}}}{\Big(\|\bar{u}\|_{2_s^{\ast}}^{2_s^{\ast}}+O(\varepsilon^3)\Big)^{\frac{3-2s}{2s}}}\\
&\leq\frac{s}{3Q(x_0)^{\frac{3-2s}{2s}}}\Big(\frac{\|(-\Delta)^{\frac{s}{2}}\bar{u}\|_2^2}{\|\bar{u}\|_{2_s^{\ast}}^2}\Big)^{\frac{3}{2s}}+O(\varepsilon^{3-2s})\\
&=\frac{s}{3Q(x_0)^{\frac{3-2s}{2s}}}\mathcal{S}_s^{\frac{3}{2s}}+O(\varepsilon^{3-2s}).
\end{align*}
Thus
\begin{align}\label{equ3-8}
I_{\varepsilon}(t_{\varepsilon}u_{\varepsilon})&\leq \sup_{t\geq0}g(t)+C\int_{\mathbb{R}^3}V(\varepsilon x)|u_{\varepsilon}|^2\, {\rm d}x+C\int_{\mathbb{R}^3}\phi_{u_{\varepsilon}}^tu_{\varepsilon}^2\,{\rm d}x-C\int_{\mathbb{R}^3}|u_{\varepsilon}|^{q+1}\, {\rm d}x\nonumber\\
&+\int_{\mathbb{R}^3}\Big(Q(x_0)-Q(\varepsilon x)\Big)|u_{\varepsilon}|^{2_s^{\ast}}\,{\rm d}x\nonumber\\
&\leq\frac{s}{3Q(x_0)^{\frac{3-2s}{2s}}}\mathcal{S}_s^{\frac{3}{2s}}+O(\varepsilon^{3-2s})+C\int_{\mathbb{R}^3}|u_{\varepsilon}|^2\, {\rm d}x+C\Big(\int_{\mathbb{R}^3}|u_{\varepsilon}|^{\frac{12}{3+2t}}\,{\rm d}x\Big)^{\frac{3+2t}{3}}\nonumber\\
&-C\int_{\mathbb{R}^3}|u_{\varepsilon}|^{q+1}\, {\rm d}x+\int_{\mathbb{R}^3}\Big(Q(x_0)-Q(\varepsilon x)\Big)|u_{\varepsilon}|^{2_s^{\ast}}\,{\rm d}x\nonumber\\
&\leq\frac{s}{3Q(x_0)^{\frac{3-2s}{2s}}}\mathcal{S}_s^{\frac{3}{2s}}+O(\varepsilon^{3-2s})+C\Big(\int_{\mathbb{R}^3}|u_{\varepsilon}|^{\frac{12}{3+2t}}\,{\rm d}x\Big)^{\frac{3+2t}{3}}-C\int_{\mathbb{R}^3}|u_{\varepsilon}|^{q+1}\, {\rm d}x\nonumber\\
&+\int_{\mathbb{R}^3}\Big(Q(x_0)-Q(\varepsilon x)\Big)|u_{\varepsilon}|^{2_s^{\ast}}\,{\rm d}x.
\end{align}
where we have used \eqref{equ3-6} and $s>\frac{3}{4}$ which implies $2<\frac{3}{3-2s}$.

By \eqref{equ3-6}, we have that
\begin{align}\label{equ3-9}
\lim_{\varepsilon\rightarrow0^{+}}\frac{\Big(\int_{\mathbb{R}^3}|u_{\varepsilon}|^{\frac{12}{3+2t}}\,{\rm d}x\Big)^{\frac{3+2t}{3}}}{\varepsilon^{3-2s}}\leq\left\{
                                                                                                                            \begin{array}{ll}
                                                                                                                              \lim\limits_{\varepsilon\rightarrow0^{+}}\frac{O(\varepsilon^{2t+4s-3})}{\varepsilon^{3-2s}}=0, & \hbox{$\frac{12}{3+2t}>\frac{3}{3-2s}$,} \\
                                                                                                                              \lim\limits_{\varepsilon\rightarrow0^{+}}\frac{O(\varepsilon^{2t+4s-3}|\log\varepsilon|)}{\varepsilon^{3-2s}}=0, & \hbox{$\frac{12}{3+2t}=\frac{3}{3-2s}$,} \\
                                                                                                                             \lim\limits_{\varepsilon\rightarrow0^{+}} \frac{O(\varepsilon^{2(3-2s)}|\log\varepsilon|)}{\varepsilon^{3-2s}}=0, & \hbox{$\frac{12}{3+2t}<\frac{3}{3-2s}$.}
                                                                                                                            \end{array}
                                                                                                                          \right.
\end{align}
Since $s>\frac{3}{4}$ and $q>\frac{3+2s}{3-2s}$, then $q+1>\frac{3}{3-2s}$, $2s-\frac{3-2s}{2}(q+1)<0$. Thus
\begin{align}\label{equ3-10}
\lim_{\varepsilon\rightarrow0^{+}}\frac{\int_{\mathbb{R}^3}|u_{\varepsilon}|^{q+1}\,{\rm d}x}{\varepsilon^{3-2s}}=\lim_{\varepsilon\rightarrow0^{+}}\frac{O(\varepsilon^{3-\frac{3-2s}{2}(q+1)})}{\varepsilon^{3-2s}}=+\infty.
\end{align}
By the hypothesis $(Q_1)$, we deduce that
\begin{align*}
\int_{\mathbb{R}^3}&\Big(Q(x_0)-Q(\varepsilon x)\Big)|u_{\varepsilon}|^{2_s^{\ast}}\,{\rm d}x=C\varepsilon^3\Big(\int_{|\varepsilon x-x_0|\leq \rho_0}\frac{|\varepsilon x-x_0|^{\alpha}\psi(x-x_0/\varepsilon)^{2_s^{\ast}}}{(\varepsilon^2+|x-x_0/\varepsilon|^2)^3}\,{\rm d}x\\
&+\int_{|\varepsilon x-x_0|> \rho_0}\frac{\psi(x-x_0/\varepsilon)^{2_s^{\ast}}}{(\varepsilon^2+|x-x_0/\varepsilon|^2)^{\frac{N-2s}{2}}}\,{\rm d}x\Big)\\
&\leq C\varepsilon^3\Big(\varepsilon^{2\alpha-3}\int_0^{\frac{\rho_0}{\varepsilon^2}}\frac{r^{2+\alpha}}{(1+r^2)^3}\,{\rm d}r+\frac{1}{\varepsilon^3}\int_{\frac{\rho_0}{\varepsilon}}^{+\infty}\frac{r^2}{(1+r^2)^3}\,{\rm d}r\Big)\\
&\leq C\varepsilon^3\Big(\varepsilon^{2\alpha-3}\int_0^1\frac{r^{2+\alpha}}{(1+r^2)^3}\,{\rm d}r+\varepsilon^{2\alpha-3}\int_1^{\frac{\rho_0}{\varepsilon^2}}\frac{r^{2+\alpha}}{(1+r^2)^{\frac{7}{4}}}\,{\rm d}r+1\Big)\\
&\leq C(\varepsilon^{2\alpha}+\varepsilon^{3}+\varepsilon^3).
\end{align*}
Since $\frac{3-2s}{2}\leq\alpha<\frac{3}{2}$, then
\begin{equation}\label{equ3-11}
\int_{\mathbb{R}^3}\Big(Q(x_0)-Q(\varepsilon x)\Big)|u_{\varepsilon}|^{2_s^{\ast}}\,{\rm d}x\leq O(\varepsilon^{3-2s}).
\end{equation}
Therefore, combining with \eqref{equ3-8}, \eqref{equ3-9}, \eqref{equ3-10} and \eqref{equ3-11}, we conclude that
\begin{equation*}
I_{\varepsilon}(t_{\varepsilon}u_{\varepsilon})<\frac{s}{3Q(x_0)^{\frac{3-2s}{2s}}}\mathcal{S}_s^{\frac{3}{2s}}
\end{equation*}
for $\varepsilon$ small enough and thus the proof is completed.
\end{proof}

\begin{lemma}\label{lem3-4}
Assume that $(V)$, $(Q_0)$, $(Q_1)$, $(K)$ and $(f_0)-(f_2)$ hold. If $c_{\varepsilon}<\min\{m_{\infty}, \frac{s}{3Q(x_0)^{\frac{3-2s}{2s}}}\mathcal{S}_s^{\frac{3}{2s}}\}$, then $I_{\varepsilon}$ satisfies the $(PS)$ condition for $c_{\varepsilon}$.
\end{lemma}
\begin{proof}
Let $\{u_n\}$ be a $(PS)$ sequence of $I_{\varepsilon}$ at the level $c_{\varepsilon}$, i.e.,
\begin{equation}\label{equ3-12}
I_{\varepsilon}(u_n)\rightarrow c_{\varepsilon},\quad I_{\varepsilon}'(u_n)\rightarrow0 \,\, \text{in}\,\, (H_{\varepsilon})'.
\end{equation}
It is easy to check that $\{u_n\}$ is bounded in $H_{\varepsilon}$. Thus, up to a subsequence, still denoted by $\{u_n\}$, we may assume that there exists $u\in H_{\varepsilon}$ such that $u_n\rightharpoonup u$ in $H_{\varepsilon}$, $u_n\rightarrow u$ in $L_{loc}^r(\mathbb{R}^3)$ for $2\leq r<2_s^{\ast}$.

Next, we aim to show that $u_n\rightarrow u$ in $H_{\varepsilon}$. For this purpose, set
\begin{equation*}
\rho_n(x)=\frac{1}{4}|(-\Delta)^{\frac{s}{2}}u_n|^2+\frac{1}{4}V(\varepsilon x)|u_n|^2+\frac{1}{4}K(\varepsilon x)(f(u_n)u_n-4F(u_n))+\frac{4s-3}{12}Q(\varepsilon x)|u_n|^{2_s^{\ast}}.
\end{equation*}
Clearly, $\{\rho_n\}$ is bounded in $L^1(\mathbb{R}^3)$. Hence, up to a subsequence, still denoted by $\{\rho_n\}$, we may assume that $\Psi(u_n):=\|\rho_n\|_1\rightarrow l$ as $n\rightarrow\infty$. Obviously, $l>0$, otherwise, we can get a contradiction with $c_{\varepsilon}>0$. In fact, $l=c_{\varepsilon}$.

Next, we apply Proposition \ref{pro2-1} to $\{\rho_n\}$.

If $\{\rho_n\}$ vanishing, then $\{u_n^2\}$ also vanishing, i.e., there exists $R>0$ such that
\begin{equation*}
\lim_{n\rightarrow\infty}\sup_{y\in\mathbb{R}^3}\int_{B_R(y)}|u_n|^2\,{\rm d}x=0.
\end{equation*}
By Lemma \ref{lem2-2}, one has $u_n\rightarrow0$ in $L^r(\mathbb{R}^3)$, $2<r<2_s^{\ast}$. Thus, by \eqref{equ2-2} and $(iii)$ of Lemma \ref{lem2-1}, we can deduce that
\begin{equation*}
\int_{\mathbb{R}^3}K(\varepsilon x)F(u_n)\,{\rm d}x\rightarrow0\quad \int_{\mathbb{R}^3}K(\varepsilon x)f(u_n)u_n\,{\rm d}x\rightarrow0,\quad \int_{\mathbb{R}^3}\phi_{u_n}^tu_n^2\,{\rm d}x\rightarrow0.
\end{equation*}
From \eqref{equ3-12}, we get
\begin{equation}\label{equ3-13}
I_{\varepsilon}(u_n)=\frac{1}{2}\|u_n\|_{\varepsilon}^2-\frac{1}{2_s^{\ast}}\int_{\mathbb{R}^3}Q(\varepsilon x)|u_n|^{2_s^{\ast}}\,{\rm d}x+o_n(1)
\end{equation}
and
\begin{equation}\label{equ3-14}
o_n(1)=\langle I_{\varepsilon}'(u_n),u_n\rangle=\|u_n\|_{\varepsilon}^2-\int_{\mathbb{R}^3}Q(\varepsilon x)|u_n|^{2_s^{\ast}}\,{\rm d}x+o_n(1).
\end{equation}
We may assume that there exist $\mathcal{L}\geq0$ such that
\begin{equation*}
\|u_n\|_{\varepsilon}^2\rightarrow \mathcal{L},\quad \int_{\mathbb{R}^3}Q(\varepsilon x)|u_n|^{2_s^{\ast}}\,{\rm d}x\rightarrow \mathcal{L}.
\end{equation*}
Obviously, $\mathcal{L}>0$, otherwise, a contradiction with $c_{\varepsilon}>0$.
By \eqref{equ2-1}, we have that
\begin{align*}
\int_{\mathbb{R}^3}Q(\varepsilon x)|u_n|^{2_s^{\ast}}\,{\rm d}x\leq Q_{x_0}\int_{\mathbb{R}^3}|u_n|^{2_s^{\ast}}\,{\rm d}x&\leq Q(x_0)\Big(\mathcal{S}_s^{-1}\int_{\mathbb{R}^3}|(-\Delta)^{\frac{s}{2}}u_n|^2\,{\rm d}x\Big)^{\frac{2_s^{\ast}}{2}}\\
&\leq Q(x_0)\mathcal{S}_s^{-\frac{3}{3-2s}}\|u_n\|_{\varepsilon}^{2_s^{\ast}}
\end{align*}
which implies that
\begin{equation*}
\mathcal{L}\geq\frac{1}{Q(x_0)^{\frac{3-2s}{2s}}}\mathcal{S}_s^{\frac{3}{2s}}.
\end{equation*}
Combing \eqref{equ3-13}, we can deduce that $c_{\varepsilon}=\frac{s}{3}\mathcal{L}\geq\frac{s}{3Q(x_0)^{\frac{3-2s}{2s}}}\mathcal{S}_s^{\frac{3}{2s}}$, this contradicts with the assumption. Hence, vanishing does not occur.

Next, we show the dichotomy does not occur. Suppose by contradiction that there exist $\alpha\in(0,l)$ and $\{y_n\}\subset \mathbb{R}^3$ such that for every $\varepsilon_n\rightarrow0$, we can choose $\{R_n\}\subset\mathbb{R}_{+}$ ($R_n>R_0/\varepsilon+R'$, for any fixed $\varepsilon>0$, $R_0,R'$ are positive constants defined later) with $R_n\rightarrow\infty$ satisfying
\begin{equation}\label{equ3-15}
\limsup_{n\rightarrow\infty}\Big|\alpha-\int_{B_{R_n}(y_n)}\rho_n(x)\,{\rm d}x\Big|+\Big|(l-\alpha)-\int_{\mathbb{R}^3\backslash B_{R_n}(y_n)}\rho_n(x)\,{\rm d}x\Big|<\varepsilon_n.
\end{equation}
Let $\xi:\mathbb{R}^{+}\cup\{0\}\rightarrow\mathbb{R}^{+}$ be a cut-off function such that $0\leq\xi\leq1$, $\xi(t)=1$ for $t\leq1$, $\xi(t)=0$ for $t\geq2$ and $|\xi'(t)|\leq2$. Set
\begin{equation*}
v_n(x)=\xi\Big(\frac{|x-y_n|}{R_n}\Big)u_n(x),\quad w_n(x)=\Big(1-\xi\Big(\frac{|x-y_n|}{R_n}\Big)\Big)u_n(x),
\end{equation*}
Then by \eqref{equ3-15}, we see that
\begin{align*}
\liminf_{n\rightarrow\infty}\Psi(v_n)\geq\alpha,\quad \liminf_{n\rightarrow\infty}\Psi(w_n)\geq l-\alpha.
\end{align*}

Denote $\Omega_n=B_{2R_n}(y_n)\backslash B_{R_n}(y_n)$, by \eqref{equ3-15}, then $\int_{\Omega_n}\rho_n(x)\,{\rm d}x\rightarrow0$ as $n\rightarrow\infty$, which leads to
\begin{equation*}
\int_{\Omega_n}\Big(|(-\Delta)^{\frac{s}{2}}u_n|^2+V(\varepsilon x)|u_n|^2\Big)\,{\rm d}x\rightarrow0,\quad \int_{\Omega_n}|u_n|^2\,{\rm d}x\rightarrow0,\quad \int_{\Omega_n}|u_n|^{2_s^{\ast}}\,{\rm d}x\rightarrow0
\end{equation*}
and thus by interpolation inequality, we can infer that $\int_{\Omega_n}\phi_{u_n}^tu_n^2\,{\rm d}x\rightarrow0$. Therefore, by similar arguments as Lemma 3.4 in \cite{Teng3}, we have that
\begin{equation}\label{equ3-16}
\int_{\mathbb{R}^3}|(-\Delta)^{\frac{s}{2}}u_n|^2\,{\rm d}x=\int_{\mathbb{R}^3}|(-\Delta)^{\frac{s}{2}}v_n|^2\,{\rm d}x+\int_{\mathbb{R}^3}|(-\Delta)^{\frac{s}{2}}w_n|^2\,{\rm d}x+o_n(1),
\end{equation}
\begin{equation}\label{equ3-17}
\int_{\mathbb{R}^3}V(\varepsilon x)u_n^2(x)\,{\rm d}x=\int_{\mathbb{R}^3}V(\varepsilon x)v_n^2(x)\,{\rm d}x+\int_{\mathbb{R}^3}V(\varepsilon x)w_n^2(x)\,{\rm d}x+o_n(1),
\end{equation}
\begin{equation}\label{equ3-18}
\int_{\mathbb{R}^3}K(\varepsilon x)F(u_n)\,{\rm d}x=\int_{\mathbb{R}^3}K(\varepsilon x)F(v_n)\,{\rm d}x+\int_{\mathbb{R}^3}K(\varepsilon x)F(w_n)\,{\rm d}x+o_n(1),
\end{equation}
\begin{equation}\label{equ3-19}
\int_{\mathbb{R}^3}K(\varepsilon x)f(u_n)u_n\,{\rm d}x=\int_{\mathbb{R}^3}K(\varepsilon x)f(v_n)v_n\,{\rm d}x+\int_{\mathbb{R}^3}K(\varepsilon x)f(w_n)w_n\,{\rm d}x+o_n(1),
\end{equation}
\begin{equation}\label{equ3-20}
\int_{\mathbb{R}^3}Q(\varepsilon x)|u_n|^{2_s^{\ast}}\,{\rm d}x=\int_{\mathbb{R}^3}Q(\varepsilon x)|v_n|^{2_s^{\ast}}\,{\rm d}x+\int_{\mathbb{R}^3}Q(\varepsilon x)|w_n|^{2_s^{\ast}}\,{\rm d}x+o_n(1),
\end{equation}
\begin{equation}\label{equ3-21}
\int_{\mathbb{R}^3}\phi_{u_n}^tu_n^2\,{\rm d}x\geq\int_{\mathbb{R}^3}\phi_{v_n}^tv_n^2\,{\rm d}x+\int_{\mathbb{R}^3}\phi_{w_n}^tw_n^2\,{\rm d}x+o_n(1).
\end{equation}

We need to check that \eqref{equ3-18} and \eqref{equ3-19} hold. Indeed, by $(K)$, \eqref{equ2-2} and H\"{o}lder's inequality, we infer that
\begin{align*}
&\Big|\int_{\mathbb{R}^3}K(\varepsilon x)\Big(F(u_n)-F(v_n)-F(w_n)\Big)\,{\rm d}x\Big|\\
&=\Big|\int_{\Omega_n}K(\varepsilon x)\Big(F(u_n)-F(\xi u_n)-F((1-\xi)u_n)\Big)\,{\rm d}x\Big|\\
&\leq C\Big(\varepsilon\int_{\Omega_n}|u_n|^4\,{\rm d}x+\int_{\Omega_n}|u_n|^p\,{\rm d}x\Big)\\
&\leq C\Big[\varepsilon\Big(\int_{\Omega_n}|u_n|^2\,{\rm d}x\Big)^{2\theta_1}\Big(\int_{\Omega_n}|u_n|^{2_s^{\ast}}\,{\rm d}x\Big)^{\frac{4(1-\theta)}{2_s^{\ast}}}+\Big(\int_{\Omega_n}|u_n|^2\,{\rm d}x\Big)^{\frac{p\theta_2}{2}}\Big(\int_{\Omega_n}|u_n|^{2_s^{\ast}}\,{\rm d}x\Big)^{\frac{p(1-\theta_2)}{2_s^{\ast}}}\Big]\\
&\rightarrow0
\end{align*}
as $n\rightarrow\infty$, where $\theta_1\in(0,1)$ and $\theta_2\in(0,1)$ satisfy $\frac{1}{4}=\frac{\theta_1}{2}+\frac{1-\theta_1}{2_s^{\ast}}$, $\frac{1}{p}=\frac{\theta_2}{2}+\frac{1-\theta_2}{2_s^{\ast}}$. Similarly, we can verify that \eqref{equ3-19} hold.

Hence, by \eqref{equ3-16}--\eqref{equ3-21}, we get
\begin{equation*}
\Psi(u_n)\geq\Psi(v_n)+\Psi(w_n)+o_n(1).
\end{equation*}
Then
\begin{equation*}
l=\lim_{n\rightarrow\infty}\Psi(u_n)\geq\liminf_{n\rightarrow\infty} \Psi(v_n)+\liminf_{n\rightarrow\infty} \Psi(w_n)\geq\alpha+l-\alpha=l,
\end{equation*}
hence
\begin{equation}\label{equ3-22}
\lim_{n\rightarrow\infty}\Psi(v_n)=\alpha,\quad \lim_{n\rightarrow\infty}\Psi(w_n)=l-\alpha.
\end{equation}
By \eqref{equ3-12}, \eqref{equ3-16}-\eqref{equ3-21}, we have
\begin{equation}\label{equ3-23}
o_n(1)=\langle I_{\varepsilon}'(u_n),u_n\rangle\geq\langle I_{\varepsilon}'(v_n),v_n\rangle+\langle I_{\varepsilon}'(w_n),w_n\rangle+o_n(1).
\end{equation}
We distinguish the following two cases:\\
Case 1. Up to a subsequence, we may assume that $\langle I_{\varepsilon}'(v_n),v_n\rangle\leq0$ or $\langle I_{\varepsilon}'(w_n),w_n\rangle\leq0$.

Without loss of generality, we suppose that $\langle I_{\varepsilon}'(v_n),v_n\rangle\leq0$, then
\begin{align}\label{equ3-24}
\|v_n\|_{\varepsilon}^2+\int_{\mathbb{R}^3}\phi_{v_n}^tv_n^2\,{\rm d}x-\int_{\mathbb{R}^3}K(\varepsilon x)f(v_n)v_n\,{\rm d}x-\int_{\mathbb{R}^3}Q(\varepsilon x)|v_n|^{2_s^{\ast}}\,{\rm d}x\leq0.
\end{align}
By $(ii)$ of Lemma \ref{lem3-1}, for any $n$, there exists $t_n>0$ such that $t_nv_n\in\mathcal{N}_{\varepsilon}$ and then $\langle I_{\varepsilon}'(t_nv_n),t_nv_n\rangle=0$, i.e.,
\begin{align}\label{equ3-25}
t_n^2\|v_n\|_{\varepsilon}^2+t_n^4\int_{\mathbb{R}^3}\phi_{v_n}^tv_n^2\,{\rm d}x-\int_{\mathbb{R}^3}K(\varepsilon x)f(t_nv_n)t_nv_n\,{\rm d}x-t_n^{2_s^{\ast}}\int_{\mathbb{R}^3}Q(\varepsilon x)|v_n|^{2_s^{\ast}}\,{\rm d}x=0.
\end{align}
Combining \eqref{equ3-24} and \eqref{equ3-25}, we have
\begin{equation*}
(\frac{1}{t_n^2}-1)\|v_n\|_{\varepsilon}^2-\int_{\mathbb{R}^3}K(\varepsilon x)\Big(\frac{f(t_nv_n)}{(t_nv_n)^3}-\frac{f(v_n)}{v_n^3}\Big)v_n^4\,{\rm d}x-(t_n^{2_s^{{\ast}-4}}-1)\int_{\mathbb{R}^3}Q(\varepsilon x)|v_n|^{2_s^{\ast}}\,{\rm d}x\geq0
\end{equation*}
which implies that $t_n\leq1$ by using $(f_1)$. Then, by $t_nv_n\in\mathcal{N}_{\varepsilon}$ and \eqref{equ2-3} (implies that $f(su)su-4F(su)$ is nondecreasing in $s\in(0,+\infty)$), we have that
\begin{align*}
c_{\varepsilon}\leq I_{\varepsilon}(t_n v_n)&=I_{\varepsilon}(t_nv_n)-\frac{1}{4}\langle I_{\varepsilon}'(t_nv_n),t_n v_n\rangle\\
&=\frac{1}{4}t_n^2\|v_n\|_{\varepsilon}^2+\frac{1}{4}\int_{\mathbb{R}^3}K(\varepsilon x)\Big(f(t_nv_n)t_nv_n-4F(t_nv_n)\Big)\,{\rm d}x\\
&+\frac{4s-3}{12}t_n^{2_s^{\ast}}\int_{\mathbb{R}^3}Q(\varepsilon x)|v_n|^{2_s^{\ast}}\,{\rm d}x\\
&\leq \Psi(v_n)\rightarrow\alpha <l=c_{\varepsilon}
\end{align*}
which is a contradiction.

Case 2. Up to a subsequence, we may assume that $\langle I_{\varepsilon}'(v_n),v_n\rangle>0$ and $\langle I_{\varepsilon}'(w_n),w_n\rangle>0$.

By \eqref{equ3-23}, we see that $\langle I_{\varepsilon}'(v_n),v_n\rangle=o_n(1)$ and $\langle I_{\varepsilon}'(w_n),w_n\rangle=o_n(1)$. In view of \eqref{equ3-16}--\eqref{equ3-21}, we have that
\begin{equation}\label{equ3-26}
I_{\varepsilon}(u_n)\geq I_{\varepsilon}(v_n)+I_{\varepsilon}(w_n)+o_n(1).
\end{equation}
If the sequence $\{y_n\}\subset\mathbb{R}^3$ is bounded, we will deduce a contradiction by comparing $I_{\varepsilon}(w_n)$ and $m_{\infty}$. In fact, by the assumptions $(V)$, $(K)$ and $(Q)$, for any $\delta>0$, there exists $R_0>0$, such that
\begin{equation}\label{equ3-27}
|K(\varepsilon x)-K_{\infty}|\leq \delta,\quad  |Q(\varepsilon x)-Q_{\infty}|\leq \delta,\quad V(\varepsilon x)-V_{\infty}>-\delta, \quad \forall |x|\geq R_0/\varepsilon.
\end{equation}
By the boundedness of $\{y_n\}\subset\mathbb{R}^3$, there exists $R'>0$ such that $|y_n|\leq R'$. Thus, $\mathbb{R}^3\backslash B_{R_n}(y_n)\subset\mathbb{R}^3\backslash B_{R_n-R'}(0)\subset \mathbb{R}^3\backslash B_{R_0/\varepsilon}(0)$ for $n$ large enough. From \eqref{equ3-27}, we can deduce that
\begin{align*}
\int_{\mathbb{R}^3}\Big(V(\varepsilon x)-V_{\infty}\Big)|w_n|^2\,{\rm d}x&=\int_{|x-y_n|>R_n}\Big(V(\varepsilon x)-V_{\infty}\Big)|w_n|^2\,{\rm d}x\\
&\geq-\delta\int_{|x-y_n|>R_n}|w_n|^2\,{\rm d}x\geq-\delta C,
\end{align*}
and by the arbitrariness of $\delta$, this leads to
\begin{equation}\label{equ3-28}
\int_{\mathbb{R}^3}\Big(V(\varepsilon x)-V_{\infty}\Big)|w_n|^2\,{\rm d}x\geq o_n(1).
\end{equation}
Moreover, it is easy to check that
\begin{equation}\label{equ3-29}
\int_{\mathbb{R}^3}\Big(K(\varepsilon x)-K_{\infty}\Big)F(w_n)\,{\rm d}x=o(1),\quad \int_{\mathbb{R}^3}\Big(K(\varepsilon x)-K_{\infty}\Big)f(w_n)w_n\,{\rm d}x=o_n(1)
\end{equation}
and
\begin{equation}\label{equ3-30}
\int_{\mathbb{R}^3}\Big(Q(\varepsilon x)-Q_{\infty}\Big)|w_n|^{2_s^{\ast}}\,{\rm d}x=o_n(1).
\end{equation}
It follows from \eqref{equ3-28}--\eqref{equ3-30} that
\begin{equation}\label{equ3-31}
I_{\varepsilon}(w_n)\geq I_{\infty}(w_n)+o_n(1)\quad \text{and}\quad o_n(1)=\langle I_{\varepsilon}'(w_n),w_n\rangle\geq\langle I_{\infty}'(w_n),w_n\rangle+o_n(1).
\end{equation}
If $\langle I_{\infty}'(w_n),w_n\rangle\leq0$ for $n$ large enough, similar to the proof of Case 1, we get that there exists $t_n\leq1$ such that $t_nw_n\in \mathcal{N}_{\infty}$. Thus, by \eqref{equ3-28}--\eqref{equ3-30}
\begin{align*}
m_{\infty}\leq I_{\infty}(t_nw_n)&=I_{\infty}(t_nw_n)-\frac{1}{4}\langle I_{\infty}'(t_nw_n),t_n w_n\rangle\\
&=\frac{1}{4}t_n^2\|w_n\|^2+\frac{1}{4}\int_{\mathbb{R}^3}K_{\infty}\Big(f(t_nw_n)t_nw_n-4F(t_nw_n)\Big)\,{\rm d}x\\
&+\frac{4s-3}{12}t_n^{2_s^{\ast}}\int_{\mathbb{R}^3}Q_{\infty}|w_n|^{2_s^{\ast}}\,{\rm d}x\\
&\leq\frac{1}{4}\|w_n\|_{\varepsilon}^2+\frac{1}{4}\int_{\mathbb{R}^3}K(\varepsilon x)\Big(f(w_n)w_n-4F(w_n)\Big)\,{\rm d}x\\
&+\frac{4s-3}{12}\int_{\mathbb{R}^3}Q(\varepsilon x)|w_n|^{2_s^{\ast}}\,{\rm d}x+o_n(1)\\
&=\Psi(w_n)+o_n(1)\rightarrow l-\alpha=c_{\varepsilon}-\alpha<c_{\varepsilon}
\end{align*}
which contradicts with the assumption $c_{\varepsilon}<m_{\infty}$.

Observing that $\langle I_{\infty}'(w_n),w_n\rangle\rightarrow0$ and $\langle I_{\varepsilon}'(v_n),v_n\rangle\rightarrow0$ as $n\rightarrow\infty$ (indeed, $\int_{\mathbb{R}^3}\Big(K_{\infty}f(w_n)w_n+Q_{\infty}|w_n|^{2_s^{\ast}}\Big)\,{\rm d}x\rightarrow A$, $\int_{\mathbb{R}^3}\Big(K(\varepsilon x)f(v_n)v_n+Q(\varepsilon x)|v_n|^{2_s^{\ast}}\Big)\,{\rm d}x\rightarrow B$, where $A,B>0$, otherwise, contradicts with \eqref{equ3-22}), by $(v)$ of Lemma \ref{lem3-1}, there exist two sequences $\{t_n\}\subset\mathbb{R}_{+}$ and $\{s_n\}\subset\mathbb{R}_{+}$ satisfying $t_n\rightarrow1$ and $s_n\rightarrow1$ as $n\rightarrow\infty$, respectively, such that $t_nw_n\in\mathcal{N}_{\infty}$, $s_nv_n\in \mathcal{N}_{\varepsilon}$. Hence, by \eqref{equ3-31}, we get
\begin{equation*}
I_{\varepsilon}(w_n)\geq I_{\infty}(w_n)+o_n(1)=I_{\infty}(t_nw_n)+o_n(1)\geq m_{\infty}+o_n(1)
\end{equation*}
and
\begin{equation*}
I_{\varepsilon}(v_n)=I_{\varepsilon}(s_nv_n)+o_n(1)\geq m_{\varepsilon}+o_n(1)=c_{\infty}+o_n(1).
\end{equation*}
Therefore, by \eqref{equ3-26}, we have $c_{\varepsilon}\geq m_{\infty}+c_{\varepsilon}>m_{\infty}$, a contradiction.

If $\{y_n\}\subset\mathbb{R}^3$ is unbounded, we choose a subsequence, stilled denoted by $\{y_n\}$, such that $|y_n|\geq 2R_n$. Then $B_{2R_n}(y_n)\subset \mathbb{R}^3\backslash B_{R_n}(0)\subset\mathbb{R}^3\backslash B_{R_0/\varepsilon}(0)$. Similarly to the proof \eqref{equ3-28}--\eqref{equ3-30}, we can infer that
\begin{equation*}
\int_{\mathbb{R}^3}\Big(V(\varepsilon x)-V_{\infty}\Big)|v_n|^2\,{\rm d}x\geq o_n(1)\quad \int_{\mathbb{R}^3}\Big(Q(\varepsilon x)-Q_{\infty}\Big)|v_n|^{2_s^{\ast}}\,{\rm d}x=o_n(1)
\end{equation*}
and
\begin{equation*}
\int_{\mathbb{R}^3}\Big(K(\varepsilon x)-K_{\infty}\Big)F(v_n)\,{\rm d}x=o(1),\quad \int_{\mathbb{R}^3}\Big(K(\varepsilon x)-K_{\infty}\Big)f(v_n)v_n\,{\rm d}x=o_n(1)
\end{equation*}
Similarly as to the case that $\{y_n\}$ is bounded, we can obtain a contradiction by comparing $I_{\varepsilon}(v_n)$ with $m_{\infty}$. Thus, dichotomy does not happen.

According to the above arguments, the sequence $\{\rho_n\}$ must be compactness, i.e., there exists $\{y_n\}\subset\mathbb{R}^3$ such that for every $\hat{\delta}>0$, there exists $\widetilde{R}>0$, we have $\int_{\mathbb{R}^3\backslash B_{\widetilde{R}}(y_n)}\rho_n(x)\,{\rm d}x<\hat{\delta}$. By the interpolation inequality, we have that
\begin{equation}\label{equ3-32}
\int_{\mathbb{R}^3\backslash B_{\widetilde{R}}(y_n)}|u_n|^m\,{\rm d}x\leq\Big(\int_{\mathbb{R}^3\backslash B_{\widetilde{R}}(y_n)}|u_n|^2\,{\rm d}x\Big)^{\frac{m\theta_3}{2}}\Big(\int_{\mathbb{R}^3\backslash B_{\widetilde{R}}(y_n)}|u_n|^{2_s^{\ast}}\,{\rm d}x\Big)^{\frac{m(1-\theta_3)}{2_s^{\ast}}}<C\hat{\delta}
\end{equation}
where $m\in[2,2_s^{\ast}]$, $\theta_3\in[0,1]$ satisfies $\frac{1}{m}=\frac{\theta_3}{2}+\frac{1-\theta_3}{2_s^{\ast}}$. This means that the sequence $\{|u_n|^m\}$ with $2\leq m\leq 2_s^{\ast}$ is also compactness.

We claim that the sequence $\{y_n\}$ is bounded. If not, up to a subsequence, we can choose $r_n$ such that $|y_n|\geq r_n\geq \widetilde{R}+R_0/\varepsilon$ with $r_n\rightarrow+\infty$. For $n$ large enough, $B_{\widetilde{R}}(y_n)\subset \mathbb{R}^3\backslash B_{r_n-\widetilde{R}}(0)\subset\mathbb{R}^3\backslash B_{R_0/\varepsilon}(0)$.   By \eqref{equ3-32}, we can infer that
\begin{align*}
\int_{\mathbb{R}^3}\Big(V(\varepsilon x)-V_{\infty}\Big)|u_n|^2\,{\rm d}x&=\int_{B_{\widetilde{R}}(y_n)}+\int_{\mathbb{R}^3\backslash B_{\widetilde{R}}(y_n)}\Big(V(\varepsilon x)-V_{\infty}\Big)|u_n|^2\,{\rm d}x\\
&\geq-\delta C+o_n(1)\geq o_n(1)
\end{align*}
Similarly, we can obtain that
\begin{equation*}
\int_{\mathbb{R}^3}\Big(K(\varepsilon x)-K_{\infty}\Big)F(u_n)\,{\rm d}x=o_n(1),\quad \int_{\mathbb{R}^3}\Big(K(\varepsilon x)-K_{\infty}\Big)f(u_n)u_n\,{\rm d}x=o_n(1),
\end{equation*}
and
\begin{equation*}
\int_{\mathbb{R}^3}\Big(Q(\varepsilon x)-Q_{\infty}\Big)|u_n|^{2_s^{\ast}}\,{\rm d}x=o_n(1).
\end{equation*}
Thus, $I_{\varepsilon}(u_n)\geq I_{\infty}(u_n)+o_n(1)$ and $o_n(1)=\langle I_{\varepsilon}'(u_n),u_n\rangle\geq\langle I_{\infty}'(u_n),u_n\rangle +o_n(1)$.

If $\langle I_{\infty}'(u_n),u_n\rangle\leq0$, by the similar arguments as Case 1, there exists $t_n\leq1$ for sufficiently large $n$ such that $t_nu_n\in N_{\infty}$. Hence,
\begin{align*}
m_{\infty}\leq I_{\infty}(t_nu_n)&=I_{\infty}(t_nu_n)-\frac{1}{4}\langle I_{\infty}'(t_nu_n),t_n u_n\rangle\\
&=\frac{1}{4}t_n^2\|u_n\|^2+\frac{1}{4}\int_{\mathbb{R}^3}K_{\infty}\Big(f(t_nu_n)t_nu_n-4F(t_nu_n)\Big)\,{\rm d}x\\
&+\frac{4s-3}{12}t_n^{2_s^{\ast}}\int_{\mathbb{R}^3}Q_{\infty}|u_n|^{2_s^{\ast}}\,{\rm d}x\\
&\leq\frac{1}{4}\|u_n\|_{\varepsilon}^2+\frac{1}{4}\int_{\mathbb{R}^3}K(\varepsilon x)\Big(f(u_n)w_n-4F(u_n)\Big)\,{\rm d}x\\
&+\frac{4s-3}{12}\int_{\mathbb{R}^3}Q(\varepsilon x)|u_n|^{2_s^{\ast}}\,{\rm d}x+o_n(1)\\
&=I_{\varepsilon}(u_n)-\frac{1}{4}\langle I_{\varepsilon}'(u_n),u_n\rangle+o_n(1)\rightarrow c_{\varepsilon}
\end{align*}
which is a contradiction.

If $\langle I_{\infty}'(u_n),u_n\rangle=o_n(1)$, by $(v)$ of Lemma \ref{lem3-1}, there exists $t_n\rightarrow 1$ such that $t_nu_n\in\mathcal{N}_{\infty}$. Hence, for $n$ large enough, we have that
\begin{equation*}
c_{\varepsilon}=I_{\varepsilon}(u_n)+o_n(1)\geq I_{\infty}(u_n)+o_n(1)=I_{\infty}(t_nu_n)+o_n(1)\geq m_{\infty}+o_n(1)\geq m_{\infty}
\end{equation*}
which is a contradiction. Therefore, the claim is true.

In view of the boundedness of $\{y_n\}$ and $u_n\rightarrow u$ in $L_{loc}^r(\mathbb{R}^3)$ for $2\leq r<2_s^{\ast}$, using \eqref{equ3-32}, it is easy to check that $u_n\rightarrow u$ in $L^r(\mathbb{R}^3)$ for $2\leq r<2_s^{\ast}$. Set $\widetilde{u_n}=u_n-u$, by the weakness convergence and Brezis-Lieb Lemma, one has
\begin{equation}\label{equ3-33}
\|u_n\|_{\varepsilon}^2=\|u\|_{\varepsilon}^2+\|\widetilde{u_n}\|_{\varepsilon}^2+o_n(1)
\end{equation}
and
\begin{equation}\label{equ3-34}
\int_{\mathbb{R}^3}Q(\varepsilon x)|u_n|^{2_s^{\ast}}\,{\rm d}x=\int_{\mathbb{R}^3}Q(\varepsilon x)|u|^{2_s^{\ast}}\,{\rm d}x+\int_{\mathbb{R}^3}Q(\varepsilon x) |\widetilde{u_n}|^{2_s^{\ast}}\,{\rm d}x+o_n(1).
\end{equation}

By H\"{o}lder's inequality and using $u_n\rightarrow u$ in $L^r(\mathbb{R}^3)$ for $2\leq r<2_s^{\ast}$, it is easy to verify that
\begin{equation}\label{equ3-35}
\int_{\mathbb{R}^3}K(\varepsilon x)F(u_n)\,{\rm d}x=\int_{\mathbb{R}^3}K(\varepsilon x)F(u)\,{\rm d}x+o_n(1)
\end{equation}
\begin{equation}\label{equ3-36}
\int_{\mathbb{R}^3}K(\varepsilon x)f(u_n)u_n\,{\rm d}x=\int_{\mathbb{R}^3}K(\varepsilon x)f(u)u\,{\rm d}x+o_n(1)
\end{equation}
\begin{equation}\label{equ3-37}
\int_{\mathbb{R}^3}K(\varepsilon x)f(u_n)v\,{\rm d}x=\int_{\mathbb{R}^3}K(\varepsilon x)f(u)v\,{\rm d}x+o_n(1)
\end{equation}
and
\begin{equation}\label{equ3-38}
\int_{\mathbb{R}^3}Q(\varepsilon x)|u_n|^{2_s^{\ast}-2}u_nv\,{\rm d}x=\int_{\mathbb{R}^3}Q(\varepsilon x)|u|^{2_s^{\ast}-2}uv\,{\rm d}x+o_n(1).
\end{equation}
Therefore, by $(iv)$ of Lemma \ref{lem2-1} and \eqref{equ3-37}--\eqref{equ3-38}, it is easy to see that $\langle I_{\varepsilon}'(u_n),v\rangle\rightarrow\langle I_{\varepsilon}'(u),v\rangle$ for any $v\in H_{\varepsilon}$, i.e., $I_{\varepsilon}'(u)=0$ and then $u\in\mathcal{N}_{\varepsilon}$, so that $I_{\varepsilon}(u)\geq0$. Consequently, by $(iv)$ of Lemma \ref{lem2-1} and \eqref{equ3-33}--\eqref{equ3-38}, we have that
\begin{equation*}
c_{\varepsilon}+o_n(1)=I_{\varepsilon}(u_n)\geq I_{\varepsilon}(u_n)-I_{\varepsilon}(u)=\frac{1}{2}\|\widetilde{u_n}\|_{\varepsilon}^2-\frac{1}{2_s^{\ast}}\int_{\mathbb{R}^3}Q(\varepsilon x) |\widetilde{u_n}|^{2_s^{\ast}}\,{\rm d}x+o_n(1)
\end{equation*}
and
\begin{align*}
o_n(1)=\langle I_{\varepsilon}'(u_n),u_n\rangle-\langle I_{\varepsilon}'(u),u\rangle=\|\widetilde{u_n}\|_{\varepsilon}^2-\int_{\mathbb{R}^3}Q(\varepsilon x) |\widetilde{u_n}|^{2_s^{\ast}}\,{\rm d}x+o_n(1)
\end{align*}
By similar arguments as the proof of vanishing (see \eqref{equ3-13}, \eqref{equ3-14}), it is easy to get a contradiction with the limit $\|\widetilde{u_n}\|_{\varepsilon}^2\rightarrow \mathcal{L}>0$. Thus, $\mathcal{L}=0$ and hence $u_n\rightarrow u$ in $H_{\varepsilon}$.
\end{proof}

\section{Relation between $c_{\varepsilon}$ and $m_0$, $m_{\infty}$}

In this section, we shall compare the energy level $c_{\varepsilon}$ of problem \eqref{equ2-2} with the energy level $m_0$ and $m_{\infty}$ of limit equation \eqref{equ2-5}. For this purpose, we should prove the existence of positive ground state solutions for the autonomous problem \eqref{equ2-5}. For reader's convenience, we rewrite it as follows
\begin{equation}\label{equ4-1}
(-\Delta)^su+\nu u+\phi_u^t u=\kappa f(u)+\mu|u|^{2_s^{\ast}-2}u, \quad \text{in}\,\,\mathbb{R}^3
\end{equation}
where $\mu,\nu,\kappa>0$ are arbitrary positive constants and $\phi_u^t=\int_{\mathbb{R}^3}\frac{u^2(y)}{|x-y|^{3-2t}}\,{\rm d}x$.

Similar to Section 3, we can easily prove that the functional $I_{\nu,\kappa,\mu}$ verifies the mountain pass geometry. Then, applying Theorem 1.15 in \cite{Willem}, there exists a $(PS)_{c_{\nu,\kappa,\mu}}$ sequence $\{u_n\}\subset E_{\nu}$ such that $I_{\nu,\kappa,\mu}(u_n)\rightarrow c_{\nu,\kappa,\mu}$ and $I'_{\nu,\kappa,\mu}(u_n)\rightarrow0$ as $n\rightarrow\infty$, where $c_{\nu,\kappa,\mu}$ can be characterized by the following relations
\begin{align*}
c_{\nu,\kappa,\mu}=\inf_{\gamma\in \Gamma_{\nu,\kappa,\mu}}\max_{t\in[0,1]}I_{\nu,\kappa,\mu}(\gamma(t))=\inf_{u\in\mathcal{N}_{\nu,\kappa,\mu}}I_{\nu,\kappa,\mu}(u)=\inf_{u\in E_{\nu}\backslash\{0\}}\max_{t\geq0}I_{\nu,\kappa,\mu}(tu),
\end{align*}
where $\Gamma_{\nu,\kappa,\mu}=\{\gamma\in C([0,1], E_{\nu})\,\,|\,\, \gamma(0)=0, I_{\nu,\kappa,\mu}(\gamma(1))<0\}$.
For obtaining a compactness of the above sequence, we give the following estimate for $c_{\nu,\kappa,\mu}$.
\begin{lemma}\label{lem4-1}
For any $\nu,\kappa,\mu>0$, the infinimum $c_{\nu,\kappa,\mu}$ satisfies
\begin{equation*}
0<c_{\nu,\kappa,\mu}<\frac{s}{3\mu^{\frac{3-2s}{2s}}}\mathcal{S}_s^{\frac{3}{2s}},
\end{equation*}
where $\mathcal{S}_s$ is the best Sobolev constant defined by \eqref{equ2-1}.
\end{lemma}
\begin{proof}
We define
\begin{equation*}
u_{\varepsilon}(x)=\psi(x)U_{\varepsilon}(x),\quad x\in\mathbb{R}^3,
\end{equation*}
where $U_{\varepsilon}(x)=\varepsilon^{-\frac{3-2s}{2}}u^{\ast}(x/\varepsilon)$, $u^{\ast}(x)=\frac{\widetilde{u}(x/\mathcal{S}_s^{\frac{1}{2s}})}{\|\widetilde{u}\|_{2_s^{\ast}}}$, $\kappa\in\mathbb{R}\backslash\{0\}$, $\mu>0$ and $x_0\in\mathbb{R}^3$ are fixed constants, $\widetilde{u}(x)=\kappa(\mu^2+|x-x_0|^2)^{-\frac{3-2s}{2}}$, and $\psi\in C^{\infty}(\mathbb{R}^3)$ such that $0\leq\psi\leq1$ in $\mathbb{R}^3$, $\psi(x)\equiv1$ in $B_{\delta}$ and $\psi\equiv0$ in $\mathbb{R}^3\backslash B_{2\delta}$. From Proposition 21 and Proposition 22 in \cite{SV}, Lemma 3.3 in \cite{Teng2}, we know that
\begin{equation}\label{equ4-2}
\int_{\mathbb{R}^{3}}|(-\Delta)^{\frac{s}{2}}u_{\varepsilon}(x)|^2\,{\rm d}x\leq\mathcal{S}_s^{\frac{3}{2s}}+O(\varepsilon^{3-2s}),
\end{equation}
\begin{equation}\label{equ4-3}
\int_{\mathbb{R}^{3}}|u_{\varepsilon}(x)|^{2_s^{\ast}}\,{\rm d}x=\mathcal{S}_s^{\frac{3}{2s}}+O(\varepsilon^3),
\end{equation}
and
\begin{equation}\label{equ4-4}
\int_{\mathbb{R}^3}|u_{\varepsilon}(x)|^p\,{\rm d}x=\left\{
\begin{array}{ll}
O(\varepsilon^{\frac{(2-p)3+2sp}{2}}),&\hbox{$p>\frac{3}{3-2s}$,} \\
O(\varepsilon^{\frac{(2-p)3+2sp}{2}}|\log\varepsilon|), & \hbox{$p=\frac{3}{3-2s}$,} \\
O(\varepsilon^{\frac{3-2s}{2}p}), & \hbox{$1< p<\frac{3}{3-2s}$.}
\end{array}
\right.
\end{equation}

Similar to the proof of $(ii)$ of Lemma \ref{lem3-1}, there exists $t_{\varepsilon}>0$ such that $\sup\limits_{t\geq0}I_{\nu,\kappa,\mu}(tu_{\varepsilon})=I_{\nu,\kappa,\mu}(t_{\varepsilon}u_{\varepsilon})$. Hence $\frac{{\rm d}I_{\nu,\kappa,\mu}(tu_{\varepsilon})}{{\rm d} t}\Big|_{t=t_{\varepsilon}}=0$, that is
\begin{align*}
t_{\varepsilon}^2\int_{\mathbb{R}^3}(|(-\Delta)^{\frac{s}{2}}u_{\varepsilon}|^2+\nu u_{\varepsilon}^2)\,{\rm d}x+t_{\varepsilon}^4\int_{\mathbb{R}^3}\phi_{u_{\varepsilon}}^tu_{\varepsilon}^2\,{\rm d}x=&\kappa\int_{\mathbb{R}^3} f(t_{\varepsilon}u_{\varepsilon})t_{\varepsilon}u_{\varepsilon}\,{\rm d}x\\
&+\mu t_{\varepsilon}^{2_s^{\ast}}\int_{\mathbb{R}^3}|u_{\varepsilon}|^{2_s^{\ast}}\,{\rm d}x.
\end{align*}
By $(f_0)$, we have that
\begin{align*}
t_{\varepsilon}^2\int_{\mathbb{R}^3}(|(-\Delta)^{\frac{s}{2}}u_{\varepsilon}|^2+\nu u_{\varepsilon}^2)\,{\rm d}x+t_{\varepsilon}^4\int_{\mathbb{R}^3}\phi_{u_{\varepsilon}}^tu_{\varepsilon}^2\,{\rm d}x\geq\mu t_{\varepsilon}^{2_s^{\ast}}\int_{\mathbb{R}^3}|u_{\varepsilon}|^{2_s^{\ast}}\,{\rm d}x.
\end{align*}
It follows from $(iii)$ of Lemma \ref{lem2-1} that
\begin{align}\label{equ4-5}
t_{\varepsilon}^{2_s^{\ast}}\leq\frac{1}{\mu\|u_{\varepsilon}\|_{2_s^{\ast}}^{2_s^{\ast}}}\Big(t_{\varepsilon}^2\int_{\mathbb{R}^3}(|(-\Delta)^{\frac{s}{2}}u_{\varepsilon}|^2+\nu u_{\varepsilon}^2)\,{\rm d}x+Ct_{\varepsilon}^4\|u_{\varepsilon}\|_{\frac{12}{3+2t}}^4\Big).
\end{align}
Therefore, \eqref{equ4-2}, \eqref{equ4-3}, \eqref{equ4-4} and \eqref{equ4-5} imply that $|t_{\varepsilon}|\leq C_1$, where $C_1$ is independent of $\varepsilon>0$ small. On the other hand, we may assume that there is a positive constant $C_2>0$ such that $t_{\varepsilon}\geq C_2>0$ for $\varepsilon>0$ small. Otherwise, we can find a sequence $\varepsilon_n\rightarrow0$ as $n\rightarrow\infty$ such that $t_{\varepsilon_n}\rightarrow0$ as $n\rightarrow\infty$. Therefore
\begin{equation*}
0<c_{\nu,\kappa,\mu}\leq \sup_{t\geq0}I_{\nu,\kappa,\mu}(tu_{\varepsilon_n})=I_{\nu,\kappa,\mu}(t_{\varepsilon_n}u_{\varepsilon_n})\rightarrow0,
\end{equation*}
which is a contradiction.

Denote $g(t)=\frac{t^2}{2}\int_{\mathbb{R}^3}|(-\Delta)^{\frac{s}{2}}u_{\varepsilon}|^2\,{\rm d}x-\frac{\mu t^{2_s^{\ast}}}{2_s^{\ast}}\int_{\mathbb{R}^3}|u_{\varepsilon}|^{2_s^{\ast}}\,{\rm d}x$, by \eqref{equ4-2} and \eqref{equ4-3}, it is easy to check that
\begin{align*}
\sup_{t\geq0}g(t)=\frac{s}{3}\frac{\Big(\int_{\mathbb{R}^3}|(-\Delta)^{\frac{s}{2}}u_{\varepsilon}|^2\,{\rm d}x\Big)^{\frac{3}{2s}}}{\Big(\mu\int_{\mathbb{R}^3}|u_{\varepsilon}|^{2_s^{\ast}}\,{\rm d}x\Big)^{\frac{3-2s}{2s}}}&=\frac{s}{3\mu^{\frac{3-2s}{2s}}}\frac{\Big(\mathcal{S}_s^{\frac{3}{2s}}+O(\varepsilon^{3-2s})\Big)^{\frac{3}{2s}}}{\Big(\mathcal{S}_s^{\frac{3}{2s}}+O(\varepsilon^3)\Big)^{\frac{3-2s}{2s}}}\\
&=\frac{s}{3\mu^{\frac{3-2s}{2s}}}\Big(\mathcal{S}_s^{\frac{3}{2s}}+O(\varepsilon^3)\Big)\Big(\frac{\mathcal{S}_s^{\frac{3}{2s}}+O(\varepsilon^{3-2s})}{\mathcal{S}_s^{\frac{3}{2s}}+O(\varepsilon^3)}\Big)^{\frac{3}{2s}}\\
&\leq\frac{s}{3\mu^{\frac{3-2s}{2s}}}\mathcal{S}_s^{\frac{3}{2s}}+O(\varepsilon^{3-2s}).
\end{align*}
Thus
\begin{align*}
I_{\nu,\kappa,\mu}(t_{\varepsilon}u_{\varepsilon})&\leq \sup_{t\geq0}g(t)+C\int_{\mathbb{R}^3}|u_{\varepsilon}|^2\, {\rm d}x+C\int_{\mathbb{R}^3}\phi_{u_{\varepsilon}}^tu_{\varepsilon}^2\,{\rm d}x-C\int_{\mathbb{R}^3}|u_{\varepsilon}|^{q+1}\, {\rm d}x\\
&\leq\frac{s}{3\mu^{\frac{3-2s}{2s}}}\mathcal{S}_s^{\frac{3}{2s}}+O(\varepsilon^{3-2s})+C\int_{\mathbb{R}^3}|u_{\varepsilon}|^2\, {\rm d}x+C\Big(\int_{\mathbb{R}^3}|u_{\varepsilon}|^{\frac{12}{3+2t}}\,{\rm d}x\Big)^{\frac{3+2t}{3}}\\
&-C\int_{\mathbb{R}^3}|u_{\varepsilon}|^{q+1}\, {\rm d}x\\
&\leq\frac{s}{3\mu^{\frac{3-2s}{2s}}}\mathcal{S}_s^{\frac{3}{2s}}+O(\varepsilon^{3-2s})+C\Big(\int_{\mathbb{R}^3}|u_{\varepsilon}|^{\frac{12}{3+2t}}\,{\rm d}x\Big)^{\frac{3+2t}{3}}-C\int_{\mathbb{R}^3}|u_{\varepsilon}|^{q+1}\, {\rm d}x.
\end{align*}
where we have used \eqref{equ4-4} and $s>\frac{3}{4}$ which implies $2<\frac{3}{3-2s}$.

Observing that
\begin{align*}
\lim_{\varepsilon\rightarrow0^{+}}\frac{\Big(\int_{\mathbb{R}^3}|u_{\varepsilon}|^{\frac{12}{3+2t}}\,{\rm d}x\Big)^{\frac{3+2t}{3}}}{\varepsilon^{3-2s}}\leq\left\{
                                                                                                                            \begin{array}{ll}
                                                                                                                              \lim\limits_{\varepsilon\rightarrow0^{+}}\frac{O(\varepsilon^{2t+4s-3})}{\varepsilon^{3-2s}}=0, & \hbox{$\frac{12}{3+2t}>\frac{3}{3-2s}$,} \\
                                                                                                                              \lim\limits_{\varepsilon\rightarrow0^{+}}\frac{O(\varepsilon^{2t+4s-3}|\log\varepsilon|)}{\varepsilon^{3-2s}}=0, & \hbox{$\frac{12}{3+2t}=\frac{3}{3-2s}$,} \\
                                                                                                                             \lim\limits_{\varepsilon\rightarrow0^{+}} \frac{O(\varepsilon^{2(3-2s)}|\log\varepsilon|)}{\varepsilon^{3-2s}}=0, & \hbox{$\frac{12}{3+2t}<\frac{3}{3-2s}$.}
                                                                                                                            \end{array}
                                                                                                                          \right.
\end{align*}
Since $s>\frac{3}{4}$ and $q>3$, then $q+1>\frac{3}{3-2s}$, $2s-\frac{3-2s}{2}(q+1)<0$. Thus
\begin{align*}
\lim_{\varepsilon\rightarrow0^{+}}\frac{\int_{\mathbb{R}^3}|u_{\varepsilon}|^{q+1}\,{\rm d}x}{\varepsilon^{3-2s}}=\lim_{\varepsilon\rightarrow0^{+}}\frac{O(\varepsilon^{3-\frac{3-2s}{2}(q+1)})}{\varepsilon^{3-2s}}=+\infty.
\end{align*}
Therefore, we have proved that for $\varepsilon$ small enough, there holds
\begin{equation*}
I_{\nu,\kappa,\mu}(t_{\varepsilon}u_{\varepsilon})<\frac{s}{3\mu^{\frac{3-2s}{2s}}}\mathcal{S}_s^{\frac{3}{2s}}
\end{equation*}
and thus the proof is completed.
\end{proof}
Similar arguments to Lemma \ref{lem3-4}, we can obtain the compactness of the $(PS)_{c_{\nu,\kappa,\mu}}$ sequence. It is stated as follows.
\begin{lemma}\label{lem4-2}
Assume $c_{\nu,\kappa,\mu}<\frac{s}{3\mu^{\frac{3-2s}{2s}}}\mathcal{S}_s^{\frac{3}{2s}}$ and let $\{u_n\}$ be the $(PS)$ sequence at the level $c_{\nu,\kappa,\mu}$. Then there exists $\{y_n\}\subset\mathbb{R}^3$ such that for every $\xi>0$, there exists $\widehat{R}>0$ such that
\begin{equation*}
\int_{\mathbb{R}^3\backslash B_{\widehat{R}}(y_n)}\Big(|(-\Delta)^{\frac{s}{2}}u_n|^2+|u_n|^2+|u_n|^{2_s^{\ast}}\Big)\,{\rm d}x<\xi.
\end{equation*}
\end{lemma}
\begin{proof}
It is easy to verify that the $(PS)_{c_{\nu,\kappa,\mu}}$ sequence $\{u_n\}$ which satisfying $I_{\nu,\kappa,\mu}(u_n)\rightarrow c_{\nu,\kappa,\mu}$ and $I'_{\nu,\kappa,\mu}(u_n)\rightarrow0$ as $n\rightarrow\infty$, is bounded in $E_{\nu}$.

Let $\rho_n(x)=\frac{1}{4}|(\Delta)^{\frac{s}{2}}u_n|^2+\frac{\nu}{4}|u_n|^2+\frac{\kappa}{4}(f(u_n)u_n-4F(u_n))+\frac{4s-3}{12}\mu|u_n|^{2_s^{\ast}}$, then $\rho_n\in L^1(\mathbb{R}^3)$ is also bounded.

Checking the proof of Lemma \ref{lem3-4} line by line, we find that it is only to prove the Case 2.

Case 2. Up to a subsequence, we may assume that $\langle I_{\nu,\kappa,\mu}'(v_n),v_n\rangle>0$ and $\langle I_{\nu,\kappa,\mu}'(w_n),w_n\rangle>0$.
Thus,  $\langle I_{\nu,\kappa,\mu}'(v_n),v_n\rangle=o_n(1)$ and $\langle I_{\nu,\kappa,\mu}'(w_n),w_n\rangle=o_n(1)$. Similar to the proof of $(v)$ of Lemma \ref{lem3-1}, there exist two sequences $\{t_n\}\subset\mathbb{R}_{+}$ and $\{s_n\}\subset\mathbb{R}_{+}$ satisfying $t_n\rightarrow1$ and $s_n\rightarrow1$ as $n\rightarrow\infty$, respectively, such that $t_nw_n\in\mathcal{N}_{\nu,\kappa,\mu}$, $s_nv_n\in \mathcal{N}_{\nu,\kappa,\mu}$. Therefore,
\begin{equation*}
I_{\nu,\kappa,\mu}(v_n)=I_{\nu,\kappa,\mu}(s_nv_n)+o_n(1),\quad I_{\nu,\kappa,\mu}(w_n)=I_{\nu,\kappa,\mu}(t_nw_n)+o_n(1)
\end{equation*}
which leads to a contradiction that
\begin{align*}
c_{\nu,\kappa,\mu}&=I_{\nu,\kappa,\mu}(u_n)+o_n(1)\geq I_{\nu,\kappa,\mu}(v_n)+I_{\nu,\kappa,\mu}(w_n)+o_n(1)\\
&= I_{\nu,\kappa,\mu}(s_nv_n)+I_{\nu,\kappa,\mu}(t_nw_n)+o_n(1)\\
&\geq 2c_{\nu,\kappa,\mu}+o_n(1).
\end{align*}
Hence, dichotomy does not happen.  The sequence $\{\rho_n\}$ must be compactness, i.e., there exists $\{y_n\}\subset\mathbb{R}^3$ such that for every $\xi>0$, there exists $\widehat{R}>0$, we have $\int_{\mathbb{R}^3\backslash B_{\widehat{R}}(y_n)}\rho_n(x)\,{\rm d}x<\xi$. The proof is completed.
\end{proof}
\begin{remark}\label{rem4-1}
By the interpolation inequality, we could show that the sequence $\{|u_n|^m\}$ with $2\leq m\leq 2_s^{\ast}$ is also compactness.
\end{remark}
\begin{proposition}\label{pro4-1}
Assume that $(f_0)$--$(f_2)$ hold. Then problem \eqref{equ2-5} has at least a positive ground state solution in $E_{\nu}$ satisfying $\lim\limits_{|x|\rightarrow\infty}u(x)=0$.
\end{proposition}
\begin{proof}
From the above arguments, we see that there exists a $(PS)_{c_{\nu,\kappa,\mu}}$ sequence for $I_{\nu,\kappa,\mu}$. From Lemma \ref{lem4-2}, the sequence $\{u_n\}$ is bounded and verifies the compactness in the sense of Proposition \ref{pro2-1}. Set $\widehat{u_n}(\cdot)=u_n(\cdot+y_n)$. Using the invariance of $\mathbb{R}^3$ by translation, we see that $\{v_n\}$ is a bounded $(PS)_{c_{\nu,\kappa,\mu}}$ sequence and
\begin{equation*}
\int_{\mathbb{R}^3\backslash B_{\widehat{R}}(0)}\Big(|(-\Delta)^{\frac{s}{2}}\widehat{u_n}|^2+|\widehat{u_n}|^2+|\widehat{u_n}|^{2_s^{\ast}}\Big)\,{\rm d}x<\xi.
\end{equation*}
Since $\{\widehat{u_n}\}$ is bounded in $E_{\nu}$, up to a subsequence, still denoted by $\{\widehat{u_n}\}$, there exists $\widehat{u}\in E_{\nu}$ such that $\widehat{u_n}\rightharpoonup\widehat{u}$ in $E_{\nu}$, $\widehat{u_n}\rightarrow\widehat{u}$ in $L_{loc}^r(\mathbb{R}^3)$ for $2\leq r<2_s^{\ast}$. From Remark \ref{rem4-1}, we see that $\int_{\mathbb{R}^3\backslash B_{\widehat{R}}(0)}|\widehat{u_n}|^r\,{\rm d}x\leq \xi$ and hence $\widehat{u_n}\rightarrow\widehat{u}$ in $L^r(\mathbb{R}^3)$ for $2\leq r<2_s^{\ast}$. Similar arguments to the proof of compactness in Lemma \ref{lem3-4}, we can conclude that $\widehat{u_n}\rightarrow\widehat{u}$ in $E_{\nu}$. Hence, $I_{\nu,\kappa,\mu}(\widehat{u})=c_{\nu,\kappa,\mu}$ and $I_{\nu,\kappa,\mu}'(\widehat{u})=0$, that is, $\widehat{u}$ is a nontrivial critical point of $I_{\nu,\kappa,\mu}$. From
the equivalent characterize of mountain value, we conclude that $\widehat{u}$ is a nontrivial ground state solution of problem \eqref{equ2-5}.

Finally, we only need to show that $\widehat{u}$ is positive. For simplicity, we replace $\widehat{u}$ by $u$ in the following discussion. If we replace $I_{\nu,\kappa,\mu}$ by the following functional
\begin{align*}
I_{\nu,\kappa,\mu}^{+}(u)=\frac{1}{2}\|u\|_{E_{\nu}}^2+\frac{1}{4}\int_{\mathbb{R}^3}\phi_u^tu^2\,{\rm d}x-\int_{\mathbb{R}^3}\kappa F(u)\,{\rm d}x-\frac{1}{2_s^{\ast}}\int_{\mathbb{R}^3}\mu|u^{+}|^{2_s^{\ast}}\,{\rm d}x,
\end{align*}
where $u^{\pm}=\max\{\pm u,0\}$, then we see that all the calculations above can be repeated word by word. So, there exists a nontrivial ground state critical point $u\in E_{\nu}$ of $I_{\nu,\kappa,\mu}^{+}$. Hence,
\begin{equation*}
0=\langle (I_{\nu,\kappa,\mu}^{+})'(u),u^{-}\rangle=\langle u,u^{-}\rangle-\int_{\mathbb{R}^3}\phi_u^t(u^{-})^2\,{\rm d}x
\end{equation*}
which implies that
\begin{equation*}
\int_{\mathbb{R}^3}\int_{\mathbb{R}^3}\frac{(u(x)-u(y))(u^{-}(x)-u^{-}(y))}{|x-y|^{3+2s}}\,{\rm d}x\,{\rm d}y\geq0.
\end{equation*}
However, by computation, we have that
\begin{align*}
&\int_{\mathbb{R}^3}\int_{\mathbb{R}^3}\frac{(u(x)-u(y))(u^{-}(x)-u^{-}(y))}{|x-y|^{3+2s}}\,{\rm d}x\,{\rm d}y=\int_{\{u(x)\geq0\}\times\{u(y)<0\}}\frac{(u(x)-u(y))u(y)}{|x-y|^{3+2s}}\,{\rm d}x\,{\rm d}y\\
&+\int_{\{u(x)<0\}\times\{u(y)\geq0\}}\frac{(u(y)-u(x))u(x)}{|x-y|^{3+2s}}\,{\rm d}x\,{\rm d}y-\int_{\{u(x)<0\}\times\{u(y)<0\}}\frac{(u(x)-u(y))^2}{|x-y|^{3+2s}}\,{\rm d}x\,{\rm d}y\leq0.
\end{align*}
Hence,
\begin{equation*}
\int_{\mathbb{R}^3}\int_{\mathbb{R}^3}\frac{(u(x)-u(y))(u^{-}(x)-u^{-}(y))}{|x-y|^{3+2s}}\,{\rm d}x\,{\rm d}y=0
\end{equation*}
which leads to $u^{-}=0$, and thus $u\geq0$ and $u\not\equiv0$.

Let $f(x,u)=\mu|u|^{2_s^{\ast}-1}+\kappa f(u)-\nu u-\phi_u^tu$, by $(f_0)$ and $(f_2)$, it is easy to check that $f(x,u)\leq \mu u^{2_s^{\ast}-1}+1$, for any $u\geq0$. Using Proposition 4.1.1 in \cite{DMV}, we see that $u\in L^{\infty}(\mathbb{R}^3)$ and check the proof of Proposition 4.1.1 word by word, using 4.1.6 and 4.1.7, we can obtain that
\begin{align}\label{equ4-5-0}
\|u\|_{\infty}\leq C(\int_{\mathbb{R}^3}|u|^{2_s^{\ast}}\,{\rm d}x+\int_{\mathbb{R}^3}|u|^{2_s^{\ast}\beta_1}\,{\rm d}x)^{\frac{1}{2_s^{\ast}(\beta_1-1)}}&\leq C\Big[\Big(\int_{\mathbb{R}^3}|u|^{2_s^{\ast}}\,{\rm d}x\Big)^{\frac{1}{2_s^{\ast}(\beta_1-1)}}\nonumber\\
&+\Big(\int_{\mathbb{R}^3}|u|^{2_s^{\ast}}\,{\rm d}x\Big)^{\frac{1}{2(\beta_1-1)}}\Big]
\end{align}
where $C>0$ independent of $u$, $\beta_1=\frac{2_s^{\ast}+1}{2}$. This yields $u\in L^r(\mathbb{R}^3)$ for all $r\in[2,+\infty]$. Moreover, $\phi_u^t\in L^{\infty}(\mathbb{R}^3)$. Therefore,
according to Proposition 2.9 in \cite{S}, and $s>\frac{3}{4}$, we see that $u\in C^{1,\alpha}(\mathbb{R}^3)$ for any $0<\alpha<2s-1$. Thus, by Lemma 3.2 in \cite{DPV}, we have that
\begin{equation*}
(-\Delta)^su(x)=-\frac{1}{2}C(3,s)\int_{\mathbb{R}^3}\frac{u(x+y)+u(x-y)-2u(x)}{|x-y|^{3+2s}}\,{\rm d}x\,{\rm d}y,\quad \forall\,\, x\in\mathbb{R}^3.
\end{equation*}
Assume that there exists $x_0\in\mathbb{R}^3$ such that $u(x_0)=0$, then from $u\geq0$ and $u\not\equiv0$, we get
\begin{equation*}
(-\Delta)^su(x_0)=-\frac{1}{2}C(3,s)\int_{\mathbb{R}^3}\frac{u(x_0+y)+u(x_0-y)}{|x_0-y|^{3+2s}}\,{\rm d}x\,{\rm d}y<0.
\end{equation*}
However, observe that $(-\Delta)^su(x_0)=-\nu u(x_0)-(\phi_u^tu)(x_0)+\kappa f(u(x_0))+\mu u(x_0)^{2_s^{\ast}-1}=0$, a contradiction. Hence, $u(x)>0$, for every $x\in\mathbb{R}^3$. Finally, the fact $u\in L^r(\mathbb{R}^3)\cap C^{1,\alpha}(\mathbb{R}^3)$ for $2\leq r\leq \infty$ implies that $\lim\limits_{|x|\rightarrow\infty}u(x)=0$. The proof is completed.
\end{proof}

\begin{lemma}\label{lem4-3}
There exists $\varepsilon^{\ast}>0$ such that $c_{\varepsilon}<m_{\infty}$ for all $\varepsilon\in(0,\varepsilon^{\ast})$.
\end{lemma}

\begin{proof}
By the assumption $(V)$, there must exist a $w\in\mathbb{R}_{+}$ such that $V_0<w<V_{\infty}$. Hence, $m_0<m_{w,K_0,Q_0}\leq m_{w,K_{\infty},Q_{\infty}}<m_{\infty}$ owing to $K_0\geq K_{\infty}$ and $Q_0\geq Q_{\infty}$. Indeed, using Proposition \ref{pro4-1}, choose $u$ be a ground state solution of problem \eqref{equ4-1} with $\nu=V_{\infty}$, $\kappa=K_{\infty}$ and $\mu=Q_{\infty}$, such that $I_{\infty}(u)=m_{\infty}$. Then there holds $I_{\infty}(u)=\max\limits_{t\geq0}I_{\infty}(tu)$ and there exists $t_0>0$ such that $t_0u\in\mathcal{N}_{w,K_{\infty},Q_{\infty}}$ and $I_{w,K_{\infty},Q_{\infty}}(t_0u)=\max\limits_{t\geq0}I_{w,K_{\infty},Q_{\infty}}(tu)$. Hence
\begin{align*}
m_{\infty}&=\max_{t\geq0}I_{\infty}(tu)\geq\max_{t\geq0}I_{w,K_{\infty},Q_{\infty}}(tu)=I_{w,K_{\infty},Q_{\infty}}(t_0u)\\
&=I_{\infty}(t_0u)=I_{w,K_{\infty},Q_{\infty}}(t_0u)+\frac{1}{2}(V_{\infty}-w)\int_{\mathbb{R}^3}|u|^2\,{\rm d}x\\
&>I_{w,K_{\infty},Q_{\infty}}(t_0u)\geq m_{w,K_{\infty},Q_{\infty}}.
\end{align*}
Similarly, we can show that $m_0<m_{w,K_{\infty},Q_{\infty}}$.

Taking $\nu=w$, $\kappa=K_{\infty}$ and $\mu=Q_{\infty}$, in view of Proposition \ref{pro4-1}, we know that there exists $v\in \mathcal{N}_{w,K_{\infty},Q_{\infty}}$ such that $I_{w,K_{\infty},Q_{\infty}}(v)=m_{w,K_{\infty},Q_{\infty}}$. Let $\eta\in C_0^{\infty}(\mathbb{R}^3,[0,1])$ be such that $\eta(x)=1$ if $|x|\leq 1$, and $\eta(x)=0$ if $|x|\geq 2$. For $\theta>0$, set $u_{\theta}(x)=\eta(x/\theta) v(x)$. Similar to the proof of $(ii)$ of Lemma \ref{lem3-1}, there exists $t_{\theta}>0$ such that $t_{\theta}u_{\theta}\in \mathcal{N}_{w,K_{\infty},Q_{\infty}}$. We claim that there exists $\theta_0>0$ such that $I_{w,K_{\infty},Q_{\infty}}(t_{\theta_0}u_{\theta_0})<m_{\infty}$. For convenience, we denote $u=t_{\theta_0}u_{\theta_0}$. In fact, if $I_{w,K_{\infty},Q_{\infty}}(t_{\theta}u_{\theta})\geq m_{\infty}$, for all $\theta>0$. From the definition of $\eta(x)$, using Lemma 5 in \cite{PP}, $u_{\theta}\rightarrow v$ in $H^s(\mathbb{R}^3)$ as $\theta\rightarrow\infty$. Recall that $v\in \mathcal{N}_{w,K_{\infty},Q_{\infty}}$, $\mathcal{N}_{w,K_{\infty},Q_{\infty}}$ is closed in $H^s(\mathbb{R}^3)$, thus we obtain that $t_{\theta}\rightarrow1$. Thus
\begin{equation*}
m_{\infty}\leq \liminf_{\theta\rightarrow\infty}I_{w,K_{\infty},Q_{\infty}}(t_{\theta}u_{\theta})=I_{w,K_{\infty},Q_{\infty}}(v)=m_{w,K_{\infty},Q_{\infty}}<m_{\infty}
\end{equation*}
which is impossible. Hence our claim is true. Since the compact support set of $u$ denoted by ${\rm supp}u$ is compact and $V(0)=V_0$, we can choose $\varepsilon^{\ast}>0$ small enough such that $V(\varepsilon x)\leq w$ for all $x\in {\rm supp}u$ and $\varepsilon\in(0,\varepsilon^{\ast})$. Thus, we have
\begin{equation*}
\max_{t\geq0}I_{\varepsilon}(tu)\leq\max_{t\geq0}I_{w,K_{\infty},Q_{\infty}}(tu)=I_{w,K_{\infty},Q_{\infty}}(u)<m_{\infty}\quad \text{for any}\,\, \varepsilon\in(0,\varepsilon^{\ast})
\end{equation*}
which implies that $c_{\varepsilon}<m_{\infty}$ for $\varepsilon\in(0,\varepsilon^{\ast})$.

\end{proof}
\begin{proposition}\label{pro4-2}
Suppose that $(V)$, $(Q_0)$, $(Q_1)$, $(K)$, $(f_0)$--$(f_2)$ hold and $s\in(\frac{3}{4},1)$. Then there exists $\widetilde{\varepsilon}_0>0$ small, such that for each $\varepsilon\in (0,\widetilde{\varepsilon}_0)$, problem \eqref{equ2-4} has at least a positive ground state solution $u_{\varepsilon}$ satisfying $\lim\limits_{|x|\rightarrow\infty}u_{\varepsilon}(x)=0$.
\end{proposition}
\begin{proof}
From Lemma \ref{lem3-2}, Lemma \ref{lem3-3}, Lemma \ref{lem4-3} and Lemma \ref{lem3-4}, there exists a small $\widetilde{\varepsilon}_0>0$, such that $I_{\varepsilon}$ has a nontrivial critical point $u_{\varepsilon}\in H_{\varepsilon}$ for $\varepsilon\in(0,\widetilde{\varepsilon}_0)$. Hence, $u_{\varepsilon}$ is a nontrivial ground state solution of problem \eqref{equ2-4}. Similar arguments as Proposition \ref{pro4-1}, we can prove that for $\varepsilon\in(0,\widetilde{\varepsilon}_0)$, $u_{\varepsilon}$ is a positive ground state solution for problem \eqref{equ2-4} with $\lim\limits_{|x|\rightarrow\infty}u_{\varepsilon}(x)=0$.
\end{proof}
In the end of this section, we will establish the relation between $\lim\limits_{\varepsilon\rightarrow0}c_{\varepsilon}$ and $m_0$. We state the relation as the following Lemma.

\begin{lemma}\label{lem4-4}
$\lim\limits_{\varepsilon\rightarrow0}c_{\varepsilon}=m_0$.
\end{lemma}
\begin{proof}
First, we show that there exists $\varepsilon_1>0$ such that $c_{\varepsilon}\geq m_0$ for all $\varepsilon\in(0,\varepsilon_1)$. We suppose by contradiction that for any given $\varepsilon_1>0$, there exists some $\varepsilon_0\in (0,\varepsilon_1)$ such that $c_{\varepsilon_0}<m_0$.  By Proposition \ref{pro4-2}, we can choose $u_0$ be a ground state solution of problem \eqref{equ2-4} such that $I_{\varepsilon_0}(u_0)=c_{\varepsilon_0}=\max_{t\geq0}I_{\varepsilon_0}(tu_0)<m_0$. Similar to the proof of $(ii)$ of Lemma \ref{lem3-1}, there exists $t_0>0$ such that $t_0u_0\in \mathcal{N}_0$ such that $I_0(t_0u_0)=\max\limits_{t\geq0}I_0(tu_0)$. Hence, by $V(\varepsilon_0 x)\geq V_0$, $K(\varepsilon_0 x)\leq K_0$ and $Q(\varepsilon_0 x)\leq Q_0$, we get that
\begin{equation*}
m_0>I_{\varepsilon_0}(u_0)=\max_{t\geq0}I_{\varepsilon_0}(tu_0)\geq\max_{t\geq0}I_0(tu_0)=I_0(t_0u_0)\geq m_0
\end{equation*}
which is a contradiction.

Next, we will prove that $\limsup\limits_{\varepsilon\rightarrow0}c_{\varepsilon}\leq m_0$. Let $u_0$ be a nontrivial ground state solution of equation \eqref{equ4-1} with $\nu=V_0$, $\kappa=K_0$ and $\mu=Q_0$, that is $I_0(u_0)=m_0$. For each $\theta>0$, set $u_{\theta}(x)=\eta(x/\theta)u_0(x)$, where $\eta\in C_0^{\infty}(\mathbb{R}^3,[0,1])$ be such that $\eta(x)=1$ if $|x|\leq 1$, and $\eta(x)=0$ if $|x|\geq 2$. From the definition of $\eta(x)$, using Lemma 5 in \cite{PP}, $u_{\theta}\rightarrow u_0$ in $H^s(\mathbb{R}^3)$ as $\theta\rightarrow\infty$.

For each $\varepsilon>0$, $\theta>0$, there exists $t_{\varepsilon,\theta}>0$ such that $I_{\varepsilon}(t_{\varepsilon,\theta}u_{\theta})=\max_{t\geq0}I_{\varepsilon}(tu_{\theta})$. Thus, $\langle I_{\varepsilon}'(t_{\varepsilon,\theta}u_{\theta}),t_{\varepsilon,\theta}u_{\theta}\rangle=0$, that is,
\begin{align}\label{equ4-6}
t_{\varepsilon,\theta}^2\int_{\mathbb{R}^3}&\Big(|(-\Delta)^{\frac{s}{2}}u_{\theta}|^2+V(\varepsilon x)|u_{\theta}|^2\Big)\,{\rm d}x+t_{\varepsilon,\theta}^4\int_{\mathbb{R}^3}\phi_{u_{\theta}}^tu_{\theta}^2\,{\rm d}x\nonumber\\
&=\int_{\mathbb{R}^3}K(\varepsilon x)f(t_{\varepsilon,\theta}u_{\theta})t_{\varepsilon,\theta}u_{\theta}\,{\rm d}x+t_{\varepsilon,\theta}^{2_s^{\ast}}\int_{\mathbb{R}^3}Q(\varepsilon x)|u_{\theta}|^{2_s^{\ast}}\,{\rm d}x.
\end{align}
Similar to Lemma \ref{lem3-4} or Lemma \ref{lem4-2}, we can deduce that for each $\theta>0$, $0<\lim\limits_{\varepsilon\rightarrow0}t_{\varepsilon,\theta}=t_{\theta}<\infty$.
Taking the limit as $\varepsilon\rightarrow0$ in \eqref{equ4-6}, we get
\begin{align*}
t_{\theta}^2\int_{\mathbb{R}^3}&\Big(|(-\Delta)^{\frac{s}{2}}u_{\theta}|^2+V_0|u_{\theta}|^2\Big)\,{\rm d}x+t_{\theta}^4\int_{\mathbb{R}^3}\phi_{u_{\theta}}^tu_{\theta}^2\,{\rm d}x\nonumber\\
&=\int_{\mathbb{R}^3}K_0f(t_{\theta}u_{\theta})t_{\theta}u_{\theta}\,{\rm d}x+t_{\theta}^{2_s^{\ast}}\int_{\mathbb{R}^3}Q_0|u_{\theta}|^{2_s^{\ast}}\,{\rm d}x.
\end{align*}
This yields to $t_{\theta}u_{\theta}\in\mathcal{N}_0$, i.e., $I_0(t_{\theta}u_{\theta})=\max\limits_{t\geq0}I_0(tu_{\theta})$. Recall that $u_{\theta}\rightarrow u_0$ in $H^s(\mathbb{R}^3)$ as $\theta\rightarrow \infty$ and $u_0\in\mathcal{N}_0$, thus, using $(ii)$ of Lemma \ref{lem3-1}, we have $t_{\theta}\rightarrow1$ as $\theta\rightarrow\infty$. By the definition of $c_{\varepsilon}$, we have that
\begin{equation*}
\limsup_{\varepsilon\rightarrow0}c_{\varepsilon}\leq\limsup_{\varepsilon\rightarrow0}\max_{t\geq0}I_{\varepsilon}(tu_{\theta})=\limsup_{\varepsilon\rightarrow0}I_{\varepsilon}(t_{\varepsilon,\theta}u_{\theta})
=I_0(t_{\theta}u_{\theta}).
\end{equation*}
Let $\theta\rightarrow\infty$, we get that $\limsup\limits_{\varepsilon\rightarrow0}c_{\varepsilon}\leq I_0(u_0)=m_0$. The proof is completed.
\end{proof}

\section{Concentration behavior}
In this section, we study the concentration behavior of ground state solutions of system \eqref{main}. In this section, we choose $H^s(\mathbb{R}^3)$ as our work space since $H^s(\mathbb{R}^3)=E_{\nu}=H_{\varepsilon}$ for any $\varepsilon>0$, $\nu>0$.

Proposition \ref{pro4-2} tells us that there exists $\widetilde{\varepsilon}_0>0$, such that for each $\varepsilon\in(0,\widetilde{\varepsilon}_0)$, problem \eqref{equ2-4} possesses a positive ground state solution $v_{\varepsilon}\in H^s(\mathbb{R}^3)$ satisfying $I_{\varepsilon}(v_{\varepsilon})=c_{\varepsilon}$ and $I_{\varepsilon}'(v_{\varepsilon})=0$. Now, we study the behavior of the family $\{v_{\varepsilon}\}$.
\begin{lemma}\label{lem5-1}
For the family $\{v_{\varepsilon}\}$ satisfying $I_{\varepsilon}(v_{\varepsilon})=c_{\varepsilon}$ and $I_{\varepsilon}'(v_{\varepsilon})=0$, there exist $\widehat{\varepsilon}>0$, such that for all $\varepsilon\in(0,\widehat{\varepsilon})$, there exist a family $\{y_{\varepsilon}\}\subset\mathbb{R}^3$, and constants $R,\sigma>0$ such that
\begin{equation}\label{equ5-1}
\int_{B_R(y_{\varepsilon})}|v_{\varepsilon}|^2\,{\rm d}x\geq\sigma.
\end{equation}
\end{lemma}
\begin{proof}
Suppose by contradiction that \eqref{equ5-1} does not happen. Then there exists a sequence $\varepsilon_n\rightarrow0$ as $n\rightarrow\infty$ such that
\begin{equation*}
\lim_{n\rightarrow\infty}\sup_{y\in\mathbb{R}^3}\int_{B_R(y)}|v_{\varepsilon_n}|^2\,{\rm d}x=0.
\end{equation*}
Taking a similar discussion as that in the proof of Lemma \ref{lem3-4}, we can easily obtain a contradiction. Hence, \eqref{equ5-1} holds.
\end{proof}

For simplicity, we denote
\begin{equation*}
w_{\varepsilon}(x):=v_{\varepsilon}(x+y_{\varepsilon})(=u_{\varepsilon}(\varepsilon x+\varepsilon y_{\varepsilon})).
\end{equation*}
By the fact that $I(v_{\varepsilon})=c_{\varepsilon}$ and $I_{\varepsilon}'(v_{\varepsilon})=0$, so $w_{\varepsilon}$ is a positive ground state solution to the following equation
\begin{equation}\label{equ5-1-0}
(-\Delta)^sw+V(\varepsilon x+\varepsilon y_{\varepsilon})w+\phi_{w}^tw=K(\varepsilon x+\varepsilon y_{\varepsilon})f(w)+Q(\varepsilon x+\varepsilon y_{\varepsilon})|w|^{2_s^{\ast}-2}w.
\end{equation}

\begin{lemma}\label{lem5-2}
The family $\{\varepsilon y_{\varepsilon}\}$ which obtained in Lemma \ref{lem5-1} is bounded in $\mathbb{R}^3$ for any $\varepsilon\in(0,\widehat{\varepsilon})$.
\end{lemma}
\begin{proof}
Suppose by contradiction that $\{\varepsilon y_{\varepsilon}\}$ is unbounded, then there exist two sequences $\varepsilon_n$ and $\{\varepsilon_ny_{\varepsilon_n}\}$ such that $\lim\limits_{n\rightarrow\infty}\varepsilon_n\rightarrow0$ and  $\lim\limits_{n\rightarrow\infty}|\varepsilon_ny_{\varepsilon_n}|=+\infty$. In the sequel, for simplicity, we denote $y_n=y_{\varepsilon_n}$ and $v_n=v_{\varepsilon_n}$. Set $w_n(\cdot)=v_n(\cdot+y_n)$, then for each $n\in\mathbb{N}$, $w_n\geq0$ satisfies that $I_{\varepsilon_n}(w_n)=c_{\varepsilon_n}$ and $I_{\varepsilon_n}'(w_n)=0$,
and from \eqref{equ5-1}, we have that
\begin{equation}\label{equ5-2}
\int_{B_R(0)}|w_n|^2\,{\rm d}x\geq\sigma>0 \quad \text{for all}\,\, n\in\mathbb{N}.
\end{equation}
By Lemma \ref{lem4-4}, it is easy to check that $\{w_n\}$ is bounded in $H^s(\mathbb{R}^3)$. Thus, up to a subsequence, still denoted by $\{w_n\}$, we assume that there exists $w\in H^s(\mathbb{R}^3)$ such that $w_n\rightharpoonup w$ in $H^s(\mathbb{R}^3)$, $w_n\rightarrow w$ in $L_{loc}^r(\mathbb{R}^3)$ for $2\leq r<2_s^{\ast}$ and $w_n\rightarrow w$ a.e. in $\mathbb{R}^3$. Obviously $w\geq0$. Moreover, from \eqref{equ5-2}, we see that $w\not\equiv0$.

For each $n\in\mathbb{N}$, there exists $t_n>0$ such that $t_nw_n\in\mathcal{N}_0$. We claim that $\lim\limits_{n\rightarrow\infty}I_0(t_nw_n)=m_0$. Since $t_nw_n\in\mathcal{N}_0$, then clearly $I_0(t_nw_n)\geq m_0$. On the other hand, by Lemma \ref{lem4-3}, we have that
\begin{align}\label{equ5-3}
I_0(t_nw_n)&\leq\frac{t_n^2}{2}\int_{\mathbb{R}^3}\Big(|(-\Delta)^{\frac{s}{2}}w_n|^2+V(\varepsilon_n x+\varepsilon_n y_n)|w_n|^2\Big)\,{\rm d}x+\frac{t_n^4}{4}\int_{\mathbb{R}^3}\phi_{w_n}^tw_n^2\,{\rm d}x\nonumber\\
&-\int_{\mathbb{R}^3}K(\varepsilon_n x+\varepsilon_n y_n)F(t_nw_n)\,{\rm d}x-\int_{\mathbb{R}^3}Q(\varepsilon_n x+\varepsilon_n y_n)|w_n|^{2_s^{\ast}}\,{\rm d}x\nonumber\\
&=I_{\varepsilon_n}(t_nw_n)\leq\max_{t\geq0}I_{\varepsilon_n}(tw_n)=I_{\varepsilon_n}(w_n)=c_{\varepsilon_n}=m_0+o_n(1)
\end{align}
which yields to $\limsup\limits_{n\rightarrow\infty}I_0(t_nw_n)\leq m_0$ and hence the claim is true.

Since $\{w_n\}$ is bounded in $H^s(\mathbb{R}^3)$, by \eqref{equ5-3}, it is easy to get that $\{t_n\}$ is bounded. Thus, up to a subsequence, still denoted by $\{t_n\}$, we may assume that $\lim\limits_{n\rightarrow\infty}t_n=t\geq0$. If $t=0$, in view of the boundedness of $\{w_n\}$ in $H^s(\mathbb{R}^3)$, we see that $t_nw_n\rightarrow0$ in $H^s(\mathbb{R}^3)$, and thus $\lim\limits_{n\rightarrow\infty}I_0(t_nw_n)=0$, a contradiction. Hence, $t>0$.

Observing that $\{t_nw_n\}$ is bounded in $H^s(\mathbb{R}^3)$, up to a subsequence, still denoted by $\{t_nw_n\}$, we may assume that $t_nw_n\rightharpoonup \widehat{w}$ in $H^s(\mathbb{R}^3)$. Since $w_n\rightharpoonup w$ in $H^s(\mathbb{R}^3)$ and $t_n\rightarrow t$ as $n\rightarrow\infty$, then $t_nw_n\rightharpoonup tw$ in $H^s(\mathbb{R}^3)$ as $n\rightarrow\infty$. By the uniqueness of weak limit, it yields to $\widehat{w}=tw$. Therefore, we obtain a bounded minimizing sequence $\{t_nw_n\}\subset\mathcal{N}_0$ as $n\rightarrow\infty$ and $\langle I_0'(t_nw_n),t_nw_n\rangle=0$ for any $n\in\mathbb{N}$. Similar proof as that done in the proof of $(vi)$ of Lemma \ref{lem3-1}, we may assume that $\{t_nw_n\}$ is a $(PS)_{m_0}$ sequence for $I_0$. By Lemma \ref{lem4-1}, using similar argument as the proof of Proposition \ref{pro4-1}, we can conclude that $t_nw_n\rightarrow tw$ in $H^s(\mathbb{R}^3)$. Moreover, $tw\in\mathcal{N}_0$. Thus,
\begin{equation*}
0\leq t\|w_n-w\|\leq|t_n-t|\|w_n\|+\|t_nw_n-tw\|\rightarrow0\quad \text{as}\,\, n\rightarrow\infty,
\end{equation*}
that is, $w_n\rightarrow w$ in $H^s(\mathbb{R}^3)$ as $n\rightarrow\infty$. Therefore, by Fatou's Lemma and $t_nw_n\in\mathcal{N}_0$, recalling that $\varepsilon_n\rightarrow0$ and $|\varepsilon_ny_n|\rightarrow\infty$ as $n\rightarrow\infty$, we have that
\begin{align}\label{equ5-4}
m_0&\leq I_0(tw)<I_{\infty}(tw)=I_{\infty}(tw)-\frac{1}{4}\langle I_0'(tw),tw\rangle\nonumber\\
&=\frac{1}{4}\int_{\mathbb{R}^3}\Big(|(-\Delta)^{\frac{s}{2}}tw|^2\,{\rm d}x+\int_{\mathbb{R}^3}(\frac{V_{\infty}}{2}-\frac{V_0}{4})|tw|^2\Big)\,{\rm d}x+\int_{\mathbb{R}^3}(\frac{Q_0}{4}-\frac{Q_{\infty}}{2_s^{\ast}})|tw|^{2_s^{\ast}}\,{\rm d}x\nonumber\\
&+\int_{\mathbb{R}^3}(\frac{K_0}{4}f(tw)tw-K_{\infty}F(tw))\,{\rm d}x\nonumber\\
&\leq\liminf_{n\rightarrow\infty}\frac{1}{4}\int_{\mathbb{R}^3}\Big(|(-\Delta)^{\frac{s}{2}}t_nw_n|^2\,{\rm d}x+\liminf_{n\rightarrow\infty}\int_{\mathbb{R}^3}(\frac{V(\varepsilon_n x+\varepsilon_ny_n)}{2}-\frac{V_0}{4})|t_nw_n|^2\Big)\,{\rm d}x\nonumber\\
&+\liminf_{n\rightarrow\infty}\int_{\mathbb{R}^3}(\frac{K_0}{4}f(t_nw_n)t_nw_n-K(\varepsilon_n x+\varepsilon_ny_n)F(t_nw_n))\,{\rm d}x\nonumber\\
&+\liminf_{n\rightarrow\infty}\int_{\mathbb{R}^3}(\frac{Q_0}{4}-\frac{Q(\varepsilon_n x+\varepsilon_ny_n)}{2_s^{\ast}})|t_nw_n|^{2_s^{\ast}}\,{\rm d}x\nonumber\\
&\leq\liminf_{n\rightarrow\infty}\Big(I_{\varepsilon_n}(t_nw_n)-\frac{1}{4}\langle I_0'(t_nw_n),t_nw_n\rangle\Big)=\liminf_{n\rightarrow\infty}I_{\varepsilon_n}(t_nw_n)\nonumber\\
&\leq\liminf_{n\rightarrow\infty}\max_{t\geq0}I_{\varepsilon_n}(tw_n)=\liminf_{n\rightarrow\infty}I_{\varepsilon_n}(w_n)=\liminf_{n\rightarrow\infty}c_{\varepsilon_n}=m_0
\end{align}
which yields a contradiction. Thus, $\{\varepsilon y_{\varepsilon}\}$ is bounded in $\mathbb{R}^3$.
\end{proof}
For any $\varepsilon_n\rightarrow0$, the subsequence
$\{\varepsilon_ny_{\varepsilon_n}\}$ of the family $\{\varepsilon y_{\varepsilon}\}$ is such that $\varepsilon_ny_{\varepsilon_n}\rightarrow x^{\ast}$ in $\mathbb{R}^3$, we will prove that $x^{\ast}\in \Theta$.
\begin{lemma}\label{lem5-3}
$x^{\ast}\in\Theta$.
\end{lemma}
\begin{proof}
Set
\begin{align*}
I^{x^{\ast}}(v)=\frac{1}{2}\int_{\mathbb{R}^3}(|(-\Delta)^{\frac{s}{2}}v|^2+V(x^{\ast}) v^2)\,{\rm d}x+&\frac{1}{4}\int_{\mathbb{R}^3}\phi_v^tv^2\,{\rm d}x-\int_{\mathbb{R}^3} K(x^{\ast})F(v)\,{\rm d}x\\
&-\frac{1}{2_s^{\ast}}\int_{\mathbb{R}^3}Q(x^{\ast})|v(x)|^{2_s^{\ast}}\,{\rm d}x.
\end{align*}
Suppose that $V(x^{\ast})>V_0$. Taking the similar arguments of Lemma \ref{lem5-2} and replacing $I_{\infty}$ by $I^{x^{\ast}}$ in \eqref{equ5-4}, we can obtain a contradiction. Hence, $x^{\ast}\in\Theta_V$. By similar discussion, we can obtain a contradiction in the case $x^{\ast}\in\Theta_K\cap\Theta_Q$. Therefore, $x^{\ast}\in \Theta=\Theta_V\cap\Theta_K\cap\Theta_Q$ and the proof is completed.
\end{proof}

Since $w_{\varepsilon}$ is a positive ground state solution of problem \eqref{equ5-1-0} and $I_{\varepsilon}(w_{\varepsilon})=c_{\varepsilon}$ (using the invariance of translation),
by Lemma \ref{lem4-4}, it is easy to check that there exists $\widetilde{\varepsilon}>0$ such that for any $\varepsilon\in(0,\widetilde{\varepsilon})$, $w_{\varepsilon}$ is bounded in $H^s(\mathbb{R}^3)$ by a constant which is independent of $\varepsilon$. Hence, for any $\varepsilon_n\rightarrow0$, the subsequence $\{w_{\varepsilon_n}\}$ is bounded in $H^s(\mathbb{R}^3)$, we may assume that up to a subsequence, $w_{\varepsilon_n}\rightharpoonup w_0$ in $H^s(\mathbb{R}^3)$ and by Lemma \ref{lem5-2}, up to a subsequence, we also may assume that $\varepsilon_ny_{\varepsilon_n}\rightarrow x_0\in\Theta$ as $n\rightarrow\infty$.

\begin{lemma}\label{lem5-4}
$w_{\varepsilon_n}\rightarrow w_0$ in $H^s(\mathbb{R}^3)$ and $w_0$ is a positive ground state solution of the following problem
\begin{equation}\label{equ5-4-1}
(-\Delta)^su+V(x_0)u+\phi_{u}^tu=K(x_0)f(u)+Q(x_0)|u|^{2_s^{\ast}-2}u.
\end{equation}
\end{lemma}

\begin{proof}
Similar proof of Lemma \ref{lem5-2}, it is easy to check that $w_{\varepsilon_n}\rightarrow w_0$ in $H^s(\mathbb{R}^3)$. Therefore, $\langle  I_{\varepsilon_n}'(w_{\varepsilon_n}),\varphi\rangle=\langle  (I^{x_0})'(w_0),\varphi\rangle$, for any $\varphi\in H^s(\mathbb{R}^3)$. This means that $w_0$ is a nontrivial ground state solution of problem \eqref{equ5-4-1}. By the similar argument of Proposition \ref{pro4-1}, we can complete the proof.
\end{proof}

Moreover, we have the following vanishing estimate of $\{w_{\varepsilon}\}$ at infinity.
\begin{lemma}\label{lem5-5}
$\lim\limits_{|x|\rightarrow\infty}w_{\varepsilon}(x)=0$ uniformly in $\varepsilon\in(0,\widetilde{\varepsilon})$.
\end{lemma}
\begin{proof}
For any $\varepsilon_n\rightarrow0$, $w_n:=w_{\varepsilon_n}$ is a positive ground state solution of problem \eqref{equ5-1-0}, then $f(x,w_n):=K(\varepsilon_n x+\varepsilon_n y_{\varepsilon_n})f(w_n)+Q(\varepsilon_n x+\varepsilon_n y_{\varepsilon_n})|w_n|^{2_s^{\ast}-2}w_n-V(\varepsilon_n x+\varepsilon_n y_{\varepsilon_n})w_{\varepsilon_n}+\phi_{w_n}^tw_n\leq C(1+|w_n|^{2_s^{\ast}-1})$, where $C$ is independent of $n$ and $w_n$, by the estimate \eqref{equ4-5-0}, we have that $\|w_n\|_{\infty}\leq C\|w_n\|^{\alpha}\leq C$, where $C>0$ is a constant independent of $n$.

Now we borrow the ideas in \cite{AMi} to complete the proof. For this purpose, we rewrite problem \eqref{equ5-1-0} as follows
\begin{equation*}
(-\Delta)^sw_n+w_n=g_n(x)
\end{equation*}
where $g_n(x):=w_n+f(x,w_n)$.
Clearly, $g_n\in L^{\infty}(\mathbb{R}^3)$ and is uniformly bounded. From Lemma \ref{lem5-4}, for $n\rightarrow\infty$, we have that $g_n\rightarrow g_0$ in $L^r(\mathbb{R}^3)$ for $2\leq r\leq 2_s^{\ast}$, where $g_0=w_0+K(x_0)f(w_0)+Q(x_0)|w_0|^{2_s^{\ast}-2}w_0-V(x_0)w_0-\phi_{w_0}w_0$. Using some results found in \cite{FQT}, we see that
\begin{equation*}
w_n(x)=\int_{\mathbb{R}^3}\mathcal{K}(x-y)g_n(y)\,{\rm d}y
\end{equation*}
where $\mathcal{K}$ is a Bessel potential, which possesses the following properties:\\
$(\mathcal{K}_1)$ $\mathcal{K}$ is positive, radially symmetric and smooth in $\mathbb{R}^3\backslash\{0\}$;\\
$(\mathcal{K}_2)$ there exists a constant $C>0$ such that $\mathcal{K}(x)\leq \frac{C}{|x|^{3+2s}}$ for all $x\in\mathbb{R}^3\backslash\{0\}$;\\
$(\mathcal{K}_3)$ $\mathcal{K}\in L^\tau(\mathbb{R}^3)$ for $\tau\in[1,\frac{3}{3-2s})$.

We define two sets $A_{\delta}=\{y\in\mathbb{R}^3\,\,|\,\, |x-y|\geq\frac{1}{\delta}\}$ and $B_{\delta}=\{y\in\mathbb{R}^3\,\,|\,\, |x-y|<\frac{1}{\delta}\}$.
Hence,
\begin{equation*}
0\leq w_n(x)\leq \int_{\mathbb{R}^3}\mathcal{K}(x-y)|g_n(y)|\,{\rm d}y=\int_{A_{\delta}}\mathcal{K}(x-y)|g_n(y)|\,{\rm d}y+\int_{B_{\delta}}\mathcal{K}(x-y)|g_n(y)|\,{\rm d}y.
\end{equation*}
From the definition of $A_{\delta}$ and $(\mathcal{K}_2)$, we have that for all $n\in\mathbb{N}$,
\begin{equation*}
\int_{A_{\delta}}\mathcal{K}(x-y)|g_n(y)|\,{\rm d}y\leq C\delta^s\|g_n\|_{\infty}\int_{A_{\delta}}\frac{1}{|x-y|^{3+s}}\,{\rm d}y\leq C\delta^s\int_{A_{\delta}}\frac{1}{|x-y|^{3+s}}\,{\rm d}y:=C\delta^{2s}.
\end{equation*}

On the other hand, by H\"{o}lder's inequality and $(\mathcal{K}_3)$, we deduce that
\begin{align*}
&\int_{B_{\delta}}\mathcal{K}(x-y)|g_n(y)|\,{\rm d}y\leq\int_{B_{\delta}}\mathcal{K}(x-y)|g_n-g_0|\,{\rm d}y+\int_{B_{\delta}}\mathcal{K}(x-y)|g_0|\,{\rm d}y\\
&\leq\Big(\int_{B_{\delta}}\mathcal{K}^{\frac{6}{3+2s}}\,{\rm d}y\Big)^{\frac{3+2s}{6}}\Big(\int_{B_{\delta}}|g_n-g_0|^{\frac{6}{3-2s}}\,{\rm d}y\Big)^{\frac{3-2s}{6}}+
\Big(\int_{B_{\delta}}\mathcal{K}^2\,{\rm d}y\Big)^{\frac{1}{2}}\Big(\int_{B_{\delta}}|g_0|^2\,{\rm d}y\Big)^{\frac{1}{2}}\\
&\leq\Big(\int_{\mathbb{R}^3}\mathcal{K}^{\frac{6}{3+2s}}\,{\rm d}y\Big)^{\frac{3+2s}{6}}\Big(\int_{\mathbb{R}^3}|g_n-g_0|^{\frac{6}{3-2s}}\,{\rm d}y\Big)^{\frac{3-2s}{6}}+
\Big(\int_{\mathbb{R}^3}\mathcal{K}^2\,{\rm d}y\Big)^{\frac{1}{2}}\Big(\int_{B_{\delta}}|g_0|^2\,{\rm d}y\Big)^{\frac{1}{2}}
\end{align*}
where we have used the fact that $s>\frac{3}{4}$ so that $\frac{6}{3+2s}<\frac{3}{3-2s}$ and $2<\frac{3}{3-2s}$.

Since $\Big(\int_{B_{\delta}}|g_0|^2\,{\rm d}y\Big)^{\frac{1}{2}}\rightarrow0$ as $|x|\rightarrow+\infty$, thus, we deduce that there exist $n_0\in\mathbb{N}$ and $R_0>0$ independence of $\varepsilon>0$ such that
\begin{equation*}
\int_{B_{\delta}}\mathcal{K}(x-y)|g_n(y)|\,{\rm d}y\leq \delta,\quad \forall n\geq n_0\quad \text{and}\quad |x|\geq R_0.
\end{equation*}
Hence,
\begin{equation*}
\int_{\mathbb{R}^3}\mathcal{K}(x-y)|g_n(y)|\,{\rm d}y\leq C\delta^{2s}+\delta,\quad \forall n\geq n_0\quad \text{and}\quad |x|\geq R_0.
\end{equation*}
For each $n\in\{1,2,\cdots,n_0-1\}$, there exists $R_n>0$ such that $\Big(\int_{B_{\delta}}|g_n|^2\,{\rm d}y\Big)^{\frac{1}{2}}< \delta$ as $|x|\geq R_n$. Thus, for $|x|\geq R_n$, we have that
\begin{align*}
\int_{\mathbb{R}^3}\mathcal{K}(x-y)|g_n(y)|\,{\rm d}y&\leq C\delta^{2s}+\int_{B_{\delta}}\mathcal{K}(x-y)|g_n(y)|\,{\rm d}y\\
&\leq C\delta^{2s}+\|\mathcal{K}\|_2\Big(\int_{B_{\delta}}|g_n|^2\,{\rm d}y\Big)^{\frac{1}{2}}\leq C(\delta^{2s}+\delta)
\end{align*}
for each $n\in\{1,2,\cdots,n_0-1\}$.
Therefore, taking $R=\max\{R_0,R_1,\cdots, R_{n_0-1}\}$, we infer that for any $n\in\mathbb{N}$, there holds
\begin{equation*}
0\leq w_n(x)\leq\int_{\mathbb{R}^3}\mathcal{K}(x-y)|g_n(y)|\,{\rm d}y\leq C\delta^{2s}+\delta,\quad \text{for all}\quad |x|\geq R
\end{equation*}
implies that $\lim\limits_{|x|\rightarrow\infty}w_n(x)=0$ uniformly in $n\in\mathbb{N}$. As a result, the conclusion follows from the arbitrariness of $\varepsilon_n$.
\end{proof}

Now, we give the estimate of decay properties of solutions $u_{\varepsilon}$.
\begin{lemma}\label{lem5-6}
There exists a constant $C>0$ such that
\begin{equation*}
w_{\varepsilon}(x)\leq \frac{C}{1+|x|^{3+2s}}, \quad \text{for all}\,\, x\in\mathbb{R}^3\,\, \text{and}\,\, \varepsilon\in(0,\widetilde{\varepsilon}).
\end{equation*}
\end{lemma}

\begin{proof}
We borrow some ideas of the proof of Theorem 1.1 in \cite{HZ1} to give the proof of Lemma \ref{lem5-6}.  By Lemma 4.2 and Lemma 4.3 in \cite{FQT}, by scaling, there exists a continuous function $W$ such that
\begin{equation}\label{equ5-5}
0<W(x)\leq \frac{C}{1+|x|^{3+2s}}
\end{equation}
and
\begin{equation*}
(-\Delta)^sW+\frac{V_0}{2}W=0 \quad \text{on}\,\,\mathbb{R}^3\backslash B_R(0).
\end{equation*}
for some suitable $R>0$. By Lemma \ref{lem5-5}, there exists $R_1>0$ (we can choose $R_1>R$) large enough such that
\begin{align*}
(-\Delta)^sw_{\varepsilon}+\frac{V_0}{2}w_{\varepsilon}&=(-\Delta)^sw_{\varepsilon}+V(\varepsilon x)w_{\varepsilon}+(\frac{V_0}{2}-V(\varepsilon x))w_{\varepsilon}\\
&=K(\varepsilon x)f(w_{\varepsilon})+Q(\varepsilon)w_{\varepsilon}^{2_s^{\ast}-1}-\phi_{w_{\varepsilon}}^tw_{\varepsilon}+(\frac{V_0}{2}-V(\varepsilon x))w_{\varepsilon}\\
&\leq K_0f(w_{\varepsilon})+Q_0w_{\varepsilon}^{2_s^{\ast}-1}-\frac{V_0}{2}w_{\varepsilon}\leq0
\end{align*}
for any $x\in\mathbb{R}^3\backslash B_{R_1}(0)$. Therefore, we have obtained that
\begin{equation}\label{equ5-6}
(-\Delta)^sW(x)+\frac{V_0}{2}W\geq(-\Delta)^sw_{\varepsilon}+\frac{V_0}{2}w_{\varepsilon}\quad \text{on}\,\, \mathbb{R}^3\backslash B_{R_1}(0).
\end{equation}

Let $\mathbb{A}=\inf\limits_{B_{R_1}(0)}W>0$, $Z_{\varepsilon}(x)=(\mathbb{B}+1)W-\mathbb{A}w_{\varepsilon}$, where $\mathbb{B}=\sup\limits_{0<\varepsilon<\widetilde{\varepsilon}}\|w_{\varepsilon}\|_{\infty}\leq C<\infty$, where $C$ is some positive constant independent of $\varepsilon$. We claim that $Z_{\varepsilon}(x)\geq0$ for all $x\in\mathbb{R}^3$ and $\varepsilon\in(0,\widetilde{\varepsilon})$. If the claim is true, we have that
\begin{equation*}
w_{\varepsilon}(x)\leq\frac{\mathbb{B}+1}{\mathbb{A}}W\leq \frac{C}{1+|x|^{3+2s}}\quad \text{for all}\,\, x\in\mathbb{R}^3\quad \text{and}\quad \varepsilon\in(0,\widetilde{\varepsilon}_1)
\end{equation*}
and the conclusion is proved.

Suppose by contradiction that there exist $\varepsilon_0\in(0,\widetilde{\varepsilon}_1)$ and $x_{\varepsilon_0}^n\in\mathbb{R}^3$ such that
\begin{equation}\label{equ5-7}
\inf_{x\in\mathbb{R}^3}Z_{\varepsilon_0}(x)=\lim_{n\rightarrow\infty}Z_{\varepsilon_0}(x_{\varepsilon_0}^n)<0.
\end{equation}
Since $\lim\limits_{|x|\rightarrow\infty}W(x)=\lim\limits_{|x|\rightarrow\infty}w_{\varepsilon}(x)=0$ uniformly for $\varepsilon\in(0,\widetilde{\varepsilon}_1)$, then $\lim\limits_{|x|\rightarrow\infty}Z_{\varepsilon_0}(x)=0$. Hence, the sequence $\{x_{\varepsilon_0}^n\}$ is bounded and then up to a subsequence, we may assume that $x_{\varepsilon_0}^n\rightarrow \widetilde{x}_{\varepsilon_0}$. From \eqref{equ5-7}, we have that
\begin{equation}\label{equ5-8}
\inf_{x\in\mathbb{R}^3}Z_{\varepsilon_0}(x)=Z_{\varepsilon_0}(\widetilde{x}_{\varepsilon_0})<0.
\end{equation}
By \eqref{equ5-8}, we get
\begin{align*}
(-\Delta)^sZ_{\varepsilon_0}(\widetilde{x}_{\varepsilon_0})&+\frac{V_0}{2}Z_{\varepsilon_0}(\widetilde{x}_{\varepsilon_0})=\frac{V_0}{2}Z_{\varepsilon_0}(\widetilde{x}_{\varepsilon_0})\\
&-\frac{1}{2}C(3,s)
\int_{\mathbb{R}^3}\frac{Z_{\varepsilon_0}(\widetilde{x}_{\varepsilon_0}+y)+Z_{\varepsilon_0}(\widetilde{x}_{\varepsilon_0}-y)-2Z_{\varepsilon_0}(\widetilde{x}_{\varepsilon_0})}{|x-y|^{3+2s}}\,{\rm d}y<0.
\end{align*}
Note that $Z_{\varepsilon_0}(x)\geq \mathbb{A}\mathbb{B}+W-\mathbb{A}\mathbb{B}>0$ on $B_{R_1}(0)$, hence $\widetilde{x}_{\varepsilon_0}\in\mathbb{R}^3\backslash B_{R_1}(0)$. From \eqref{equ5-6}, by computation, we have that
\begin{align*}
(-\Delta)^sZ_{\varepsilon_0}(\widetilde{x}_{\varepsilon_0})+\frac{V_0}{2}Z_{\varepsilon_0}(\widetilde{x}_{\varepsilon_0})&=\Big[(\mathbb{B}+1)\Big((-\Delta)^sW+\frac{V_0}{2}W\Big)
\\
&-\mathbb{A}\Big((-\Delta)^sw_{\varepsilon_0}+\frac{V_0}{2}w_{\varepsilon_0}\Big)\Big]\Big|_{x=\widetilde{x}_{\varepsilon_0}}\geq0
\end{align*}
which is a contradiction. Thus, the claim holds true and the proof is completed.

\end{proof}

{\bf Proof of Theorem \ref{thm1-1}.}
$(i)$ Taking $\varepsilon_0=\min\{\widetilde{\varepsilon}_0, \widetilde{\varepsilon},\widehat{\varepsilon}\}$. By Proposition \ref{pro4-2}, for each $\varepsilon\in(0,\varepsilon_0)$, problem \eqref{equ2-4} has at least a positive ground state solution $v_{\varepsilon}$. Hence, let $\phi_{\varepsilon}=\phi_{v_{\varepsilon}}^t$, then $(v_{\varepsilon},\phi_{\varepsilon})$ is a positive solution for system \eqref{main} for $\varepsilon\in(0,\varepsilon_0)$.

$(ii)$ From Proposition \ref{lem4-2} and Lemma \ref{lem5-5}, there exists $R>0$ such that the global maximum point of $w_{\varepsilon}(x)=v_{\varepsilon}(x+y_{\varepsilon})$, denoted by $\widetilde{x}_{\varepsilon}$, is located in $B_{R}(0)$. Thus, the global maximum point of $v_{\varepsilon}$ is given by $x_{\varepsilon}=\widetilde{x}_{\varepsilon}+y_{\varepsilon}$. Observing that $u_{\varepsilon}(x)=v_{\varepsilon}(x/\varepsilon)$, then we have that $(u_{\varepsilon}(x),\phi_{u_{\varepsilon}}^t(x))$ is a positive ground state solution of system \eqref{main} and $u_{\varepsilon}$ has a global maximum point $z_{\varepsilon}=\varepsilon x_{\varepsilon}$. It follows from Lemma \ref{lem5-2}, Lemma \ref{lem5-3} and $\widetilde{x}_{\varepsilon}\in B_R(0)$ that $\lim\limits_{\varepsilon\rightarrow0}V(z_{\varepsilon})=V_0$, $\lim\limits_{\varepsilon\rightarrow0}K(z_{\varepsilon})=K_0$ and $\lim\limits_{\varepsilon\rightarrow0}Q(z_{\varepsilon})=Q_0$. Moreover, in view of Lemma \ref{lem5-4}, we see that $z_{\varepsilon}\rightarrow x_0$ if $\varepsilon\rightarrow0$, then $u_{\varepsilon}(\varepsilon x+z_{\varepsilon})$ converges to $u$ and $u$ is a solution for problem \eqref{equ5-4-1}. As a result, $(u,\phi)$ is a solution of system \eqref{equ1-2}.

$(iii)$ By Lemma \ref{lem5-6}, we have that
\begin{align*}
u_{\varepsilon}(x)=v_{\varepsilon}(x/\varepsilon)=w_{\varepsilon}(x/\varepsilon-y_{\varepsilon})&=w_{\varepsilon}(\frac{x+\varepsilon\widetilde{x}_{\varepsilon}-z_{\varepsilon}}{\varepsilon})
\leq\frac{C}{1+|\frac{x+\varepsilon\widetilde{x}_{\varepsilon}-z_{\varepsilon}}{\varepsilon}|^{3+2s}}\\
&\leq\frac{C\varepsilon^{3+2s}}{\varepsilon^{3+2s}+|x-z_{\varepsilon}|^{3+2s}-\varepsilon^{3+2s}R^{3+2s}}\\
&:=\frac{C\varepsilon^{3+2s}}{C_0\varepsilon^{3+2s}+|x-z_{\varepsilon}|^{3+2s}},
\end{align*}
where $C_0=1-R^{3+2s}$.

\section{Nonexistence of ground states}
In this section, our goal is to show the nonexistence of ground state solution to system \eqref{main}, that is, for each $\varepsilon>0$, the ground energy $c_{\varepsilon}$ is not attained.
\begin{lemma}\label{lem6-1}
Assume the continuous functions $V(x)$, $K(x)$, $Q(x)$ satisfies $(\mathcal{H})$ and $(f_0)$--$(f_2)$ hold. Then for each $\varepsilon>0$, $c_{\varepsilon}=m_{\infty}$.
\end{lemma}
\begin{proof}
Noting that $H_{\varepsilon}=H^s(\mathbb{R}^3)$, for any $\varepsilon>0$. By $(\mathcal{H})$, we have $I_{\infty}(u)\leq I_{\varepsilon}(u)$, for all $u\in H^s(\mathbb{R}^3)$. By $(ii)$ of Lemma \ref{lem3-1}, for each $u\in \mathcal{N}_{\infty}$, there exists $t_u>0$ such that $t_uu\in\mathcal{N}_{\varepsilon}$. Hence, for each $u\in \mathcal{N}_{\infty}$, we have that
\begin{align*}
0<m_{\infty}&=\inf_{u\in\mathcal{N}_{\infty}}I_{\infty}(u)\leq\max_{t\geq0}I_{\infty}(tu)\leq\max_{t\geq0}I_{\varepsilon}(tu)=I_{\varepsilon}(t_uu).
\end{align*}
By $(i)$ of Lemma \ref{lem3-1}, one has
\begin{equation*}
0<m_{\infty}\leq\inf_{u\in\mathcal{N}_{\infty}}I_{\varepsilon}(t_uu)=\inf_{v\in\mathcal{N}_{\varepsilon}}I_{\varepsilon}(v)=c_{\varepsilon}.
\end{equation*}
So, it suffices to show that $c_{\varepsilon}\leq m_{\infty}$.

By Proposition \ref{pro4-1}, there exists $u_{\infty}\in\mathcal{N}_{\infty}$ is a ground state solution of \eqref{equ2-5} with $\nu=V_{\infty}$, $\kappa=K_{\infty}$ and $\mu=Q_{\infty}$. Set $e_n(x)=u_{\infty}(x-y_n)$ where $y_n\in\mathbb{R}^3$ and $|y_n|\rightarrow\infty$ as $n\rightarrow\infty$. Then, there exists $t_n(e_n)>0$ such that $t_ne_n\in\mathcal{N}_{\varepsilon}$, that is,
\begin{align}\label{equ6-1}
t_n^2\int_{\mathbb{R}^3}&\Big(|(-\Delta)^{\frac{s}{2}}u_{\infty}|^2+V(\varepsilon x+\varepsilon y_n)|u_{\infty}|^2\Big)\,{\rm d}x+t_n^4\int_{\mathbb{R}^3}\phi_{u_{\infty}}^tu_{\infty}^2\,{\rm d}x\nonumber\\
&=\int_{\mathbb{R}^3}K(\varepsilon x+\varepsilon y_n)f(t_nu_{\infty})t_nu_{\infty}\,{\rm d}x+t_n^{2_s^{\ast}}\int_{\mathbb{R}^3}Q(\varepsilon x+\varepsilon y_n)|u_{\infty}|^{2_s^{\ast}}\,{\rm d}x.
\end{align}
This implies that $t_n$ cannot converge to zero and infinity. Suppose that $t_n\rightarrow t$ as $n\rightarrow\infty$. Letting $n\rightarrow\infty$ in \eqref{equ6-1}, we have that
\begin{align*}
\int_{\mathbb{R}^3}&\Big(|(-\Delta)^{\frac{s}{2}}u_{\infty}|^2+V_{\infty}|u_{\infty}|^2\Big)\,{\rm d}x+t^2\int_{\mathbb{R}^3}\phi_{u_{\infty}}^tu_{\infty}^2\,{\rm d}x\\
&=\int_{\mathbb{R}^3}K_{\infty}\frac{f(tu_{\infty})u_{\infty}}{t}\,{\rm d}x+t^{2_s^{\ast}-2}\int_{\mathbb{R}^3}Q_{\infty}|u_{\infty}|^{2_s^{\ast}}\,{\rm d}x.
\end{align*}
In view of $u_{\infty}\in\mathcal{N}_{\infty}$, we conclude that $t=1$. Since
\begin{align}\label{equ6-2}
c_{\infty}&\leq I_{\varepsilon}(t_n e_n)=I_{\infty}(t_nu_{\infty})+\frac{t_n^2}{2}\int_{\mathbb{R}^3}(V(\varepsilon x+\varepsilon y_n)-V_{\infty})|u_{\infty}|^2\,{\rm d}x\nonumber\\
&-\int_{\mathbb{R}^3}(K(\varepsilon x+\varepsilon y_n)-K_{\infty})F(t_nu_{\infty})\,{\rm d}x-\frac{t_n^{2_s^{\ast}}}{2_s^{\ast}}\int_{\mathbb{R}^3}(Q(\varepsilon x+\varepsilon y_n)-Q_{\infty})|u_{\infty}|^{2_s^{\ast}}\,{\rm d}x.
\end{align}
By the assumption on $V$, for any $\delta>0$, there exists $R>0$ such that
\begin{equation*}
\int_{|x|\geq\frac{R}{\varepsilon}}|V(\varepsilon x+\varepsilon y_n)-V_{\infty}||u_{\infty}|^2\,{\rm d}x\leq \delta.
\end{equation*}
By $|y_n|\rightarrow\infty$, according to Lebesgue¡¯s theorem, we have
\begin{equation*}
\lim_{n\rightarrow\infty}\int_{|x|<\frac{R}{\varepsilon}}|V(\varepsilon x+\varepsilon y_n)-V_{\infty}||u_{\infty}|^2\,{\rm d}x=0.
\end{equation*}
Thus,
\begin{equation*}
\lim_{n\rightarrow\infty}\int_{\mathbb{R}^3}|V(\varepsilon x+\varepsilon y_n)-V_{\infty}||u_{\infty}|^2\,{\rm d}x=0.
\end{equation*}
Similarly, we deduce that
\begin{equation*}
\lim_{n\rightarrow\infty}\int_{\mathbb{R}^3}|K(\varepsilon x+\varepsilon y_n)-K_{\infty}F(u_{\infty})\,{\rm d}x=0,\quad\lim_{n\rightarrow\infty}\int_{\mathbb{R}^3}|Q(\varepsilon x+\varepsilon y_n)-Q_{\infty}|u_{\infty}|^{2_s^{\ast}}\,{\rm d}x=0.
\end{equation*}
Therefore, using $t_n\rightarrow1$ and letting $n\rightarrow\infty$ in \eqref{equ6-2}, we infer that $c_{\varepsilon}\leq m_{\infty}$ and the proof is completed.
\end{proof}

{\bf Proof of Theorem \ref{thm1-2}.} Suppose by contradiction that there exist some $\varepsilon_0>0$ and $u_0\in\mathcal{N}_{\varepsilon_0}$ such that $I_{\varepsilon_0}(u_0)=c_{\varepsilon_0}$. From Lemma \ref{lem6-1}, $c_{\varepsilon_0}=m_{\infty}$. In view of Lemma \ref{lem3-1}, there exists $t_0>0$ such that $t_0u_0\in\mathcal{N}_{\infty}$. Thus, using the fact $u_0\in\mathcal{N}_{\varepsilon_0}$, we have that
\begin{equation*}
m_{\infty}\leq I_{\infty}(t_0u_0)\leq I_{\varepsilon_0}(t_0u_0)\leq \max_{t\geq0}I_{\varepsilon_0}(tu_0)=I_{\varepsilon_0}(u_0)=c_{\varepsilon_0}=m_{\infty}.
\end{equation*}
Thus, $m_{\infty}=I_{\infty}(t_0u_0)=I_{\varepsilon_0}(t_0u_0)$. However,
\begin{align*}
I_{\infty}(t_0u_0)&=I_{\varepsilon_0}(t_0u_0)+\frac{t_0^2}{2}\int_{\mathbb{R}^3}(V_{\infty}-V(\varepsilon_0 x))|u_0|^2\,{\rm d}x\nonumber\\
&+\int_{\mathbb{R}^3}(K(\varepsilon_0 x)-K_{\infty})F(t_0u_0)\,{\rm d}x+\frac{t_0^{2_s^{\ast}}}{2_s^{\ast}}\int_{\mathbb{R}^3}(Q(\varepsilon_0 x)-Q_{\infty})|u_0|^{2_s^{\ast}}\,{\rm d}x\\
&<I_{\varepsilon_0}(t_0u_0)
\end{align*}
which is a contradiction. The proof is completed.

\section{Appendix}
In this section, By a similar argument of Section 4 in \cite{SV}, we give the estimate of \eqref{equ3-4}, \eqref{equ3-5} and \eqref{equ3-6}. We prove them in the general case.

Let $\mathcal{S}_s=\mathcal{S}_s(\widetilde{u})$, $\widetilde{u}=\kappa(\mu^2+|x-x_0|^2)^{-\frac{N-2s}{2}}$, $x\in\mathbb{R}^N$ with $\kappa\in\mathbb{R}\backslash\{0\}$, $\mu>0$ and $x_0\in\mathbb{R}^N$ are fixed constant. Using scaling and translation transform, we see that
\begin{equation*}
\mathcal{S}_s=\mathcal{S}_s(\widetilde{u})=\mathcal{S}_s(\bar{u}), \quad \bar{u}(x)=\frac{\kappa}{(1+|x|^2)^{\frac{N-2s}{2}}}, x\in\mathbb{R}^N.
\end{equation*}
Set
\begin{equation*}
u^{\ast}(x)=\frac{\kappa}{(\varepsilon^2+|x-x_0/\varepsilon|^2)^{\frac{N-2s}{2}}},\,\, U_{\varepsilon}(x)=\varepsilon^{\frac{N-2s}{2}}u^{\ast}(x),\quad x\in\mathbb{R}^N,\,\,\varepsilon>0.
\end{equation*}
It is easy to check that
\begin{equation*}
\|(-\Delta)^{\frac{s}{2}}U_{\varepsilon}\|_2^2=\|(-\Delta)^{\frac{s}{2}}\bar{u}\|_2^2=\mathcal{S}_s\|\bar{u}\|_{2_s^{\ast}}^2=\mathcal{S}_s\|U_{\varepsilon}\|_{2_s^{\ast}}^2.
\end{equation*}
Let $\eta\in C_0^{\infty}(\mathbb{R}^N)$ be such that $0\leq \eta\leq1$, $\eta\Big|_{B_R(0)}=1$, ${\rm supp}\eta\subset B_{2R}(0)$. Set $u_{\varepsilon}(x)=\eta( x-x_0/\varepsilon)U_{\varepsilon}(x)$, for any $\varepsilon>0$ and $x\in\mathbb{R}^N$. By computation, we have the following estimate for $u_{\varepsilon}$.

{\bf Lemma A.1}.
$(i)$ Let $\rho>0$. If $x\in \mathbb{R}^N\backslash B_{\rho}(x_0/\varepsilon)$, then
\begin{equation*}
|u_{\varepsilon}(x)|\leq|U_{\varepsilon}(x)|\leq C_1\varepsilon^{\frac{N-2s}{2}},\quad |\nabla u_{\varepsilon}(x)|\leq C_2\varepsilon^{\frac{N-2s}{2}}
\end{equation*}
for any $\varepsilon>0$ and for some positive constants $C_1$ and $C_2$, possibly depending on $N,s,\rho$.\\
$(ii)$ For any $x\in\mathbb{R}^N$ and $y\in\mathbb{R}^N\backslash B_R(x_0/\varepsilon)$, with $|x-y|\leq\frac{R}{2}$,
\begin{equation*}
|u_{\varepsilon}(x)-u_{\varepsilon}(y)|\leq C_3\varepsilon^{\frac{N-2s}{2}}|x-y|.
\end{equation*}
$(iii)$ For any $ x,  y\in\mathbb{R}^N\backslash B_R(x_0/\varepsilon)$,
\begin{equation*}
|u_{\varepsilon}(x)-u_{\varepsilon}(y)|\leq C_4\varepsilon^{\frac{N-2s}{2}}\min\{1,|x-y|\}.
\end{equation*}
for any $\varepsilon>0$ and for some positive constants $C_3$ and $C_4$, possibly depending on $N,s,\rho$.
\begin{proof}
$(i)$ By the definition of $u_{\varepsilon}$, for any $x\in\mathbb{R}^N\backslash B_{\rho}(x_0/\varepsilon)$, we have
\begin{align*}
|u_{\varepsilon}(x)|\leq |U_{\varepsilon}(x)|&\leq C\varepsilon^{\frac{N-2s}{2}}\frac{1}{(\varepsilon^2+|x-x_0/\varepsilon|^2)^{\frac{N-2s}{2}}}\leq C\varepsilon^{\frac{N-2s}{2}}\frac{1}{(\varepsilon^2+\rho^2)^{\frac{N-2s}{2}}}\\
&\leq C_1\varepsilon^{\frac{N-2s}{2}}
\end{align*}
and
\begin{align*}
\Big|\nabla u_{\varepsilon}(x)\Big|&=\Big|\nabla (\eta(x-x_0/\varepsilon))U_{\varepsilon}(x)+\eta(x-x_0/\varepsilon)\nabla U_{\varepsilon}(x)\Big|\\
&\leq C\Big(|U_{\varepsilon}(x)|+\varepsilon^{\frac{N-2s}{2}}\frac{| x-x_0/\varepsilon|}{(\varepsilon^2+|x-x_0/\varepsilon|^2)^{\frac{N+2-2s}{2}}}\Big)\\
&\leq C\Big(|U_{\varepsilon}(x)|+\varepsilon^{\frac{N-2s}{2}}\frac{|x-x_0/\varepsilon|^2}{\rho(\varepsilon^2+|x-x_0/\varepsilon|^2)^{\frac{N+2-2s}{2}}}\Big)\\
&\leq C_2\varepsilon^{\frac{N-2s}{2}}.
\end{align*}

$(ii)$ Let $x\in\mathbb{R}^N$, $y\in \mathbb{R}^N\backslash B_{R}(x_0/\varepsilon)$, with $|x-y|\leq\frac{R}{2}$, and let $\xi$ be any point on the segment joining $x$ and $y$. Then
$\xi=tx+(1-t)y$, $t\in[0,1]$. Hence
\begin{equation*}
|\xi-x_0/\varepsilon|\geq|x-x_0/\varepsilon|-t|x-y|\geq R-\frac{R}{2}=\frac{R}{2}.
\end{equation*}
Taking $\rho=\frac{R}{2}$ in $(i)$, implies that $|\nabla u_{\varepsilon}(\xi)|\leq C\varepsilon^{\frac{N-2s}{2}}$ and so by Taylor expansion,
\begin{equation}\label{A-2}
|u_{\varepsilon}(x)-u_{\varepsilon}(y)|\leq C\varepsilon^{\frac{N-2s}{2}}|x-y|, \quad \text{for any}\,\, \varepsilon>0.
\end{equation}

$(iii)$ Let $ x, y\in\mathbb{R}^N\backslash B_R(x_0/\varepsilon)$. If $|x-y|\leq\frac{R}{2}$, then the second conclusion of $(ii)$ follows from \eqref{A-2}. If $|x-y|>\frac{R}{2}$, using the conclusion $(i)$, we have that
\begin{equation*}
|u_{\varepsilon}(x)-u_{\varepsilon}(y)|\leq |u_{\varepsilon}(x)|+|u_{\varepsilon}(y)|\leq C\varepsilon^{\frac{N-2s}{2}}.
\end{equation*}
The proof is completed.
\end{proof}

{\bf Lemma A.2}. Let $s\in(0,1)$ and $n>2s$. Then the following estimate holds true
\begin{equation*}
\int_{\mathbb{R}^N}|(-\Delta)^{\frac{s}{2}}u_{\varepsilon}|^2\,{\rm d}x\leq \mathcal{S}_s^{\frac{N}{2s}}+O(\varepsilon^{N-2s})
\end{equation*}
as $\varepsilon\rightarrow0$.

\begin{proof}
By the definition of $u_{\varepsilon}$, we can rewrite $\|(-\Delta)^{\frac{s}{2}}u_{\varepsilon}\|_2^2$ as follows
\begin{align*}
\|(-\Delta)^{\frac{s}{2}}u_{\varepsilon}\|_2^2&=\frac{1}{C(N,s)}\int_{\mathbb{R}^{2N}}\frac{|u_{\varepsilon}(x)-u_{\varepsilon}(y)|^2}{|x-y|^{N=2s}}\,{\rm d}x\,{\rm d}y\\
&=\frac{1}{C(N,s)}\Big(\int_{B_R(x_0/\varepsilon)\times B_R(x_0/\varepsilon)}\frac{|U_{\varepsilon}(x)-U_{\varepsilon}(y)|^2}{|x-y|^{N+2s}}\,{\rm d}x\,{\rm d}y\\
&+\int_{\mathbb{R}^N\backslash B_R(x_0/\varepsilon)\times \mathbb{R}^N\backslash B_R(x_0/\varepsilon)}\frac{|u_{\varepsilon}(x)-u_{\varepsilon}(y)|^2}{|x-y|^{N+2s}}\,{\rm d}x\,{\rm d}y\\
&+2\int_{\mathbb{D}}\frac{|u_{\varepsilon}(x)-u_{\varepsilon}(y)|^2}{|x-y|^{N+2s}}\,{\rm d}x\,{\rm d}y+2\int_{\mathbb{E}}\frac{|u_{\varepsilon}(x)-u_{\varepsilon}(y)|^2}{|x-y|^{N+2s}}\,{\rm d}x\,{\rm d}y\Big)
\end{align*}
where the sets $\mathbb{D}$ and $\mathbb{E}$ are given by
\begin{equation*}
\mathbb{D}=\Big\{(x,y)\in\mathbb{R}^{2N}\,\,\Big|\,\, x\in B_R(x_0/\varepsilon),\,\, y\in\mathbb{R}^N\backslash B_R(x_0/\varepsilon)\,\, \text{and}\,\, |x-y|>\frac{R}{2} \Big\}
\end{equation*}
and
\begin{equation*}
\mathbb{E}=\Big\{(x,y)\in\mathbb{R}^{2N}\,\,\Big|\,\, x\in B_R(x_0/\varepsilon),\,\, y\in\mathbb{R}^N\backslash B_R(x_0/\varepsilon)\,\, \text{and}\,\, |x-y|\leq\frac{R}{2} \Big\}.
\end{equation*}
Similar to the proof of Proposition 21 in \cite{SV}, using Lemma A.1, it is easy to show that, as $\varepsilon\rightarrow0$, there hold
\begin{equation}\label{A-2}
\int_{\mathbb{E}}\frac{|u_{\varepsilon}(x)-u_{\varepsilon}(y)|^2}{|x-y|^{N+2s}}\,{\rm d}x\,{\rm d}y\leq O(\varepsilon^{N-2s})
\end{equation}
and
\begin{equation}\label{A-3}
\int_{\mathbb{R}^N\backslash B_R(x_0/\varepsilon)\times \mathbb{R}^N\backslash B_R(x_0/\varepsilon)}\frac{|u_{\varepsilon}(x)-u_{\varepsilon}(y)|^2}{|x-y|^{N+2s}}\,{\rm d}x\,{\rm d}y\leq O(\varepsilon^{N-2s}).
\end{equation}
To complete the proof, checking the proof of Proposition 21 in \cite{SV}, it is sufficient to verify the following estimate holds
\begin{equation*}
\int_{\mathbb{D}}\frac{|U_{\varepsilon}(x)||U_{\varepsilon}(y)|}{|x-y|^{N+2s}}\,{\rm d}x\,{\rm d}y\leq O(\varepsilon^{N-2s})
\end{equation*}
as $\varepsilon\rightarrow0$.

In fact, by $(i)$ of Lemma A.1, making the change of variables $\xi=x-y$, we deduce that
\begin{align*}
\int_{\mathbb{D}}\frac{U_{\varepsilon}(x)||U_{\varepsilon}(y)|}{|x-y|^{N+2s}}\,{\rm d}x\,{\rm d}y&\leq C\varepsilon^{N-2s}\int_{\mathbb{D}}\frac{1}{(\varepsilon^2+|x-x_0/\varepsilon|^2)^{\frac{N-2s}{2}}}\frac{1}{|x-y|^{N+2s}}\,{\rm d}x\,{\rm d}y\\
&=C\varepsilon^{N-2s}\int_{\mathbb{D}_{\varepsilon}}\frac{1}{(1+|x|^2)^{\frac{N-2s}{2}}}\frac{1}{|x-y|^{N+2s}}\,{\rm d}x\,{\rm d}y\\
&\leq C\varepsilon^{N-2s}\int_{B_{R/\varepsilon}(0)}\frac{1}{(1+|x|^2)^{\frac{N-2s}{2}}}\,{\rm d}x\int_{|\xi|>\frac{R}{2\varepsilon}}\frac{1}{|\xi|^{N+2s}}\,{\rm d}\xi\\
&\leq C\varepsilon^{N-2s}\varepsilon^{-2s}\varepsilon^{2s}=O(\varepsilon^{N-2s})
\end{align*}
where $\mathbb{D}_{\varepsilon}$ is defined as
\begin{equation*}
\mathbb{D}_{\varepsilon}=\Big\{(x,y)\in\mathbb{R}^{2N}\,\,\Big|\,\,x\in B_{R/\varepsilon}(0), y\in \mathbb{R}^N\backslash B_{R/\varepsilon}(0), |x-y|>\frac{R}{2\varepsilon}\Big\}.
\end{equation*}
Finally, combing with \eqref{A-2}, \eqref{A-3}, we conclude that
\begin{align*}
\|(-\Delta)^{\frac{s}{2}}u_{\varepsilon}\|_2^2&\leq\frac{1}{C(N,s)}\int_{B_R(x_0/\varepsilon)\times B_R(x_0/\varepsilon)}\frac{|U_{\varepsilon}(x)-U_{\varepsilon}(y)|^2}{|x-y|^{N+2s}}\,{\rm d}x\,{\rm d}y+O(\varepsilon^{N-2s})\\
&\leq \frac{1}{C(N,s)}\int_{\mathbb{R}^{2N}}\frac{|U_{\varepsilon}(x)-U_{\varepsilon}(y)|^2}{|x-y|^{N+2s}}\,{\rm d}x\,{\rm d}y+O(\varepsilon^{N-2s})\\
&=\|(-\Delta)^{\frac{s}{2}}\bar{u}\|_2^2+O(\varepsilon^{N-2s})
\end{align*}
as $\varepsilon\rightarrow0$.
\end{proof}

{\bf Lemma A.3.}
Let $s\in(0,1)$ and $N>2s$, $p\in(1,2_s^{\ast})$. Then the following estimates hold.
\begin{equation}\label{A-4}
\int_{\mathbb{R}^N}|u_{\varepsilon}(x)|^p\,{\rm d}x=\left\{
          \begin{array}{ll}
           O(\varepsilon^{N-\frac{N-2s}{2}p}),&\hbox{$p>\frac{N}{N-2s}$,} \\
O(\varepsilon^{N-\frac{N-2s}{2}p}|\log\varepsilon|), & \hbox{$p=\frac{N}{N-2s}$,} \\
O(\varepsilon^{\frac{N-2s}{2}p}), & \hbox{$1<p<\frac{N}{N-2s}$.}
          \end{array}
        \right.
\end{equation}
and
\begin{equation}\label{A-5}
\int_{\mathbb{R}^N}|u_{\varepsilon}(x)|^{2_s^{\ast}}\,{\rm d}x=\int_{\mathbb{R}^N}|\bar{u}(x)|^{2_s^{\ast}}\,{\rm d}x+O(\varepsilon^N)
\end{equation}
as $\varepsilon\rightarrow0$.

\begin{proof}
By the definition of $u_{\varepsilon}$, we have
\begin{align*}
  \int_{\mathbb{R}^N}|u_{\varepsilon}(x)|^p\,{\rm d}x&=\int_{B_{R}(x_0/\varepsilon)}|U_{\varepsilon}(x)|^p\,{\rm d}x+\int_{B_{2R}(x_0/\varepsilon)\backslash B_R(x_0/\varepsilon)}|\psi(x-x_0/\varepsilon)U_{\varepsilon}(x)|^p\,{\rm d}x \\
  &\geq C\varepsilon^{\frac{N-2s}{2}p}\int_{B_{R}(x_0/\varepsilon)}\frac{1}{(\varepsilon^2+|x-x_0/\varepsilon|^2)^{\frac{(N-2s)p}{2}}}\,{\rm d}x\\
&\geq C\varepsilon^{N-\frac{N-2s}{2}p}\int_{\delta}^{R/\varepsilon}\frac{1}{(1+r^2)^{\frac{(N-2s)p}{2}}}r^{N-1}\,{\rm d}r
\end{align*}
for any $0<\delta<R/\varepsilon$. Therefore,
\begin{align*}
   \int_{\mathbb{R}^3}|u_{\varepsilon}(x)|^p\,{\rm d}x&\geq C\varepsilon^{N-\frac{N-2s}{2}p}\int_{\delta}^{R/\varepsilon}\frac{1}{(1+r^2)^{\frac{(N-2s)p}{2}}}r^{N-1}\,{\rm d}r  \\
  &\geq C\varepsilon^{N-\frac{N-2s}{2}p}\int_{\delta}^{R/\varepsilon}\frac{1}{r^{N(p-1)-2s p+1}}dr \\
  &=\left\{
          \begin{array}{ll}
            -C_{1,s}\varepsilon^{\frac{Np}{2}-sp}+C_{2,s}\varepsilon^{N-\frac{(N-2s)p}{2}}, & \hbox{$\frac{p-1}{p}N>2s$,} \\
            C_{1,s}\varepsilon^{N-\frac{(N-2s)p}{2}}|\log\varepsilon|+C_{2,s}\varepsilon^{N-\frac{N-2s}{2}p}, & \hbox{$\frac{p-1}{p}N=2s$,} \\
            C_{1,s}\varepsilon^{\frac{Np}{2}-sp}-C_{2,s}\varepsilon^{N-\frac{(N-2s)p}{2}}, & \hbox{$\frac{p-1}{p}N<2s$}
          \end{array}
        \right.\\
   &=\left\{
          \begin{array}{ll}
           O(\varepsilon^{N-\frac{N-2s}{2}p}),&\hbox{$p>\frac{N}{N-2s}$,} \\
O(\varepsilon^{N-\frac{N-2s}{2}p}|\log\varepsilon|), & \hbox{$p=\frac{N}{N-2s}$,} \\
O(\varepsilon^{\frac{N-2s}{2}p}), & \hbox{$1<p<\frac{N}{N-2s}$.}
          \end{array}
        \right.
\end{align*}
On the other hand, arguing in the same way, we infer that
\begin{align*}
 \int_{\mathbb{R}^N}|u_{\varepsilon}(x)|^p\,{\rm d}x&\leq\int_{B_{2R}(x_0/\varepsilon)}|U_{\varepsilon}(x)|^p\,{\rm d}x\\
&\leq C\varepsilon^{N-\frac{N-2s}{2}p}\int_{0}^{2R/\varepsilon}\frac{r^{N-1}}{(1+r^2)^{\frac{N-2s}{2}p}}dr\\
&=C\varepsilon^{N-\frac{N-2s}{2}p}\Big(\int_0^{\delta_1}\frac{r^{N-1}}{(1+r^2)^{\frac{N-2s}{2}p}}dr+\int_{\delta_1}^{2R/\varepsilon}\frac{1}{r^{N(p-1)-2s p+1}}dr\Big)\\
&=\left\{
          \begin{array}{ll}
            O(\varepsilon^{N-\frac{N-2s}{2}p}),&\hbox{$p>\frac{N}{N-2s}$,} \\
O(\varepsilon^{N-\frac{N-2s}{2}p}|\log\varepsilon|), & \hbox{$p=\frac{N}{N-2s}$,} \\
O(\varepsilon^{\frac{N-2s}{2}p}), & \hbox{$1<p<\frac{N}{N-2s}$,}
          \end{array}
        \right.
\end{align*}
where $0<\delta_1<2R/\varepsilon$. Hence, \eqref{A-4} holds true.

$(ii)$

\begin{align*}
\int_{\mathbb{R}^N}|u_{\varepsilon}(x)|^{2_s^{\ast}}\,{\rm d}x&=\int_{\mathbb{R}^N}|U_{\varepsilon}(x)|^{2_s^{\ast}}\,{\rm d}x+\int_{\mathbb{R}^N}\Big(|\eta(x-x_0/\varepsilon)|^{2_s^{\ast}}-1\Big)|U_{\varepsilon}(x)|^{2_s^{\ast}}\,{\rm d}x\\
&=\|\bar{u}\|_{2_s^{\ast}}^{2_s^{\ast}}+\int_{\mathbb{R}^N\backslash B_R(x_0/\varepsilon)}\Big(|\eta(x-x_0/\varepsilon)|^{2_s^{\ast}}-1\Big)|U_{\varepsilon}(x)|^{2_s^{\ast}}\,{\rm d}x\\
&=\|\bar{u}\|_{2_s^{\ast}}^{2_s^{\ast}}+\int_{\mathbb{R}^N\backslash B_{R/\varepsilon}(0)}\Big(|\eta(x)|^{2_s^{\ast}}-1\Big)\frac{1}{(1+|x|^2)^{N}}\,{\rm d}x\\
&\approx\|\bar{u}\|_{2_s^{\ast}}^{2_s^{\ast}}+\int_{\mathbb{R}^N\backslash B_{R/\varepsilon}(0)}\frac{1}{|x|^{2N}}\,{\rm d}x\\
&=\|\bar{u}\|_{2_s^{\ast}}^{2_s^{\ast}}+O(\varepsilon^N)
\end{align*}
as $\varepsilon\rightarrow0$.
\end{proof}

{\bf Acknowledgements.}
This work was done when the first author visited Department of Mathematics, Texas A$\&$M University-Kingsville under the support of
China Scholarship Council (201508140053), and he thanks Department of Mathematics, Texas A$\&$M University-Kingsville for their kind hospitality.
The first author is also supported by NSFC grant 11501403.

\end{document}